\documentclass[12pt]{article}
\usepackage[utf8]{inputenc}
\usepackage[english]{babel}
\usepackage{amsmath,amssymb,mathrsfs,color}
\usepackage{amsthm}
\usepackage{mathtools,bm}
\usepackage{enumerate}
\usepackage{hyperref}
\usepackage{fullpage}
\usepackage{graphicx}
\usepackage{caption}
\usepackage{subcaption}
\usepackage{stmaryrd}
\SetSymbolFont{stmry}{bold}{U}{stmry}{m}{n}

\numberwithin{equation}{section}
\theoremstyle{plain}
\newtheorem{theorem}{Theorem}[section]

\newtheorem{corollary}[theorem]{Corollary}

\newtheorem{lemma}[theorem]{Lemma}
\newtheorem{proposition}[theorem]{Proposition}

\theoremstyle{definition}
\newtheorem{definition}[theorem]{Definition}

\newtheorem{example}[theorem]{Example}

\newtheorem{remark}[theorem]{Remark}

\newcommand{\abs}[1]{\left\vert #1 \right\vert}
\newcommand{\compAbs}[1]{\left\vert#1\right\vert_{\text{c}}}
\newcommand{\bracket}[1]{\left\llbracket #1 \right\rrbracket}
\newcommand{\scalar}[1]{\left\langle #1 \right\rangle}
\newcommand{\norm}[1]{\left\Vert #1\right\Vert}

\newcommand{\R}{\ensuremath{{\mathbb R}}}

\newcommand{\Id}{\mbox{Id}}
\newcommand{\myspan}{\text{span}}

\def\gs{\sigma}

\def\cF{\mathcal{F}}
\def\cP{\mathcal{P}}\def\cQ{\mathcal{Q}}
\def\cR{\mathcal{R}}
\def\cI{\mathcal{I}}
\def\cC{\mathcal{C}}

\def\cS{\mathcal{S}}\def\cD{\mathcal{D}}

\def\bN{\mathbb{N}}
\def\bR{\mathbb{R}}

\def\bZ{\mathbb{Z}}
\def\bE{\mathbb{E}}
\def\bP{\mathbb{P}}

\def\fc{\mathbf c}

\def\fk{\mathbf k}\def\fl{\mathbf l}
\def\fm{\mathbf m}\def\fn{\mathbf n}\def\fN{\mathbf N}

\def\fr{\mathbf r}
\def\fs{\mathbf s}\def\ft{\mathbf t}
\def\fu{\mathbf u}\def\fv{\mathbf v}
\def\fx{\mathbf x}
\def\fy{\mathbf y}\def\fz{\mathbf z}

\def\falpha{{\boldsymbol\alpha}}
\def\fbeta{{\boldsymbol\beta}}
\def\fgamma{{\boldsymbol\gamma}}
\def\fdelta{{\boldsymbol\delta}}
\def\fepsilon{{\boldsymbol\epsilon}}
\def\flambda{{\boldsymbol\lambda}}
\def\fmu{{\boldsymbol\mu}}

\def\fxi{{\boldsymbol\xi}}
\def\foo{{\boldsymbol 0}}
\def\fone{{\boldsymbol 1}}

\def\fT{\mathbf T}

\def\supp{\mbox{supp}}

\title{A stochastic reconstruction theorem on rectangular increments with an application to a mixed hyperbolic SPDE}
\author{
   Carlo Bellingeri\thanks{ \texttt{carlo.bellingeri@univ-lorraine.fr}}\\
  IECL, Université de Lorraine
  \and
 Hannes Kern \thanks{\texttt{kern@math.tu-berlin.de}}\\
 TU Berlin
}
\date{}

\begin{document}

\maketitle
\begin{abstract}
We extend the stochastic reconstruction theorem to a setting in which the underlying family of distributions satisfies natural conditions involving rectangular increments. This allows us to establish the well-posedness of a new class of mixed stochastic partial differential equations of hyperbolic type, which combine standard Walsh stochastic integration with Young products.
	\end{abstract}
 \medskip
 

\medskip

\tableofcontents

\section{Introduction}
\subsection{Stochastic reconstruction and hyperbolic SPDEs}
Since its introduction in \cite{Hairer2014}, the reconstruction theorem has not only proved to be a foundational result in the theory of regularity structures but also a powerful result in the theory of distributions, as explained in \cite{zambotti2020}.  Loosely speaking,  this theorem  identifies optimal conditions on a  family of distributions $F=\{ F_\fx\colon \fx\in \mathbb{R}^d\}$, called germ, under which there exists a unique distribution  $\mathcal{R}(F)$ which provides a   ``good approximation" of  $F_\fx$ around each point $\fx\in \mathbb{R}^d$. The simplest set of conditions to formulate these properties involves the property of coherence\footnote{The usual definition of coherence involves $\gamma\in \mathbb R$ but in the context of this article, we focus on the case of unique reconstruction, which requires $\gamma>0$. }:  for some $\gamma>0$ and $\alpha\leq 0$ one has a growth condition of the form
\begin{equation}\label{ineq:coherenceIntro}
    \abs{(F_\fx- F_\fy)(\psi_\fx^\lambda)} \lesssim \lambda^{\alpha}(|\fx- \fy|+ \lambda)^{\gamma- \alpha}
\end{equation}
where $\lambda\in (0,1]$, $\fx, \fy\in \mathbb R^d$, 
 $\psi_\fx^\lambda(\cdot)=\lambda^{-d}\psi(\lambda^{-1}(\cdot- \fx))$ is a rescaled and recentered version of a test function $\psi$ and $\lesssim$ denotes inequality up to a constant. Further analytic conditions were studied in \cite{Broux2023,ZorinKranich2022}.

In this paper, we are interested in a generalization of this result where the underlying family $F$ is composed of random distributions depending on a probability space  $(\Omega, \mathcal{F}, \mathbb P)$ which are properly adapted to a multiparameter filtration $\mathcal{F}=\{\mathcal{F}_{\fx} \colon \fx\in \mathbb R^d\}$ (see Definition \ref{defn_adapted}). This setting is not novel in the literature. A previous result \cite{kern2021} showed we can modify the conditions \eqref{ineq:coherenceIntro}  into
\begin{align}
\begin{split}\label{ineq:stochCoherenceIntro}
    \norm{(F_\fx- F_\fy)(\psi_\fx^\lambda)}_m &\lesssim \lambda^{\alpha}(|\fx- \fy|+ \lambda)^{\gamma- \alpha}\,, \\ 
    \norm{\bE^i_{\fx}(F_\fx- F_\fy)(\psi_\fx^\lambda)}_m &\lesssim \lambda^\alpha(\abs{\fx-\fy}+\lambda)^{\gamma-\alpha+\delta}\quad \text{for any}\;  i\leq d\,,
\end{split}
\end{align}
where $\delta>0$,  $ \Vert X \Vert_{m}= \mathbb E (|X|^m)^{1/m}$ and  $\bE^i_{\fx}(X)=\bE[X|\cF^i_\fx]$ is the conditional expectation with respect to  $\cF^i_\fx$, the  filtration generated by the $i$-th component of $\fx$. Remarkably, the uniqueness result holds even in the case $0>\gamma>-d/2$ provided that $\gamma+ \delta>0$. Heuristically, we leverage the gain of regularity obtained by conditioning the germ by allowing  a lower index of general regularity $\gamma$. This type of argument is not new in the rough analysis literature but  it follows the same idea of the stochastic sewing lemma \cite{le2020}, from which  \cite{kern2021} was inspired. Moreover, stochastic sewing arises as a special case of the stochastic reconstruction when $d=1$, see \cite[Section 5]{kern2021}. From a stochastic analysis perspective, stochastic reconstruction provides a  new analytic framework to study Walsh integration \cite{walsh1986} with respect to a white noise $\xi$ over $\mathbb R^d_+:=[0,+ \infty)^d$ as a Young product, see \cite{zambotti2020}. Recalling the existence of a measurable modification of $\xi$ in the space of distribution, for any adapted process $Y$ satisfying some weak regularity assumption one can check that the germ $(Y\cdot \xi)_{\fx}(\psi):=Y_{\fx}\xi(\psi) $ satisfies \eqref{ineq:stochCoherenceIntro} and one obtains 
\begin{equation}\label{eq_sto_integral}
\int_{[0, +\infty)^d}\psi(\fx)Y_{\fx}\xi(d\fx)=\mathcal{R}(Y\cdot\xi)(\psi)\,,\end{equation}
where the left-hand side is a standard Walsh integral, see \cite[Theorem 53]{kern2021}.

However, looking at the several applications  coming from stochastic sewing in the recent literature, very few of them were extended in higher dimensions via stochastic reconstruction.  One key example to understand this discrepancy is the impossibility of multiplying a multiparameter martingale with a deterministic distribution. Consider e.g. a $d$-dimensional Brownian sheet $B$  and a deterministic distribution $\zeta$ and take the germ $(B\cdot\zeta)_{\fx}(\psi) := B_\fx\zeta(\psi)$. In that case, we see that  the quantity 
\begin{equation}\label{eq:key_point}
    \bE^i_{\fx} [((B\cdot\zeta)_{\fy}- (B\cdot\zeta)_{\fx}) (\psi_\fx^\lambda) ]
\end{equation}
has a better regularity (it just vanishes) only when $d=1$. In higher dimensions we do not have any improvement, locking the possibility of using the stochastic reconstruction theorem. 

The first major novelty of this paper is the introduction of new conditions on random germs so that it is still possible to reconstruct stochastic integrals, and they also allow a germ like $B\cdot\zeta$ to be reconstructed. Indeed in the simple case $d=2$ we can consider the rectangular increment of the previous germ $B\cdot\zeta$
\[ \square^{\{1,2\}}_{\fx,\fy}(B\cdot\zeta):= (B\cdot\zeta)_{(y_1,y_2)}-(B\cdot\zeta)_{(x_1,y_2)} - (B\cdot\zeta)_{(y_1,x_2)}+  (B\cdot\zeta)_{(x_1, x_2)}\,.\]
A simple computation  shows that for $i \in \{1,2\}$ 
  \[\bE^i_{\fx} (\square^{\{1,2\}}_{\fx,\fy}(B\cdot\zeta)(\psi_{(x_1,x_2)}^\lambda))=0\,.\]
And we recover the regularising effect of conditioning.  By simply noticing that this cancellation also holds even when we consider the germ $(Y\cdot \xi)$  a version of stochastic reconstruction on rectangular increment could be a more powerful tool to extend the properties of stochastic sewing in several dimensions. This idea does fit also with the historical multiparameter extension of stochastic calculus, see \cite{Imkeller85, Khoshnevisan02, norris95} where the combinatorics of rectangular increments play an important role. 

To state a proper result we will need to apply conditional expectation in different directions using the notation $\bE^\eta_\fx := \prod_{i\in\eta} \bE^i_\fx$ for any subset $\eta \subset\{1\,,\ldots \,, d\}$. Note that this is only well-defined if the filtrations $\mathcal{F}^i_\fx$ fulfill a commuting property, which is standard for multiparameter martingale theory, see for example \cite[Section 3.4]{Khoshnevisan02}. Combining the conditioning operators and the rectangular increments of the form 
$\abs{\square^\theta_{\fx,\fy} F(\psi)}$ for each $\theta\subset\{1,\dots,d\}$, see Section \ref{sec:preliminaries}, we can formulate in a simplified form the first main result of the paper. For a rigorous statement see Theorem \ref{theo:reconstruction}.
\begin{theorem}\label{thm:rec_form}
Let $F$ be a stochastic germ satisfying the rectangular coherence property:  for some $m\ge 2$, $\alpha<0$, $\delta>0$ and $\gamma>-\frac 12$ such that $\gamma+\delta > 0$ one has the growth condition
  \[ 
  \norm{\bE^\eta_\fx \square^\theta_{\fx,\fy} F(\psi_\fy^{\flambda})}_m \lesssim \prod_{i=1}^d\lambda_i^{\alpha}\prod_{i\in \theta}(|x_i-y_i|+\lambda_i)^{\gamma-\alpha}\prod_{j\in \eta}(|x_j-y_j|+\lambda_i)^{\delta}\,,
  \]
 where $\flambda=(\lambda_1\,, \ldots\,, \lambda_d)\in (0,1]^d$, $\fx= (x_1\,, \ldots ,x_d)$ $\fy= (y_1\,, \ldots ,y_d)$ are in compact set, and  $\psi_\fy^{\flambda} $ is a rescaled test function like in  \eqref{eq:rescaled_psi}. Then there exist unique family of distributions $\mathcal{R}^\theta_\fx (F)$ for each $\theta\subset\{1,\dots,d\}$ , such that $\mathcal{R} _\fx^\theta(F)$ does not depend on $x_i$ for $i\in\theta$ and for any non empty subset $\eta\subset\theta\neq\emptyset$ one has the growth condition 
\begin{equation}\label{eq:reconstruction,estimate}
\norm{\bE^\eta_\fx\sum_{\hat\theta \subset\theta}(-1)^{\sharp\hat\theta} \mathcal{R} ^{\hat\theta}_\fx(F)(\psi^\flambda_\fz)}_m \lesssim \prod_{i\in \theta}\lambda_i^{\gamma}\prod_{j\in\eta} \lambda_j^\delta\prod_{k\in \theta^c}\lambda_k^{\alpha}\,,
 \end{equation}
 for any $\fx,\fz$ such that $x_i=z_i$ for $i\in\theta$ and $\flambda\in(0,1]^d$. We call $\mathcal{R}^{\{1, \ldots , d\}}_{\fx}(F)$ the multiparameter  reconstruction of the stochastic germ $F$ and we denote it by $\mathcal{R} (F)$.
\end{theorem}
\begin{remark}
    Unlike the statements in \cite{zambotti2020,kern2021}, in this statement we allow one to freeze certain variables $x_i, i\notin\theta\subset\{1,\dots,d\}$ and only apply the reconstruction techniques for directions $i\in\theta$. As a result, we do not get a single distribution $\cR(F)$, but a family of partial reconstructions $\cR^\theta_\fx(F)$. Let us illustrate these on example \eqref{eq_sto_integral} in $d=2$. By the same arguments as \cite[Theorem 53]{kern2021},
    \begin{align*}
        \cR^\emptyset_\fx(Y\cdot\xi)(\psi) &= \int_{[0,+\infty)^2} \psi(\fz)Y_\fx\xi(d\fz) \\
        \cR^{\{1\}}_\fx(Y\cdot\xi)(\psi) &= \int_{[0,+\infty)^2} \psi(\fz)Y_{(z_1,x_2)}\xi(d\fz)\\
        \cR^{\{2\}}_\fx(Y\cdot\xi)(\psi) &= \int_{[0,+\infty)^2} \psi(\fz)Y_{(x_1,z_2)}\xi(d\fz)\\
        \cR(Y\cdot\xi)(\psi) = \cR^{\{1,2\}}_\fx(Y\cdot\xi)(\psi) &= \int_{[0,+\infty)^2} \psi(\fz)Y_\fz\xi(d\fz)\,,
    \end{align*}
    are the Walsh integrals, where $\theta$ decides over which variables we integrate $Y$. For $i\notin\theta$, $Y$ gets evaluated at the ``frozen'' point $x_i$. This also links \eqref{eq:reconstruction,estimate} to rectangular increments, as in the above example
    \[
        \sum_{\hat\theta \subset\theta}(-1)^{\sharp{\hat\theta}} \mathcal{R} ^{\hat\theta}_\fx(Y\cdot\xi)(\psi) = \int_{[0,+\infty)^2} \psi(\fz)\square^\theta_{\fx,\fz} Y \xi(d\fz)
    \]
    is the integral over the rectangular increments of $Y$ for all $\theta\subset\{1,2\}$.
\end{remark}
\begin{remark}
It should be noted that the proof of the stochastic reconstruction theorem is closer to the original proof of Martin Hairer than the later improvements \cite{zambotti2020,Broux2023,ZorinKranich2022}.  This is because Martin Hairer's proof is based on wavelet approximations, which approximate the target distribution $f$ with the  sum
\begin{equation}\label{eq:approx_reconstruction}
\mathcal{R}(F)(\psi) \approx \sum_{\fx\in\Delta_\fn} F_\fx(\phi^\fn_\fy)\scalar{\phi^\fn_\fy,\psi}\,,
\end{equation}
where $\phi^\fn_\fy$ is a Daubechies-wavelet and $\Delta_\fn$ is an increasingly fine mesh over $\bR^d$, see Section \ref{sec:preliminaries}. This particular structure allows to treat the convergence of these sums in the same way as  \cite{le2020} by using a proper estimate on sums via a generalized BDG inequality, see Lemma \ref{lem:multiparameterBDG}. Moreover, wavelets also play an important role when we want to integrate the reconstruction to obtain a proper random field (see Proposition \ref{prop:h}).
\end{remark}

From the possibility of  multiplying a Brownian sheet with a deterministic distribution via stochastic reconstruction in Theorem \ref{thm:rec_form} we will be able to provide a full solution theory  for a new type of mixed non-linear SPDEs
\begin{equation}\label{eq:spde}
\begin{split}
    \frac{\partial^d }{\partial x_1\ldots\partial x_d}u(\fx) &= \gs(u(\fx))\xi+ f(\fx)u(\fx)\frac{\partial^d }{\partial x_1\ldots\partial x_d}Z(\fx)\quad  \fx \in \mathbb R_+^d\,\\
u(\fx)&=v(\fx) \qquad  \quad \fx\in \partial \mathbb R_+^d=\bigcup_{i=1}^d\{\fx\in \mathbb R_+^d \colon x_i=0 \}
\end{split}
\end{equation}
where $\xi$ is a white noise over $ \mathbb R_+^d$, $\sigma\colon \mathbb R\to \mathbb R$ is a generic non-linearity, $f\colon \Omega\times \mathbb R_+^d\to \mathbb R$ is a given random field, and $Z\colon \mathbb R_+^d\to \mathbb R$ is a deterministic H\"older field in $C^{\beta}$, see Section \ref{sec:hoelderSpacesRandom}.

Equation \eqref{eq:spde} might seem unusual but when $d=2$  arises under different forms in new and old problems in PDE and SPDE theory: firstly, when $\sigma=0$ and $f\equiv 1$, the equation \eqref{eq:spde}  becomes the signature kernel PDE, see \cite{kernel_pde}, which   plays an important role in the theory of signature Kernels, a possible theory to implement machine learning techniques with sequential data. In particular, equation  \eqref{eq:spde} when $d=2$ represents a possible  stochastic perturbation  of the  signature kernel PDE. Concerning the deterministic literature, equation \eqref{eq:spde} represents a stochastic perturbation of a special differential  
equation driven by a H\"older field on hyper-cubes, see \cite{harang21}, hence we decided to formulate the equation in a generic dimension $d$. Concerning the opposite case when $f=0$ the underlying SPDE is classical and it has been intensively studied, see e.g. 
\cite{cairoli72,nualart89,nualart89A,millet94} in dimension $d=2$. 
We also remark that the differential operator $\frac{\partial^2 }{\partial x_1\partial x_2}u$ is naturally linked with the wave operator. Therefore  equation  \eqref{eq:spde} when $d=2$ could  help us to understand  path-wise well-posedness of stochastic wave equation in future investigations. From a mathematical perspective, equations such as \eqref{eq:spde} do not seem to have been addressed in the literature yet, and from a technical point of view, both classical deterministic techniques and techniques based on standard stochastic integration fail to establish well-posedness of this problem.

To give a rigorous meaning to  equation \eqref{eq:spde}, we recall the identification of stochastic integration with the reconstruction in \eqref{eq_sto_integral} and we will describe the products  in the right-hand side as reconstruction of germs of the simple form recalled before. Then we can write the equation into its integral form after applying an integration map
\[f\to \int_{\foo}^{\fx}f (d\fz)\,, \quad \int_{\foo}^{\fx}= \int_{0}^{x_1} \ldots \int_{0}^{x_d} \]
taking the primitive of a distribution in $d$ different variables, and we get to the equation 
\begin{equation}\label{equation_mild}
   u(\fx) =  \mathcal{I}(v)(\fx) + \int_{\foo}^{\fx} \sigma(u(\fy)) \xi(d\fy) + \int_\foo^\fx f(\fy)u(\fy)\frac{\partial^d }{\partial x_1\ldots \partial x_d}Z(d\fy) \,,
\end{equation}
for some explicit boundary term $\mathcal I(v)(\fx)$, see equation \eqref{equation_boundary_function}. Our second main result will be the identification of a class of random field $u$, the space $C^{\alpha, \delta}L_m$, see Definition \ref{def:C_alpha-stoc} where the equation \eqref{equation_mild} is globally well-posed in the so-called Young regime, i.e. when the trajectories of $Z$ are slightly more regular than the trajectories of Brownian sheet. See Theorem \ref{spde_thm} for a rigorous statement.

\begin{theorem}
Let $m$ be in $[2,+\infty)$ and $ Z\in  C^{\beta}$ with $\beta >1/2$. Let $\sigma$ be a Lipschitz function and consider the parameters $\delta$ and $\alpha$ such that $\delta= \beta- 1/2$ and $\delta\leq \alpha< 1/2$,  $\alpha+\beta+ \delta>1$. Then for 
every $v\colon \partial \mathbb R_+^d\to \mathbb R$ such that  $\mathcal{I}(v)\in C^{\alpha, \delta}L_m$ and any $f\in  C^{\alpha, \delta}L_{\infty}$ there exists a unique solution $u\in  C^{\alpha, \delta}L_m$ to \eqref{equation_mild} starting
from $v$ over any finite time interval.
\end{theorem}
\begin{remark}
Recently, a new stochastic extension of the multiparameter sewing lemma was proposed in \cite{multiparameterStochSewing} to study the effect of regularisation by noise phenomena on a similar SPDE. The techniques developed in that article are fundamental to the proof of the reconstruction theorem, which leads to Theorem \ref{thm:rec_form} as a natural extension of stochastic multiparameter sewing in a distributional setting. Concerning the equation \eqref{eq:spde}, using the $C^{\alpha, \delta}L_{m}$ spaces introduced in this paper, one can also solve it using the multiparameter stochastic sewing lemma, provided that the result of sewing lies in the appropriate spaces (see Remark \ref{rk_sew_fix_point}), which we believe requires the same complexity as the results in Section \ref{spde_section}. 
In contrast to the stochastic multiparameter sewing lemma, stochastic reconstruction has the advantage of not dealing with multiparameter delta operators $\delta^{\eta}$ (see Section \ref{sec:sto_sewing}), which become cumbersome in high dimensions. Moreover, it allows for a notion of general stochastic integration when there is no clear definition of a $2d$-parameter process to integrate (see Proposition \ref{prop:YoungProduct}). Finally, a formulation of this SPDE using a reconstruction theorem could be more helpful in extending techniques inspired by regularity structures to more general hyperbolic SPDEs. Further details on the relationship between our version of stochastic reconstruction and the stochastic multiparameter sewing lemma are studied in Section \ref{sec:sto_sewing}.
\end{remark}

\subsection{Outline of the paper}

We  outline the paper by summarising the content of its sections. In Section \ref{sec:preliminaries}, 
we start by introducing in full detail the main tools to prove the reconstruction theorem on rectangular increments. We
recall the main properties of rectangular increments of maps as explained in \cite{harang21} and extend them to germs. We then recall the main properties of wavelets  \cite{debauchies} that we will adapt to define the approximating sequence in \eqref{eq:approx_reconstruction}. We also recall some basic results on multiparameter filtrations and the main stochastic estimate in Lemma \eqref{ineq:lem3Condtition}.

In Section \ref{sec:reconstruction}, we then pass to properly define the main spaces of random distributions, random fields, and random germs to  
properly state Theorem \ref{thm:rec_form} and prove it. All definitions will also allow different H\"older parameter and regularisation regularities $\falpha=(\alpha_1\,, \ldots \alpha_d)$, $\fdelta=(\delta_1\,,\ldots \delta_d)$ to keep track of the regularisation effect on any components.

In Section 4, we finally introduce the main operations to define the SPDE \eqref{equation_mild} as a fixed point of the space $C^{\alpha, \delta}L_m$. We will consider in particular the composition (see Proposition \ref{prop:g(u)CanBeReconstructed}),  the pointwise product  (see Proposition \ref{prop:ProductCalphaSpace})  and the pointwise product between a random field and a distribution (see Propositions \ref{lem:prodDeterministicZeta} and \ref{prop:ItoRecon}) which allows us to define all the non-linearities at the level of random distributions. We then combine this operation with  a usual integration map, which extends a deterministic result obtained in \cite{Brault2019} in several dimensions. Theorem \ref{spde_thm} then checks that all these operations provide a proper fixed point over the spaces introduced in the previous section. 

In Section 5, we show that the multiparameter stochastic sewing lemma from \cite{multiparameterStochSewing} is a special case of our reconstruction theorem \ref{theo:reconstruction}. This is shown in Proposition \ref{prop:SewingIsReconstruction}. It is based on the same construction, which connects sewing and reconstruction (see e.g. \cite{Broux2022}), as well as stochastic sewing and stochastic reconstruction (see \cite{kern2021}), in dimension $d=1$.

\subsection{Open extensions arising from our work}\label{sec:OpenExtensions}

We conclude the introduction by outlining two natural extensions raised by the
results we established. In our current formulation, we can only solve \eqref{equation_mild} for non-linearities of the form $f(\fx)u(\fx)$. One natural extension of the problem would be to consider a generic non-linearity $g(u)$ in \eqref{eq:spde}. However, it seems that the structure of the spaces $C^{\alpha,\delta} L_m$ is not well suited to take generic compositions $g(u)$ for $u\in C^{\alpha,\delta} L_m$ into account. That is, even for smooth functions $g$, $g(u)$ might not belong to the same space as $u$. This stands in stark contrast to the deterministic case: In $d=2$, it is well-known that if $g\in C^2$ then  $g(u)\in C^\falpha$ for any $u\in C^\falpha$ and $\falpha=(\alpha_1, \alpha_2)\in (0,1)^2$ (see e.g. \cite[Lemma 3.1]{tindel07}). Moreover,  for any H\"older function $\sigma\in C^{\eta}$ for some $\eta\in(1,2)$, we still have that $\sigma(u)\in C^{\tilde\alpha}$ for a some $\tilde\alpha < \alpha$ (see \cite[Lemma 15]{bechtold2022}). Both of these results rely on the application of the classical Taylor formula to the rectangular increment
 \[
 \square^{\{1,2\}}_{\fx,\fy} g(u)=g(u(y_1,y_2))- g(u(x_1,y_2))- g(u(y_1,x_2)) + g(u(x_1,x_2))\,,
 \]
 which behaves like the product $\square^{\{1\}}_{\fx,\fy} u \square^{\{2\}}_{\fx,\fy} u$ under some knowledge on the second derivative of $g$. However, in the stochastic case, we need  to control $\Vert \square^{\{1,2\}}_{\fx,\fy} g(u)\Vert_m$, where the same argument provides a control of this quantity  in terms of  $\Vert\square^{\{1\}}_{\fx,\fy} u \square^{\{2\}}_{\fx,\fy} u \Vert_m$, which might explode since $\square^{\{1\}}_{\fx,\fy} u \square^{\{2\}}_{\fx,\fy} u$ only has finite moments up to  $m/2$.

In \cite{Friz_Le_Hocquet21}, the authors developed a complete theory of  rough stochastic differential equations of the form \begin{equation}\label{eq:RSDE} dY_t= b_t(Y_t)+\sigma_t(Y_t) dW_t +   f_t(Y_t) d\mathbf{X}_t \end{equation}with $W$ a Brownian motion, $\mathbf{X}= (X, \mathbb X)$  a H\"older rough path and the data $b,\sigma, f$ are taken significantly less regular than the purely deterministic theory of rough differential equation, representing the one dimensional extension of \eqref{eq:spde} with a generic non-linearity.  To achieve this, the authors introduced a new family of Banach spaces where the solution belongs using $L_{m,n}$ stochastic norms, see \cite[Section 2]{Friz_Le_Hocquet21}. Using our notations, this would be equivalent to controlling the quantities 
\begin{equation*}
    \norm{\square^{\{1,2\}}_{\fx,\fy} g(u)}_{m,n}^{\eta} := \norm{(\bE_\fx^\eta|\square^{\{1,2\}}_{\fx,\fy} g(u)|^m)^{\frac{1}{m}}}_n \, \quad \emptyset\neq \eta\subset \{1,2\}\,.
\end{equation*}
The presence of a conditional expectation inside the norm with an extra parameter $n$ were two fundamental tools to obtain a well-posed composition map in \cite[Lemma 3.13]{Friz_Le_Hocquet21}. Here, the lack of integrability is absorbed in the parameter $n$ (which becomes smaller). We expect that a similar result holds even in the rectangular setting  modulo a proper extension of the space. Unfortunately, this space is not suited to establish a proper reconstruction theorem in dimension $d>1$. While one might  generalize Theorem \ref{theo:reconstruction} to an $L_{m,n}$ germ $F$, the resulting estimate \eqref{eq:reconstruction,estimate} coming from rectangular increments does not give us a direct estimate on $\mathcal{R}(F)(\psi_{\fx}^{\lambda})$, but the best we can hope for is to have bounds on
\[
    \norm{\bE_\fx^\eta  \sum_{\hat\theta\subset\theta}(-1)^{\sharp{\hat\theta}} \mathcal{R}^{\hat\theta}_\fx(F)(\psi^{\lambda}_\fz)}_{m,n}^{\theta\setminus\eta}\,.
\]
In dimension one, this estimate together with the assumption that $\norm{F_\fx}_{n}$ is bounded uniformly  suffices to ensure that 
\[
\norm{\mathcal{R}(F)(\psi_\fx^\lambda)}_{m,n}^{\{1\}}<\infty\,.
\] 
However, already in dimension two the same assumption on $F_\fx$ only gives that
\begin{equation*}
    \norm{\mathcal{R}^{\{1\}}_\fx(F) (\psi_\fx^\lambda)}_{m,n}^{\{1\}}, \quad \norm{\mathcal{R}^{\{2\}}_\fx (F)(\psi_\fx^\lambda)}_{m,n}^{\{2\}} <\infty\,,
\end{equation*}
but makes no statement about
\[
\norm{\mathcal{R}^{\{2\}}_\fx(F) (\psi_\fx^\lambda)}_{m,n}^{\{1\}},\quad \norm{\mathcal{R}^{\{1\}}_\fx\,(F) (\psi_\fx^\lambda)}_{m,n}^{\{2\}}\,.
\]
Therefore we cannot conclude that $\norm{\mathcal{R} (\psi_\fx^\lambda)}_{m,n}^{\{i\}}$ is finite for $i=1$ or $i=2$. An extension to understand this problem could unlock the possibility of solving a more general class of SPDEs than \eqref{eq:spde}.

Another natural extension arising in this work would be to consider a deterministic perturbation $Z$ in  \eqref{eq:spde} with lower H\"older regularity $\beta\leq 1/2$ exactly as studied in \cite{Friz_Le_Hocquet21}. While it seems like a naive question, a possible solution to this question would naturally involve a rough path theory of differential equations driven by a H\"older field on hyper-cubes. A consistent solution theory for such an equation has been an open problem in the theory of rough paths since the introduction of rough sheets \cite{chouk2014roughsheets,tindel15} and it does seem out of reach with the actual technology. However, we think that the presence of this article may arouse new interest in the issue and become a further motivation to analyze this problem again.

\subsection*{Acknowledgements}
CB is supported by the ERC Starting Grant “LoRDeT.” He also acknowledges the support of TU Berlin, where he was employed when this project was initiated. HK is supported by the Research Unit FOR2402 at TU Berlin and the International Research Training Group 2544 of the Berlin Mathematical School. Both authors would like to thank the organizers of the conference “Structural Aspects of Signatures and Rough Paths” held at the Centre for Advanced Study in Oslo, where a preliminary version of these results was presented.

\subsection*{Notation}
 Throughout the paper, we will adopt the convention of writing multi-indices in $\mathbf{n}\in\bN^d$ and points in $\fx\in \bR^d$ with bold characters, with the usual notations $\fx=(x_1\,, \ldots\,, x_d)$ and $\mathbf{n}=(n_1\,, \ldots \,,n_d) $ and the convention $\bN= \{0,1,2, \ldots\}$. We also use the shorthand notation $[d]= \{1\,,\ldots\,,d\}$ for some integer $d\geq 1$, $\mathbf{0}$ for the zero vector and $\mathbf{1}$ if all the components are equal to $1$.  We write  $\fx\le \fy$ for $ \fx, \fy\in \bR^d $  if $x_i\le y_i$ for all $i\in [d]$. We extend $<,>$ and $\ge$ analogously. Generic  subsets $\theta \subset [d]$ will also play an important role at the level of multi-indices. We denote  by $\mathbf{1}_{\theta}$ the multi-index whose components are equal to $1$ according to the set $\theta$. By definition $\mathbf{0}=\mathbf{1}_{\emptyset}$,  $\mathbf{1}=\mathbf{1}_{[d]}$ and in these two special cases we will adopt the  notation of the left-hand side. More generally, for any $\mathbf{x}\in \mathbb{R}^d$ we will denote by $\mathbf{x}_{\theta}$ the vector whose components are equal to $\mathbf{x} $ according to $\theta$ and they are $0$ otherwise. Furthermore, we set $\sharp\theta$ to be the cardinality of $\theta\subset[d]$.

Most classical operations on real numbers extend trivially  to the components of a vector. For any non-empty indexset $\theta\subset [d]$ , we set in particular 
\[   (\fy)^{\fx}_\theta = \prod_{i\in \theta} y_i^{x_i} \,, 
\]
\[   \max(\fx, \fy)=\fx\vee\fy = (x_1\vee y_1,\dots, x_d\vee y_d)\,, \quad \min (\fx, \fy)=  \fx\wedge\fy= (x_1\wedge y_1,\dots, x_d\wedge y_d)\,,\]
\[
\fx * \fy = (x_1 y_1,\dots, x_d y_d)\,,
\]
 for $\fy,\fx\in\R^d$ whenever each component in the product is well-defined. Note that by this notation, $(\fy^\fx)_\theta = \fy^{\fx_\theta}$. When we have $\theta= [d]$ we will omit the set in the operation. This expression also extends to the case $\theta= \emptyset$ by imposing $(\fy^\fx)_\emptyset=1$. Similarly, we even allow mixed expressions among scalars and vectors like $\alpha^{\fy} =\alpha^{y_1}\cdots \alpha^{y_d}$ without any ambiguities. The absolute value of a vector in our setting is also a vector which is given by $\abs{\fx}_{\text{c}} = (\abs{x_1}\,, \ldots,  \abs{x_n})$. This notation extends to any subset $\theta \subset [d]$ the product
\begin{equation*}
    \abs{\fx-\fy}^{\falpha}_{\text{c},\theta} = \prod_{i\in \theta}\abs{x_i-y_i}^{\alpha_i}\,.
\end{equation*}
in a synthetic way.


Partial derivatives of  a differentiable function $\psi\colon \R^d\to \R$ are denoted by
\[\partial^{\mathbf{n}}\psi= (\partial^{n_1}_{x_1}\ldots\partial^{n_d}_{x_d})\psi\,.\]
We denote by $C^r_c$, $r\in \bN$ the space of $r$-times continuously differentiable functions $\psi\colon \R^d\to \R$ with compact support. We will refer to them as test functions. We equip $C^r_c$ with the norm
\begin{equation*}
    \norm{\psi}_{C^r_c} := \sum_{\abs{\fn}\le r} \norm{\partial^\fn\phi}_\infty\,,
\end{equation*}
where $\abs{\fn}= n_1+ \cdots + n_d$ and $\norm{\cdot}_\infty$ is the sup norm. 
For a compact set $K\subset\bR^d$, we will denote by $C^r(K)$  the subset of $C^r_c$  containing functions $\psi$ whose  support $\supp(\psi) $ is contained in $K$ and by   $B^r(K)$ the subset of functions of $C^r(K)$ such that $\norm{\psi}_{C^r}=1$. 


Throughout the text, we will use the standard notation $x \lesssim y$ to denote the existence of a constant $C > 0$, independent of the variables in $x$ and $y$, such that $x\leq C y$.

\section{Preliminaries}\label{sec:preliminaries}
We recall now the main analytic and probabilistic tools to formulate the stochastic reconstruction theorem in the context of rectangular increments.

\subsection{Rectangular increments of germs}
We begin by recalling the rectangular increments of a generic function $f\colon \R^d\to \R$ and of a generic germ of distributions. This part follows \cite{multiparameterStochSewing,harang21} with some slight differences in the notation. For more classical introductions to the topic, especially concerning multiparameter martingale theory, we recommend \cite{Khoshnevisan02,Cairoli1975}.

For any given $i\in  [d]$ and $\fx,\fy\in\R^d$ we set the projection $\pi^i_{\fx}\colon \R^d\to \R^d$ on the $i$-th coordinate of $\fx$ by
\begin{equation*}
    \pi^i_{\fx} \fy := (y_1,\dots, y_{i-1}, x_i, y_{i+1},\dots, y_d)\,.
\end{equation*}
The operators $\pi^i_{\fx}$ commute among themselves, allowing us to extend these projections on any subset $\theta\subset[d]$ and $ \fx\in \R^d$ as follows  
\begin{equation}\label{eq:projectionDef}
\pi^{\theta}_{\fx}:= \prod_{i\in\theta} \pi^i_{\fx}\,,
\end{equation}
where we adopt the shorthand notation $\pi^i_{\fx}\pi^j_{\fx}=\pi^i_{\fx}\circ \pi^j_{\fx} $ and the convention $\pi^{\emptyset}_{\fx}=\text{Id}$ the identity function on $\R^d$. We apply these projections to any function $f\colon \R^d\to \R$ by setting  
$\pi^{\theta}_{\fx}f(\fy):= f(\pi^{\theta}_{\fx}\fy)$.

For any given $\fx,\fy\in\R^d$ and subset $\theta\subset [d]$, the rectangular increment of a function $f\colon \R^d\to \R$ is defined by the expression
\begin{align}\label{rectangular_increment}
\square_{\fx,\fy}^\theta f :=&\prod_{i\in\theta}(\text{Id} - \pi^i_{\fx})f(\pi^{\theta^c}_{\fx}\fy) = \sum_{\theta'\subset\theta} (-1)^{\sharp(\theta\setminus\theta')} f(\pi^{\theta'}_{\fy}\fx)
= \prod_{i\in\theta}(\pi^i_\fy - \Id)f(\fx)\,, 
\end{align}
where we use the same convention of writing composition of operators as product and we set  $\theta^c= [d]\setminus \theta$. The second expression in \eqref{rectangular_increment} is a sum over all subset of $\theta'\subset \theta $. For $\theta= \emptyset$ we have $\square_{\fx,\fy}^\emptyset f= f(\fx)$. The equality  of these two expressions above follows from straight-forward combinatorial considerations, see  \cite[Definition 1]{harang21}.

For instance, when $d=2$ we have the expressions
\[\square_{(x_1,x_2),(y_1,y_2)}^{\{1\}} f= f(y_1, x_2)- f(x_1,x_2)\,,\quad  \square_{(x_1,x_2),(y_1,y_2)}^{\{2\}} f= f(x_1, y_2)- f(x_1,x_2)\,, \]
\[\square_{(x_1,x_2),(y_1,y_2)}^{\{1,2\}}f= f(y_1, y_2)- f(y_1, x_2)- f(x_1, y_2)+  f(x_1,x_2)\,. \]
Using this expression, we can synthetically show some lengthy identities.

\begin{lemma}\label{lem:splitSquareIncrement}
Let $\fx,\fy\in \R^d$ and $f,g\colon \R^d\to \R$. We then have the following properties.
\begin{enumerate}
    \item For any couple of disjoint sets $\theta_1,\theta_2\subset[d]$,
\begin{equation}\label{eq:compositionOfSquares}
    \square_{\fx,\fy}^{\theta_1}(\square_{\cdot,\fy}^{\theta_2} f )=  \square^{\theta_1\cup\theta_2}_{\fx,\fy} f\,.
\end{equation}
\item For every $\theta\subset[d]$, we have
\begin{equation}\label{eq:productOfSquares}
    \square_{\fx,\fy}^\theta(f\cdot g) = \sum_{\theta = \theta_1\cup\theta_2} \square^{\theta_1}_{\fx,\fy} f\cdot \square_{\fx,\fy}^{\theta_2} g\,,
\end{equation}
    where the sum runs over all $\theta_1,\theta_2\subset[d]$ (not necessarily disjoint!), such that $\theta = \theta_1\cup\theta_2$.
\item  For any couple of subsets $\eta\subset\theta\subset[d]$
\begin{equation}\label{eq:splitUpSquare}
    \square_{\fx,\fy}^\theta f = \sum_{\theta'\subset\theta\setminus\eta}(-1)^{\#(\theta\setminus(\eta\cup\theta'))}\square^\eta_{\pi^{\theta'}_{\fy}\fx,\fy} f\,.
\end{equation}
\end{enumerate}
\end{lemma}
\begin{proof}
Composing the two definitions of rectangular increments we have 
\[ \square_{\fx,\fy}^{\theta_1}(\square_{\cdot,\fy}^{\theta_2} f ) =\prod_{i\in\theta_1}(\text{Id} - \pi^i_{\fx}) \prod_{i\in\theta_2}(\text{Id} - \pi^i_{\pi^{\theta_1^c}_{\fx}\fy})f(\pi^{\theta_2^c}_{\pi^{\theta_1^c}_{\fx}\fy}\fy)\,.\]
From the identities
\[\pi^{\theta_2^c}_{\pi^{\theta_1^c}_{\fx}\fy}\fy=\pi^{\theta_2^c\cap \theta_1^c}_{\fx}\fy\,, \quad \pi^i_{\pi^{\theta_1^c}_{\fx}\fy}= \pi^i_{\fx}\, \quad \text{for any $i\in \theta_2$}\] 
we deduce the desired equality. For the second part, note that for each $i\in[d]$, direct computation shows
\[
\square^{\{i\}}_{\fx,\fy}(f\cdot g) = (\square^{\{i\}}_{\fx,\fy} f) g(\fx) + f(\fx)(\square^{\{i\}}_{\fx,\fy} g) + (\square^{\{i\}}_{\fx,\fy} f) (\square^{\{i\}}_{\fx,\fy} g)\,.
\]
Iterating this with the use of \eqref{eq:compositionOfSquares} leads to \eqref{eq:productOfSquares}.

The third part follows automatically from \eqref{eq:compositionOfSquares} by choosing $\theta_2 = \eta$, $\theta_1 = \theta\setminus\eta$ and writing out the sum definition of $\square_{\fx,\fy}^{\theta\setminus\eta}$.
\end{proof}

%

The definition of rectangular increments \eqref{rectangular_increment} ensures that $f(\fx)$ is always contained in $\square^\theta_{\fx,\fy} f$, but $f(\fy)$ might not be. In the context of the coherence property (see Definition \ref{def:extendedCoherence}), it is more natural to work with rectangular increments containing $f(\fy)$. The following lemma allows us to do so.

\begin{lemma}\label{lem:shiftSquareIncrements}
    Fix $\fx,\fy\in\R^d$ and a function $f:\R^d\to\R$. For any couple of disjoint sets  $\theta,\eta\subset[d]$  one has
    \begin{equation}
        \square_{\fx,\fy}^{\theta} f = \sum_{\eta = \eta_1\sqcup\eta_2} (-1)^{\sharp{\eta_2}} \square^{\theta\cup\eta_2}_{\pi^{\eta_1}_\fy \fx, \fy}f\,,
    \end{equation}
     where we sum over all possible ways to represent $\eta$ as the disjoint union of $\eta_1$ and $\eta_2$.
\end{lemma}

\begin{proof} We think of $\square^\theta_{\fx,\fy} f$ as a function of $\fx$ for fixed $\fy\in\bR^d$. Then direct computation shows
    \begin{align*}
        \square_{\fx,\fy}^\theta f &= \prod_{j\in\eta}(\pi_\fy^j-(\pi_\fy^j- \text{Id}))\square_{\fx,\fy}^\theta f\\
        &= \sum_{\eta = \eta_1\sqcup \eta_2}(-1)^{\sharp{\eta_2}}\pi_\fy^{\eta_1}\prod_{i\in\eta_2}(\pi^i_\fy-\text{Id})\square_{\fx,\fy}^{\theta}f\\
        &= \sum_{\eta = \eta_1\sqcup\eta_2} (-1)^{\sharp{\eta_2}} \square^{\theta\cup\eta_2}_{\pi^{\eta_1}_\fy \fx, \fy}f\,.
    \end{align*}
\end{proof}

\begin{example}
    If we work in dimension $d=2$, Lemma \ref{lem:shiftSquareIncrements} simplifies to the following three identities applied to the sets $\theta=\emptyset, \{1\}$ and $\eta= \{1,2\}$, $\{2\}$.
 \begin{align*}
 f(\fx) &=\square_{\fx,\fy}^\emptyset f=  \pi^1_{\fy} f(\fx) - \square^{\{1\}}_{\fx,\fy} f \,, \\ f(\fx) &=\square_{\fx,\fy}^\emptyset f= f(\fy) -\square^{\{1\}}_{\pi^2_\fy\fx,\fy} f - \square^{\{2\}}_{\pi^1_\fy\fx,\fy}f + \square^{\{1,2\}}_{\fx,\fy}f\,,\\
        \square^{\{1\}}_{\fx,\fy} f &= \square^{\{1\}}_{\pi_\fy^2 \fx,\fy} f - \square^{\{1,2\}}_{\fx,\fy} f\,.
\end{align*}
    It especially follows, that all the rectangular increments on the right-hand side contain $f(\fy)$, as long as $\theta\cup\eta = [d]$. This observation extends to the general case $d\ge 2$.
\end{example}

The notion of rectangular increment does not only apply to real-valued functions $f\colon \R^d\to \R$ but it can be easily extended to $f$ taking values in a generic vector space. In our case, the desired vector space will be a proper set of distributions on $\R^d$ and the resulting functions will take the name of germs. Let us first introduce our space of test functions:


\begin{definition}
    A linear map $T:C_c^r \to \bR$ is called a local distribution of order $r\in \bN$ if for all compact sets $ K\subset \R^d$ there exists a constant $C(K) > 0$ such that
    \[\abs{T(\psi)}\le C(K) \norm{\psi}_{C^r_c}\,,\]
    for all $\psi\in C^r_c(K)$. We denote by $\mathcal{D}_r'$ the set of local distributions of order $r\in \bN$. For some index $r$ we call any map $F\colon \R^d\to \mathcal{D}_r'$ such that for all test functions $\psi$ the function $\fx\to F_{\fx}(\psi)$ is a measurable germ.
    



\end{definition}

\begin{remark}
The original definition of germ, see \cite[Definition 4.1]{zambotti2020} is given for general distributions without mentioning the order. This different choice will be motivated by a slightly different analytic setting we will use to prove the reconstruction theorem. As illustrated in the definition, we will consistently write in subscript the evaluation of a germ on a point $\fx$. We can equivalently represent a germ $F$ as a family of distributions $F=(F_{\fx})_{\fx\in\R^d}$.
\end{remark}

 For any given germ $F=(F_{\fx})_{\fx\in\R^d}$, we can transfer \eqref{rectangular_increment} by fixing a test function $\psi$ and applying the definition to the function $f(\fx)=F_{\fx}(\psi)$.

 Let us introduce the notion of scaling in our setting. Given a vector $\fx\in \R^d$ we want to rescale a test function $\psi\colon \R^d\to \R$  with respect to a generic vector $\flambda\in (0,1]^d$ with the definition
 \begin{equation}\label{eq:rescaled_psi}
 \psi_\fx^\flambda(\fy):= \flambda^{-\fone}\psi((y_1-x_1)\lambda_1^{- 1}\,, \ldots \,, (y_d-x_d)\lambda_d^{- 1})\,.
 \end{equation}
 For $\fx= \mathbf{0}$ we write $\psi^\flambda(\fy)= \psi_{\mathbf{0}}^\flambda(\fy) $. For any given $T\in \mathcal{D}_r'$ the map $\flambda\to T(\psi_\fx^\flambda(\fy)) $
describes  the behavior of $T$ close to $\fx$ at different scales.
%

%

%
%

\subsection{Wavelet bases}\label{sec:wavelet}

To prove the reconstruction theorem in the rectangular increment setting we will have to use all the regularising effects of wavelet theory, going back to \cite{debauchies}. First, we will recall some general facts about wavelet basis, mainly following \cite{Hairer2014,kern2021,meyer}.


Multidimensional wavelets can be constructed from one-dimensional wavelets, which are fundamental tools to describe the Hilbert space $L^2(\R^d)$ with its standard scalar product

\[
\langle f,g\rangle:= \int_{\R^d}f(x)g(x)dx\,.
\]

 While the scaling introduced in the last section preserves the $L^1(\bR^d)$-norm this does not hold true for the $L^2(\bR^d)$ norm. For wavelets $\phi$  we will instead use the $L^2$-scaling: Given any $d$-dimensional wavelet $\phi$ or $\hat\phi^\zeta$ (see \eqref{def:wavelets} for the definition), $\fx\in\bR^d$ and any multi-index $\fn\in\bN^d$, we set
\[
    \phi^\fn_\fx(\fy) := 2^{\frac{1}2\fn} \phi(2^{n_1}(y_1-x_1), \ldots , 2^{n_d}(y_d-x_d))\,,
\]
where we define $\hat\phi^{\zeta,\fn}_\fx$ analogously.  Using the notation $\Delta_n := \{2^{-n} k~\colon~ k\in\bZ\}$ the one-dimensional grid of mesh size $2^{-n}$ for all $n\in\bN$, we recall basic results of wavelet theory which we will use as starting point of our analysis, see \cite[ Chapter 3]{meyer} for the proof. 

\begin{theorem}\label{thm_wave}
For any $r>0$  there exist two functions $\varphi$, $\hat\varphi\in C^r_c(\bR)$ such that
    \begin{enumerate}
        \item $\varphi$ is self replicable: i.e. there exists finitely many $a_k\in\bR, a_k\neq 0$, $k\in\Delta_1$ such that  
        \begin{equation}\label{eq:selfReplicable}
            \varphi= \sum_{k\in\Delta_1}a_k \varphi^1_k\,.
        \end{equation}
        \item For all integer $0\leq m\le r$,
        \[\int_{\R} x^m \hat{\varphi}(x) dx=0\,.\]

        \item For all $n\in\bN$, the set  $\{\varphi^n_x\,,\hat\varphi^m_y\, \colon \,m\ge n\,,x\in\Delta_n\,,y\in\Delta_m\}$ is a complete orthonormal system of $L^2(\R)$.

        \item For all $n\in\bN$, the $\myspan\{\hat\varphi_x^n\colon x\in\Delta_n\}$ is the orthogonal complement of $ \myspan\{\varphi_x^n~\colon x\in\Delta_n\}$ as a subspace of $\myspan\{\varphi_x^{n+1}\colon x\in\Delta_{n+1}\}$. More precisely, any linear combination
        \[
            \psi = \sum_{x\in\Delta_{n+1}} c_x\varphi^{n+1}_x
        \]
        where $c_x = 0$ for all but finitely many $x$ holds can be written as a finite linear combination
        \[
            \psi = \sum_{x\in\Delta_n} b_x\varphi^n_x + \sum_{x\in\Delta_n} \hat b_x\hat\varphi_x^n\,.
        \]
        

        \item For all $\psi\in C_c^r$ one has the convergence
        \begin{align}
            \psi &= \lim_{m\to\infty} \sum_{x\in\Delta_m}\varphi_x^m\scalar{\varphi_x^m,\psi}\label{eq:daubechi1}= \sum_{x\in\Delta_n}\varphi_x^n\scalar{\varphi_x^n,\psi} + \sum_{m\ge n} \sum_{x\in\Delta_m}\hat\varphi_x^m\scalar{\hat\varphi_x^m,\psi}\,,
        \end{align}
        where the limits hold in $L^2(\bR)$\,.
    \end{enumerate}
\end{theorem}
    
 We now pass to the multi-dimensional setting. For any integer $d\geq 1$ and multi-index $\fn\in \mathbb{N}^d$ we set the  $\fn$-th grid to be
\begin{align*}
    \Delta_\fn &:= \left \{\sum_{i=1}^d 2^{-n_i}k_i e_i~\colon~\fk\in\bZ^d\right\}\subset\bR^d\,,
\end{align*}
where $e_i$ is the canonical basis of $\mathbb{R}^d$. Similarly, for any  subset $\zeta \subset [d]$ we can also project the $\fn$-th grid to 
\[  
    \Delta^\zeta_\fn :=  \pi^{\zeta}_{\mathbf{0}}(\Delta_\fn)\,.
\]
 We will use the letter $\zeta$ to usually describe subsets regarding wavelets and will reserve $\theta,\eta$ as much as possible for the rectangular increments of a germ.  The subset notation factorizes easily at the level of wavelets
\begin{definition}\label{defn_wavelets}
For any $r>0$  and  $\zeta\subset[d]$, we define the $d$-dimensional wavelets $\phi$, $\hat\phi\in C^r_c(\bR^d)$ by setting 
\begin{align}
\begin{split}\label{def:wavelets}
    \phi(\fx)= \hat{\phi}^{\emptyset}(\fx)&:= \prod_{i=1}^d\varphi(x_i) \,, \quad 
    \hat\phi^\zeta(\fx) := \prod_{i\notin\zeta}\varphi(x_i)\prod_{i\in\zeta}\hat\varphi(x_i)\,,
\end{split}
\end{align}
where the functions  $\varphi$, $\hat\varphi$ are given by Theorem \eqref{thm_wave}.
\end{definition}


 As in \cite{kern2021}, we can replace the one-dimensional wavelets $\varphi,\hat\varphi$ with shifted wavelets $\varphi_x,\hat\varphi_x$ for any $x\in\bR$ without losing any properties of the wavelets. Thus we can assume that the  wavelets $\phi, \hat{\phi}$ have a support in $[C,R]^d$ for some constants $0<C<R$ and that $a_k\neq 0$ only holds for $k\ge 0$.

We set $V_{\fn} := \text{span}\{\phi_{\fx}^{\fn}\colon \fx\in\Delta_{\fn}\}$ and $\hat V_{\fn}^\zeta := \text{span}\{\hat\phi^{\zeta,\fn}_{\fx}\colon x\in\Delta_{\fn}\}$\, and we denote by  $P_{\fn},\hat P_{\fn}^{\zeta}$ the projections into the respective spaces. Iterating \eqref{eq:daubechi1}, one has formally 
\begin{align*}
    \psi &= \lim_{\fn\to\infty} P_{\fn}\psi = \sum_{\zeta\subset[d]}\sum_{\fk\in\pi_{\mathbf{0}}^{[d]\setminus \zeta}(\bN^d)}\hat P^\zeta_{\fn+\fk}\psi
\end{align*}
and it is easy to see that these equalities do hold for all $\psi\in C_c^r(\bR^d)$ as limits in $L^2(\bR^d)$. The limit $\fn\to\infty$ stands for a limit along any sequence $(\fn^{(k)})_{k\in\bN^d}$ such that $n^{(k)}_i$ goes to $\infty$ for each $i\in[d]$. By \cite[Lemma 62]{kern2021}, the above limits also hold in $C_c^q$ for $q<r$ as long as $r$ is big enough, as well as 
\begin{align*}
    f(\psi) &= \lim_{\fn\to\infty} f(P_{\fn}(\psi)) = \sum_{\zeta\subset[d]}\sum_{\fk\in\pi_{\mathbf{0}}^{[d]\setminus \zeta}(\bN^d)}f(\hat P^\zeta_{\fn+\fk}\psi)
\end{align*}
holds for all (random) distributions $f$. This convergence also comes with quantitative bounds. To calculate those, we use the fact that our wavelets eliminate all functions, which are polynomials of degree $\le r$ in at least one index $i\in\zeta$. More precisely, for any non- empty indexset $\zeta\subset[d]$, $i\in\zeta$ and level $m\le r$, it holds that $\scalar{x_i^m\pi_{\mathbf{0}}^{i}(\psi) , \hat\phi^\zeta} = 0$ for any test function $\psi\in C_c^r(\bR^d)$. We use this to show the following estimate:

\begin{lemma}\label{lem:inequalityPhi}
   Let $r>0$ and let $\psi\in C_c^{rd}$. Then:
    \begin{equation*}
\abs{\scalar{\hat\phi^{\zeta,\fn}_{\fy},\psi^{\flambda}_{\fz}}} \lesssim   2^{-r\fn_\zeta-\frac 12\fn}\flambda^{-r\mathbf{1_\zeta}-\fone}\norm{\psi}_{C_c^{rd}}
    \end{equation*}
uniformly in $\fn\in\bN^d$, $\fy\in\Delta_{\fn}$, 
 $\fz\in\bR^d$, $\zeta\subset[d]$ 
 non empty and $\foo<\flambda\le \fone$.
 \end{lemma}

\begin{proof}
    For $i\in[d]$, we denote by $T^i_{\mathbf{z}}$ the Taylor-polynomial of level $r-1$ in the $i$-th coordinate, i.e.
    \begin{equation*}
        T^i_{\mathbf{z}}\psi(\fy) = \sum_{k=0}^{r-1} \frac 1{k!}\frac{\partial^k}{\partial x_i^k} \psi(\pi_{z_i}^i \fy)(y_i-z_i)^k\,.
    \end{equation*}
   Extending multiplicatively $T^i_{\mathbf{z}}$  we denote by $T^\theta_{\fz} := \prod_{i\in\theta} T^i_{\mathbf{z}}$. Then it holds from the usual Taylor formula with integral remainder that
    \begin{equation*}
        \prod_{i\in\theta}(\Id - T^i_{\mathbf{z}})(\psi)(\fy) = R^\theta(\fz,\fy)
    \end{equation*}
    with the remainder  function $R^{\theta}$ given  by
    \begin{equation*}
        R^\theta(\fz,\fy) = \int_{\fz_\theta}^{\fy_\theta}\left(\frac 1{r!}\right)^{\sharp\theta}\partial^{r\mathbf{1}_\theta}\psi(\pi^\theta_{\ft_\theta}\fy)\prod_{i\in\theta}(y_i-t_i)^{r-1}d\ft_\theta\,.
    \end{equation*}
    By hypothesis on $\psi $ it follows that
    \begin{equation*}
        \abs{R^\theta(\fz,\fy)}\lesssim \norm{\psi}_{C_c^{rd}}\abs{\fy-\fz}^{r}_{\text{c},\theta}\,.
    \end{equation*}
   Since  $\hat{\phi}^\zeta$ integrate 
 to zero when tested against the polynomials $T^\eta_{\fz}\psi$ for  all $\emptyset\neq\eta\subset\zeta$, we have 
   \[\scalar{\hat{\phi}^\zeta,\psi}= \scalar{\hat{\phi}^\zeta,\prod_{i\in\zeta}(\Id-T^i_{\fz})\psi}\,.\]
   Considering now the rescaling, we set $\fgamma := (2^{-n_1},\dots,2^{-n_d})$ and $\fmu := (\lambda_1^{-1},\dots,\lambda_d^{-1})$ and by applying a classical change of variable we calculate
    \begin{align*}
    \abs{\scalar{\hat\phi^{\zeta,\fn}_{\fy},\psi^{\flambda}_{\fz}}} &=2^{-\frac{\fn}{2}} \flambda^{-\fone}\abs{\int_{\R^d}\hat\phi^\zeta(\fu) \psi\bigg(\fgamma*\fmu*\fu\underbrace{-\fmu*\fz+\fmu*\fy}_{=:\fx}\bigg)d\fu}\\ 
    &\le 2^{-\frac{\fn}{2}} \flambda^{-\fone}\int_{\bR^d}\abs{\hat\phi^\zeta(\fu)}\abs{\prod_{i\in\zeta}(\Id-T^i_{\fx})\psi(\fx+\fgamma*\fmu*\fu)}d\fu\\&\lesssim 2^{-\frac{\fn}{2}} \flambda^{-\fone}\int_{\bR^d} \abs{\hat\phi^\zeta(\fu)}\norm{\psi}_{C_c^{rd}}2^{-(r\fn)_{\zeta}}\flambda^{-(r\mathbf{1})_\zeta}\abs{\fu}^{r}_{\text{c},\zeta}d\fu\\
        &= 2^{-\frac{\fn}{2}} 2^{-(r\fn)_{\zeta}}\flambda^{-(r\mathbf{1})_\zeta}\flambda^{-\fone}\norm{\psi}_{C_c^{rd}}\int_{\bR^d}\abs{\hat\phi^\zeta(\fu)}\abs{\fu}^{r}_{\text{c},\zeta}d\fu\,.
    \end{align*}
\end{proof}
\begin{remark}
Note that the convergence speed is faster than one would expect: If one starts with any function $\phi$, one would expect the rates $2^{-\fn/2}\flambda^{-\fone}$ simply from the localizations of $\phi^\fn_\fy$ and $\psi^\flambda_\fz$. However, the fact that $\phi^\zeta$ is orthogonal to polynomials of degree $\le r$ in some $i\in\zeta$ allows us to get the additional factor $2^{-(r\fn)_{\zeta}}\flambda^{-(r\mathbf{1})_\zeta}$. This is a more precise version of \cite[Lemma 61]{kern2021} (which itself is derived from \cite[Proposition 3.20]{Hairer2014}), where they only used the weaker condition, that all $\phi^\zeta$ for $\zeta\neq\emptyset$ eliminate multivariate polynomials of degree $\le r$.
\end{remark}

\subsection{Commuting filtration}
We now recall the probabilistic assumptions to formulate our  results. Throughout this paper, we fix a complete probability space $(\Omega,\cF,\bP)$. Given $m\geq 1$ the set  of $m$- integrable random variables $X\colon \Omega\to \R$ will be denoted by $L_m(\Omega)$ and its norm by $\Vert \cdot \Vert_m $. We will keep this discussion rather short, as it has been done in length in \cite{multiparameterStochSewing}. One can find more classical references in \cite{Imkeller85,walsh1986,Khoshnevisan02}. We equip our probability space with a filtration which is compatible with the use of rectangular increments. 

\begin{definition}
    A family of sigma-algebras $(\cF_\ft)_{\ft\in\R^d}, \cF_\ft\subset\cF$ for all $\ft\in\bR^d$, is called a multiparameter filtration, if for all $\fs\le\ft$, we have $\cF_\fs\subset\cF_\ft$.
\end{definition}
\begin{remark}
    Alternatively, it is possible to consider filtrations indexed by $\ft\in K$ with $K\subset \R^d$ a compact set. We will use that when considering SPDEs and all definitions will extend trivially.
\end{remark}

\begin{definition}
For any given multiparameter filtration  $(\cF_\ft)_{\ft\in\R^d}$ and subset $\theta\subset[d]$, we set
\begin{equation*}
    \cF^\theta_\ft :=\bigvee_{\fs\in\R^d}\cF_{\pi^\theta_\ft\fs}\,,
\end{equation*}
where $\bigvee$ represents the sigma-algebra generated by the union. When  $\theta = \{i\} $ we call $\cF^i := \cF^{\{i\}}$ the $i$-th marginal filtration of $(\cF_\ft)$. 
\end{definition}
 We denote by $\bE^\theta_\ft$ the conditional expectation $\bE[\cdot\vert\cF_\ft^\theta]$. A natural condition to impose on a filtration  $\cF$ to compute $\bE[\cdot\vert\cF_\ft^\theta]$ efficiently is the notion of commutativity. I.e. we assume that for all $i,j\in[d]$ and $\fs, \ft\in \R^d$ one has the a.s. identity
\begin{equation}\label{defn_comm_filtration}
        \bE^i_\ft \bE^j_\fs = \bE^j_\fs \bE^i_\ft\,.
\end{equation}
We call a filtration $(\cF_\ft)_{\ft\in\bR^d}$ fulfilling \eqref{defn_comm_filtration} a commuting filtration. It should be noted that in the literature, several equivalent properties to \eqref{defn_comm_filtration} are used. Walsh introduced in \cite[Section 1]{Cairoli1975},  the condition that for all $\fx\in\bR^d$, $\cF^{i}_\fx, i=1,\dots, d$ are conditionally independent on $\cF_\fx$. \cite[Section 3.4]{Khoshnevisan02}  prefers to assume that for all bounded random variables $X:\Omega\to\bR$ which are $\cF_\ft$ measurable for some $\ft\in\bR^d$, we have 
\begin{equation*}
    \bE^{[d]}_\fs X = \bE^{[d]}_{\fs\wedge\ft} X\,,
\end{equation*}
 for all $\fs\in\R^d$. In both cases, the authors show that their condition is equivalent to \eqref{defn_comm_filtration}. We recall two basic results of multiparameter filtrations.

\begin{proposition}
Let $(\cF_\ft)_{\ft\in\R^d}$ be a multiparameter filtration. We then have the following properties 
\begin{itemize}
    \item[i)] If $(\cF_\ft)$ is complete, then so is $(\cF_\ft^\theta)$ for all subsets $\theta\subset[d]$. 
    \item[ii)]  If $\cF$ is commuting, then for all subsets $\theta\subset[d]$
    \begin{equation*}
        \bE^\theta_\ft = \prod_{i\in\theta} \bE^i_\ft.
    \end{equation*}
\end{itemize}
\end{proposition}

\begin{proof}
    i) immediately follows from $\cF_\ft\subset\cF_\ft^\theta$ for all $\ft\in\bR^d, \theta\subset[d]$.
    ii) is shown in \cite{Khoshnevisan02} for $\theta =[d]$ in Theorem 3.4.1. The case $\theta\subset[d]$ follows with the same proof.
\end{proof}
    
 The stochastic reconstruction theorem, as the stochastic sewing lemma, is based on BDG-type inequalities which in $d=1$ can be found in \cite{le2020}, Equation (2.5). The multiparameter version of this inequality has been found in \cite[Lemma 25]{multiparameterStochSewing}.

\begin{lemma}[extended BDG inequality]\label{lem:multiparameterBDG}
    Let  $i\in\{1,\dots,d\}$ and $I_i\subset\bN$ be a finite set with minimal element $y_i$. Define  for each indexset $\theta\subset [d]$ 
    \[
    I_\theta = \prod_{i\in\theta} I_i\times\prod_{j\notin\theta}\{y_j\} := \{\fx\in\bN^d~\colon~ x_i\in I_i \text{ if $i\in\theta$, } x_i = y_i\text{ otherwise}\}\,.
    \]
    Further, let $(\mathcal F_\fn)_{\fn\in\bN^d}$ be a commuting $d$-parameter filtration and let $Z_\fk\in L_m(\Omega)$ for each $\fk\in I_{[d]}$. We assume that $Z_\fk$ is $\cF_{\fk+ \boldsymbol 1}$ measurable. Assume that there exist constants $0<b_{i,k_i}, a_{i,k_i}$, $i\in[d], k_i\in I_i$, as well as a $c>0$, such that for each $\theta \subset[d]$:
\begin{equation}\label{ineq:lem3Condtition}
    \norm{\bE^\theta_\fk Z_{\fk}}_m\le c \,b_{\fk,\eta} \,a_{\fk,\eta^c}\,,
\end{equation}
where $b_{\fk,\eta} := \prod_{i\in\eta} b_{i,k_i}$ and $a_{\fk,\eta^c} = \prod_{i\in\eta^c}a_{i,k_i}$. Then, it follows that for each $\theta,\eta\subset[d]$:
\begin{equation}\label{ineq:lem3Result}
    \norm{\bE^\eta_{\fy}\sum_{\fk\in I_{\theta}} Z_{\fk}}_m\lesssim c\sum_{\theta\setminus\eta = \theta_1\sqcup \theta_2} \left(\sum_{\fk\in I_{\theta\setminus\theta_2}} b_{\fk,\eta\cup\theta_1}\right) \left(\sum_{\fl\in I_{\theta_2}} a_{\fl,(\eta\cup\theta_1)^c}^2\right)^{\frac 12}\,,
\end{equation}

where we sum over all disjoint $\theta_1,\theta_2$, such that $\theta\setminus\eta=\theta_1\cup\theta_2$. It especially follows

\begin{equation*}
    \norm   {\sum_{\fk\in I_{\theta}} Z_{\fk}}_m\lesssim c\sum_{\theta = \theta_1\sqcup\theta_2} \left(\sum_{\fk\in I_{\theta_1}} b_{\fk,\theta_1}\right) \left(\sum_{\fl\in I_{\theta_2}} a_{\fl,\theta_2}^2\right)^{\frac 12}\,.
\end{equation*}
\end{lemma}

 Passing to the reconstruction theorem, we can relate the properties of a multiparameter filtration with the definition of a random germ, i.e. a random field with state space $\mathcal{D}_r'$ in a proper notion of measurability. Recall that we can think of locally integrable stochastic processes $X$ as the distribution $\psi\mapsto \int_{\bR^d} \psi(\fx) X_\fx d\fx$, which motivates the measurability notion for a distribution: We say that $f$ is adapted to $(\cF_\fx)_{\fx\in\bR^d}$ if $f(\psi)$ is $\cF_\fy$-measurable for all $\psi$ with $\supp(\psi)\subset(-\infty,y_1]\times\dots\times(-\infty,y_d]$. If we replace the distribution $f$ with a germ $(F_\fx)_{\fx\in\bR^d}$, it is reasonable to assume that $F_\fx$ might additionally depend on the information contained in $\cF_\fx$. This leads to the following definition.

\begin{definition}\label{defn_adapted}
    We say that a random germ $(F_\fx)$ is adapted to a multiparameter filtration $(\cF_\fx)$, if  for all $\fx\in\bR^d$ and test function $\psi\in C_c^r$ with support in $(-\infty,\fy]$, $F_\fx(\psi)$ is $\cF_{\fx\vee\fy}$-measurable.
\end{definition}

\section{The stochastic reconstruction theorem}\label{sec:reconstruction}
We now pass to formulating the main theorem of the paper. To make it rigorous we will have to introduce  the proper spaces of functions and distributions where the reconstruction map acts and describe their main properties. We fix a time horizon $T>0$ throughout this section.

\subsection{H\"older spaces for random distributions}\label{sec:hoelderSpacesRandom}
 We start by recalling the deterministic multiparameter H\"older spaces, see  \cite{tindel07,harang21} for further references. Given $\falpha\in [0,1]^d$ we set $C^\falpha$ to be the space of all maps $Y\colon [0,T]^d\to \mathbb{R}$, such that
\begin{equation}\label{eq_det_Calpha}
\norm{Y}_{C^\falpha} := \sum_{\theta\subset[d]} \norm{\square^\theta Y}_{\falpha,\theta} <\infty\,,\quad  \norm{\square^\theta Y}_{\falpha,\theta}:=\sup_{\fx, \fy\in [0,T]^d,\, \fx\neq \fy} \frac{|\square^\theta_{\fx,\fy} Y|}{|\fx- \fy|^{\falpha}_{\text{c}, \theta}}\,.
\end{equation}
This definition can be extended naturally to any $\falpha\ge\foo$  but in our case, we will always consider $\falpha\le \fone$. On the other hand, we will also consider multi-index H\"older spaces where the components of  $\falpha$ are negative. In this case we  set $r$ as the smallest possible integer  such that $ r\ge \sum_{i=1}^d \abs{\alpha_i}$ and we look at continuous linear maps over $f:C_c^r ([0,T]^d)\to\R$. More precisely, we set $C^{\falpha}([0,T]^d)$ as the set of all $f\in \cD^r([0,T]^d)$ such that 
\begin{equation*}
    \norm{f}_{\falpha} = \sup_{\fx,\psi,\flambda\in(0,1]} \frac{\abs{f(\psi_{\fx}^{\flambda})}}{\flambda^{\falpha}}<\infty\,,
\end{equation*}
where the supremum runs over all $\fx\in[0,T]^d, \psi\in B^r([0,T]^d)$ and $\foo<\flambda\le \fone$ such that $\psi_\fx^{\flambda}$ has support in $[0,T]^d$. 

To extend these spaces when there is a filtered probability space $(\Omega,\cF,\bP)$ with a commuting $d$-parameter filtration $(\cF_\ft)_{\ft\in[0,T]^d}$, we need to condition the random distributions over a careful choice of filtrations. More specifically, we want to quantify the seminorm $\square^\theta_{\fx,\fy} Y$  in \eqref{eq_det_Calpha} with a  $L_m$ norm and additionally we want the conditional expectation $\bE^\eta_\fx$ to add another regularity factor.
\begin{definition}\label{def:C_alpha-stoc}
    Let $\fdelta \ge \foo$, $\falpha\in [0,1]^d$ and $m\ge 1$. We then set $C^{\falpha,\fdelta}L_m$ to be the space of random fields $Y:[0,T]^d\times\Omega\to\bR$ such that for all $\ft\in[0,T]^d$, $Y_\ft$ is $\mathcal F_\ft$-measurable and
     for all $\eta\subset\theta\subset[d]$ one has
    \begin{align}\label{theta_eta_seminorm}
        \norm{\square Y}_{\theta,\eta,m} &:= \sup_{\fx,\fy\in[0,T]^d,\; x_i\le y_i \text{ for } i\in\theta}\frac{\norm{\bE^\eta_\fx\square^\theta_{\fx,\fy} Y}_m}{\abs{\fx-\fy}^{\falpha}_{\text{c},\theta}\abs{\fx-\fy}^{\fdelta}_{\text{c},\eta}}<\infty\,.
    \end{align}
\end{definition}

 In most cases, $\falpha,\fdelta$, and $T$ are clear from the context. If they are not, we will write $\norm{\square Y}_{\falpha,\fdelta,\theta,\eta,m}$ and $C^{\falpha,\fdelta}L_m([0,T]^d)$. Recall that for $\theta = \emptyset$, the above quantities simplify to
\begin{equation*}
\norm{Y}_{\emptyset,\emptyset,m} = \sup_{\fx\in [0,T]^d}\norm{Y_\fx}_m  =:\norm{Y}_{m,\infty}\,.  
\end{equation*}
The space $C^{\falpha,\fdelta}L_m$ is trivially a Banach space with respect to the norm
\begin{align*}
    \norm{Y}_{C^{\falpha,\fdelta}L_m} &:= \norm{Y}_{m,\infty}+\bracket{Y}_{C^{\falpha,\fdelta}L_m}\,, \\ \bracket{Y}_{C^{\falpha,\fdelta}L_m} &:= \sum_{\eta\subset\theta\subset[d],\theta\neq\emptyset} \norm{\square Y}_{\theta,\eta,m}\,.
\end{align*}
Note that $\bracket{Y}$ is not a norm on $Y$ but only a semi-norm over $C^{\falpha,\fdelta}L_m$. For $\fdelta = \foo$, we write $C^{\falpha} L_m := C^{\falpha,\foo}L_m$. 

 The most simple examples of element of $C^{\falpha,\fdelta}L_m$ spaces is a Brownian sheets $B=\{B_\ft \colon \ft\in\bR_+^d\}$ together with its canonical  filtration. By definition, a Brownian sheet is a centred Gaussian process characterized by the covariance structure
\[
    \bE(B_\ft B_\fs) = \prod_{i=1}^d t_i\wedge s_i,\qquad \fs,\ft\in\bR^d_+\,,
\]
see for example \cite{Khoshnevisan02,Cairoli1975}. It especially fulfills for any $m\ge 1$
\begin{equation*}
    \norm{\square^\theta_{\fs,\ft} B}_m \le c(m)\abs{\ft-\fs}_{\text{c},\theta}^{\frac 12 \fone}.
\end{equation*}
for some constant $c(m)\geq 0$ depending on $m$ and all $\fs\le\ft,\theta\subset[d]$. Considering its conditional expectation with respect to the canonical filtration of the process, for any $\emptyset\neq\eta\subset\theta$ one has 
\begin{equation*}
    \bE^\eta_\fs \square^\theta_{\fs,\ft} B = 0\,.
\end{equation*}
This condition is due to the property of being a strong martingales, as described by \cite{Imkeller85,walsh1986}. To take in account the ``infinite regulation" of these elements,  we set
\[\norm{Y}_{C^{\falpha,\infty}L_m} = \sup_{\fdelta>\foo} \norm{Y}_{C^{\falpha,\fdelta}L_m}\]
and write $Y\in C^{\falpha,\infty} L_m$ if $\norm{Y}_{C^{\falpha,\infty}L_m}<\infty$, so that one has $B\in C^{\frac 12\fone,\infty} L_m$. Another simple example of an element in $C^{\falpha,\fdelta}L_m$ is given by any random field of the form $Y_{\ft}=B_{\ft}+ Z_{\ft}$ with $Z\in C^\fbeta$ for  $\beta\geq 1/2\fone $. In this case one has $Y\in C^{\falpha,\fdelta}L_m$ with  $\fdelta = \beta- 1/2\fone$. 

We now pass on describing the embedding properties of these spaces. Concerning the index  $\falpha$ the spaces $C^{\falpha,\fdelta}L_m $ behave exactly as the classical H\"older spaces. That is, for $\foo < \falpha <\fbeta <\mathbf 1$ one  has $C^{\fbeta,\fdelta} L_m \subset C^{\falpha,\fdelta} L_m$. Moreover, it follows trivially from the definition \eqref{theta_eta_seminorm} that for any $\eta\subset\theta\subset[d],\theta\neq\emptyset$ we have
\begin{equation}\label{ineq:CalphaEmbedding0}
    \norm{\square Y}_{\falpha,\fdelta,\theta,\eta,m} \le (T\mathbf{1})^{\fbeta-\falpha}_\theta\norm{\square Y}_{\fbeta,\fdelta,\theta,\eta,m}\,.
\end{equation}
This estimate gets even stronger if $Y$ vanishes on $\foo$ or alternatively  on the whole boundary $\partial_{\foo} [0,T]^d := \bigcup_{i=1}^{d}\{\fx\in[0,T]^d\colon  x_i = 0\}$.
\begin{proposition}\label{prop:CalphaEmbedding}
    For any $\foo < \falpha <\fbeta <\mathbf 1$ we have $C^{\fbeta,\fdelta}L_m([0,T]^d)\subset C^{\falpha,\fdelta}L_m([0,T]^d)$. Moreover, we have the following estimates:
    \begin{itemize}
        \item for any $Y\in C^{\fbeta,\fdelta}L_m$ such that $Y(\foo)=0$
        \begin{equation}\label{ineq:CalphaEmbedding1}
            \norm{Y}_{C^{\falpha,\fdelta}L_m}\lesssim \left(\sum_{\emptyset\neq\theta\subset[d]}(T\fone)^{\fbeta}_{\theta}+ \sum_{\emptyset\neq\theta\subset[d]} (T\mathbf{1})^{(\fbeta-\falpha)}_\theta\right)\norm{Y}_{C^{\fbeta,\fdelta} L_m}\,.
        \end{equation}

        \item for any $Y\in C^{\fbeta,\fdelta}L_m$ such that $Y(\fx)=0$ for any $\fx\in \partial_{\foo} [0,T]^d$
        \begin{equation}\label{ineq:CalphaEmbedding2}
            \norm{Y}_{C^{\falpha,\fdelta}L_m}\lesssim \left(\sum_{\theta\subset[d]}(T\mathbf{1})^{\fbeta-\falpha_\theta}\right)\norm{Y}_{C^{\fbeta,\fdelta}L_m}\,.
        \end{equation}
    \end{itemize}
\end{proposition}

\begin{proof}
    Let us first show \eqref{ineq:CalphaEmbedding1}. Applying  Lemma \ref{lem:shiftSquareIncrements}, we can write
    \begin{equation*}
        Y(\fx) = Y(\foo) + \sum_{\emptyset\neq\theta\subset[d]} (-1)^{\sharp(\theta^c)}\square^{\theta}_{\pi_\fx^{\theta^c} \foo,\fx} Y\,. 
    \end{equation*}
    Thus from the condition $Y(\foo) = 0$ and the usual triangular inequality  one immediately has 
    \begin{equation*}
        \norm{Y}_{m,\infty}\le \sum_{\emptyset\neq\theta\subset[d]} (T\fone)^{\fbeta}_{\theta}\norm{\square Y}_{\fbeta,\fdelta,\theta,\emptyset,m} \,.
    \end{equation*}
    The final inequality \eqref{ineq:CalphaEmbedding1} then follows from the above together with \eqref{ineq:CalphaEmbedding0}.

  We pass to \eqref{ineq:CalphaEmbedding2}. If $Y(\fx) = 0$ for all $\fx\in B$, we get that
    \begin{equation*}
        Y(\fx) = \square_{\foo,\fx}^{[d]} Y
    \end{equation*}
    which immediately gives us
    \begin{equation*}
        \norm{Y}_{m,\infty}\le (T\fone)^{\fbeta}\norm{\square Y}_{\fbeta,\fdelta,[d],\emptyset,m} \,.
    \end{equation*}
    Furthermore, from the same hypothesis we have for all $\fx,\fy\in[0,T]^d$ and $\theta\subset[d]$
    \begin{equation*}
        \square_{\fx,\fy}^\theta Y = \square^{[d]}_{\pi^\theta_\fx \foo,\fy} Y\,,
    \end{equation*}
    which immediately gives one for any $\eta\subset\theta\subset[d]$ the inequality
    \begin{equation*}
        \norm{\square Y}_{\falpha,\fdelta,\theta,\eta,m} \lesssim (T\fone)^{\fbeta-\falpha_{\theta}}\norm{\square Y}_{\fbeta,\fdelta,[d],\eta,m} \,.
    \end{equation*}
    Summing both sides over $\eta\subset\theta\subset[d],\theta\neq\emptyset$ we derive \eqref{ineq:CalphaEmbedding2}.
\end{proof}
Using the same reasoning, we can also describe the embedding properties of $C^{\falpha,\fdelta}L_m  
$ with respect to the parameter $\fdelta$.
\begin{proposition}\label{prop:CalphaEmbedding2}
    Let $\falpha\in[0,1]^d$ and let $0\le\fdelta\le\fgamma$. Then we have $C^{\falpha,\fgamma} L_m([0,T]^d)\subset C^{\falpha,\fdelta} L_m([0,T]^d)$ with the estimates:
    \begin{equation*}
        \norm{Y}_{C^{\falpha,\fdelta}L_m} \lesssim \left(\sum_{\eta\subset[d]} T_\eta^{\fgamma-\fdelta}\right) \norm{Y}_{C^{\falpha,\fgamma}L_m}\,.
    \end{equation*}
    where the constant in $\lesssim$ only depends on $d$.
\end{proposition}

For negative $\falpha < \foo$, we no longer have an equivalent of  $\square^\theta Y$ for $\theta\subsetneq [d]$. Thus, we can forget one of the indexsets in Definition \ref{def:C_alpha-stoc} and naturally adapt the deterministic definition into a stochastic setting.

\begin{definition}
    Let $\falpha<\foo$, $1\le m$ and $\fdelta\ge \foo$.  We define  $C^{\falpha,\fdelta}L_m$ as the space of all random distributions $f\in\cD^r$ with $r$ chosen as in the deterministic setting, satisfying the following properties:
    
    \begin{itemize}
        \item For any test function $\psi$ with support in $(-\infty,\fx]$, $f(\psi)$ is $\cF_{\fx}$-measurable.

        \item for all $\psi\in B^r([1/4,3/4]^d)$ and $\flambda\in(0,1]^d$, $\fx\in[0,T]^d$ such that $[\fx,\fx+\flambda]\subset[0,T]^d$ we have  uniformly over $\eta\subset[d]$:
        \begin{align}
        \begin{split}\label{ineq:defNegativHoelderSpace}
            \norm{f(\psi^{\flambda}_\fx)}_m &\lesssim \flambda^{\falpha}\\
             \norm{\bE^\eta_\fx f(\psi^{\flambda}_\fx)}_m&\lesssim \flambda^{\falpha}\flambda^{\fdelta}_{\eta}\,,
        \end{split}
        \end{align}
    \end{itemize}
\end{definition}
\begin{remark}\label{rk:eps_rmk}
    Note that wavelet approximations $P_\fn\psi$ enlarge the support of $\psi$. Thus, by using test functions with support in $[1/4,3/4]^d$ we can ensure that $P_\fn(\psi_\fx^\flambda)$ has support in $[\fx,\fx+\flambda]$ for any $\fx\in\bR^d,\flambda\in(0,1]^d$ and large enough $\fn\in\bN^d$. This is especially important to calculate the Hölder regularity of the reconstruction, see Corollary \ref{cor:reconExtendedCoherent}. The constants $1/4,3/4$ are chosen arbitrarily, any $0<C<R<1$ would work the same.
\end{remark}
$C^{\falpha,\fdelta}L_m$ becomes a Banach space by using the norm
\begin{equation*}
\norm{f}_{C^{\falpha,\fdelta}L_m}:= \sup_{\psi,\fx,\flambda,\eta\subset[d]}\frac{\norm{\bE^\eta_\fx f(\psi^{\flambda}_\fx)}_m}{\flambda^{\falpha}\flambda^{\fdelta}_{\eta}}\,,
\end{equation*}
where the supremum runs over the $\psi,\fx,\flambda$ as described above. As in the positive case, we use the notation $C^{\falpha}L_m := C^{\falpha,\foo}L_m$ and write $f\in C^{\falpha,\infty}L_m$ if
\[
    \norm{f}_{C^{\falpha,\infty}L_m} = \sup_{\fdelta >\foo} \norm{f}_{C^{\falpha,\fdelta}L_m}<\infty\,.
\]

Similarly to the case when $\falpha>0$, in this setting we can  simple explain the  elements of $C^{\falpha,\fdelta}L_m$ with negative $\falpha$ by simply applying a mixed distributional derivative of the previous examples.  In case of Brownian sheet over $\bR_+^d$ we obtain white noise over $\bR_+^d$ that is a the centered Gaussian process $\xi(\psi)$, $\psi\in L_2(\bR_+^d)$ with covariance structure given by
\[
\bE(\xi(\psi_1)\xi(\psi_2)) = \int_{\bR^d_+} \psi_1(\fz)\psi_2(\fz) d\fz,\qquad \psi_1,\psi_2\in L_2(\bR_+^d)\,,
\]
see \cite{nualart1995malliavin} for reference. Using the Gaussianity of $\xi$, one can calculate
\begin{align*}
    \norm{\xi(\psi_\fx^\flambda)}_{L_m} = c(m) \norm{\xi(\psi_\fx^\flambda)}_{L_2}= c(m) \flambda^{-\frac 12 \fone}\norm{\psi}_{L_2(\bR^d)}\,,
\end{align*}
for any $\psi\in L_2(\bR^d), \fx\in\bR^d_+$ and $\flambda\in(0,1]^d$, such that $\psi_\fx^\flambda$ has support in $\bR^d_+$. It follows also that
\[
    \bE^\eta_\fx \xi(\psi_\fx^\flambda) = 0\,,
\]
for $\eta\subset[d]$ and $\supp(\psi)\subset[0,1]^d$, thereby obtaining $\xi\in C^{-\frac 12\fone,\infty}L_m$. Simple examples of elements  in  $C^{-\frac 12\fone,\fdelta}L_m$  with finite $\fdelta$ are given by random distribution of the form $\xi+ \zeta $ where $\zeta\in  C^{\fbeta}$  with $-1/2\fone<\fbeta<\foo$. We conclude by providing a useful result regarding the sum of distributions in different $C^{\falpha,\fdelta}L_{m}$ spaces.

\begin{lemma}\label{lem:sumOfCalpha}
    Let $f\in C^{\falpha,\fdelta}L_m$ and $g\in C^{\tilde\falpha,\tilde\fdelta}L_m$ for $\fdelta,\tilde\fdelta\ge 0$ and $\falpha,\tilde\falpha <0$. Then we have $f+g\in C^{\fbeta,\fgamma}L_m$ where 
    \begin{align*}
        \fbeta = \falpha\wedge\tilde\falpha\,, 
        \quad \fgamma = (\falpha+\fdelta)\wedge(\tilde\falpha+\tilde\fdelta)-\fbeta\,.
    \end{align*}
    Furthermore, we have the estimate
    \begin{equation}\label{ineq:sumOfCalphaNegative}
        \norm{f+g}_{C^{\fbeta,\fgamma}L_m} \le \norm{f}_{C^{\falpha,\fdelta}L_m} + \norm{g}_{C^{\tilde\falpha,\tilde\fdelta}L_m}\,.
    \end{equation}
\end{lemma}

\begin{proof}
    Simple calculation gives
    \begin{align*}
    \norm{\bE_\fx^\eta(f+g)(\psi_\fx^\flambda)}_m &\le \norm{\bE_\fx^\eta f(\psi_\fx^\flambda)}_m + \norm{\bE_\fx^\eta g(\psi_\fx^\flambda)}_m \\
    &\le (\norm{f}_{C^{\falpha,\fdelta}L_m}+\norm{g}_{C^{\falpha,\fdelta}L_m}) \flambda^{\fbeta}\flambda_\eta^{\gamma}\,,
    \end{align*}
  where $\flambda,\fx,\psi,\eta$ are chosen as in \eqref{ineq:defNegativHoelderSpace}. One checks that $\norm{(f+g)(\psi_\fx^\flambda)}_m$ follows a similar bound. A straightforward calculation then gives \eqref{ineq:sumOfCalphaNegative}.
\end{proof}

\begin{remark}\label{rem:sumCAlpha}
    It is straightforward to see that a similar lemma holds for positive $\falpha,\tilde\falpha$. 
\end{remark}

\subsection{Coherent stochastic germs and main statement}

We now pass to the main statement of the article. Similar to \cite{zambotti2020, kern2021}, we first introduce the notion of \emph{coherence}, which is the minimal assumption under which reconstruction is possible. For any $\falpha<\foo$, we choose $r>0$ to be the smallest integer, such that $r+\alpha_i > 0$ holds for all $i\in[d]$. Throughout this section, we assume $(\cF_\ft)_{\ft\in[0,T]^d}$ to be a commuting filtration.

\begin{definition}[stochastic coherence]\label{def:extendedCoherence}
Let $\falpha,\fgamma,\fdelta\in\R^d$ with $\fdelta \ge \foo$ and $\fgamma\ge\falpha$. An $(\mathcal F_\fx)_{\fx\in[0,T]^d}$-adapted germ $(F_\fx)$ is called stochastic $(\falpha,\fgamma,\fdelta)$-coherent, if for all $\eta\subset\theta\subset[d]$, $\psi\in B^{r}([0,1]^\theta\times[-1,1]^{\theta^c})$ and $\fx,\fy\in [0,T]^d$ such that $\fx_\theta\le \fy_\theta$ and $\supp(\psi_\fy^\flambda)\in[0,T]^d$, it holds that
\begin{align}\label{ineq:coherence}
    \norm{\bE^\eta_\fx \square^\theta_{\fx,\fy} F(\psi_\fy^{\flambda})}_m &\lesssim \flambda^{\falpha}(\compAbs{\fx-\fy}+\flambda)^{\fgamma-\falpha}_{\theta}(\compAbs{\fx-\fy}+\flambda)^{\fdelta}_{\eta}\,.
\end{align}
\end{definition}

\begin{remark}
Let us explain why we prefer to consider test functions with support in $[0,1]^\theta\times[-1,1]^{\theta^c}$. While one can work with test functions with generic support in $[0,T]^d$, it is more convenient to work with a normalized support in $[-1,1]^d$ and explore more general cases by recentering and rescaling $\psi$ to get $\psi_\fx^\flambda$, see for example \cite{zambotti2020} or \cite{Hairer2014} for similar settings. Since our theory is stochastic, we make the further distinction of whether a test function only ``sees the future'' in the direction $i\in[d]$, i.e. has support in $[0,1]$, or has generic support in $[-1,1]$. So the condition $\supp(\psi)\subset [0,1]^\theta\times[-1,1]^{\theta^c}$ allows us to get the correct measurability conditions later on.
\end{remark}

\begin{remark}
As it turns out, this is equivalent to working with test functions $\psi$ with support in $[0,1]^d$. Since $\square^\theta_{\fx,\fy} F$ does not depend on $\fy_{\theta^c}$, it holds for any $\fz\in\bR^d$, that
\[
\square_{\fx,\fy}^\theta F(\psi_\fy) = \square_{\fx,\fy+\fz_{\theta^c}}^\theta F(\hat\psi_{\fy+\fz_{\theta^c}})\,,
\]
where $\hat\psi = \psi_{-\fz_{\theta^c}}$. This allows one to switch from test functions with support in $[0,1]^d$ to test functions with support in $[0,1]^\theta\times[-\frac 12,\frac 12]^{\theta^c}$ at no cost. We prefer to remember that the support of $\psi$ is allowed ``to see the past'' for indexes $i\in\theta^c$, so we explicitly allow supports in $[0,1]^\theta\times[-1,1]^{\theta^c}$.
\end{remark}

 We will refer to these elements as stochastic $(\falpha,\fgamma,\fdelta)$-coherent germs and we denote them with the set $G^{\falpha,\fgamma,\fdelta}L_m$. We set the norm for this space to be
\begin{equation}\label{ineq:coherenceNorm}
    \norm{F}_{G^{\falpha,\fgamma,\fdelta}L_m} := \sup_{\fx,\fy,\theta,\eta,\flambda,\psi}\frac{\norm{\bE^\eta_\fx \square^\theta_{\fx,\fy} F(\psi_\fy^{\flambda})}_m}{\flambda^{\falpha}(\compAbs{\fx-\fy}+\flambda)^{\fgamma_\theta-\falpha_\theta+\fdelta_\eta}_{\theta}}\,,
\end{equation}
where the supremum runs over $\fx,\fy,\eta,\theta,\flambda$ and $\psi$ as in the above definition.  Note that this implies for $\eta=\emptyset=\theta$ that    $\norm{F_\fx(\psi_\fy^\flambda)}_m\lesssim \flambda^{\falpha}\,,$
giving us that $F_\fx$ belongs to $C^{\falpha,\foo}L_m([0,T]^d)$. This is only reasonable for $\falpha<\foo$, so we will assume from this point onwards that $\falpha <\foo$. Further note that the assumption $\fgamma\ge\falpha$ comes for free, as we can always lower $\falpha$ at the cost of a constant and a higher $r$.

Since the right-most point of $\square_{\fx,\fy}^\theta$ is not $\fy$ but $\pi^\theta_\fy \fx$, it would be reasonable to use $\psi_{\pi^\theta_\fy\fx}^\flambda$ as the localized test function. That is, one might want to replace property \eqref{ineq:coherence} with
\begin{equation*}
    \norm{\bE^\eta_\fx \square^\theta_{\fx,\fy} F(\psi_{\pi^\theta_{\fy}\fx}^{\flambda})}_{m} \lesssim \flambda^{\falpha}(\compAbs{\fx-\fy}+\flambda)^{\fgamma-\falpha}_{\theta}(\compAbs{\fx-\fy}+\flambda)^{\fdelta}_{\eta}\,.
\end{equation*}
The following result shows some equivalent ways to state the coherence.
\begin{lemma}
    The following conditions are equivalent:
    \begin{enumerate}[i)]
        \item $F$ is a stochastic $(\falpha,\fgamma,\fdelta)$-coherent germ.
        \item $F$ fulfills for each $\eta\subset\theta\subset[d]$ and $\fx,\fy$ with $\fx_{\theta}\le\fy_{\theta}$
        \begin{equation*}    \norm{\bE^\eta_\fx\square_{\fx,\fy}^\theta F(\psi_{\pi^\theta_\fy\fx}^{\flambda})}_m \lesssim \flambda^{\falpha}(\compAbs{x-y}+\flambda)^{\fgamma-\falpha}_{\theta}(\compAbs{x-y}+\flambda)^{\fdelta}_{\eta}\,.
        \end{equation*}
        \item For all $\fk = \fk_{\theta^c}\in [0,T]^d$ and for $\eta,\theta,\fx,\fy,\psi,\flambda$ as in the coherence property, we have
        \begin{equation*}
\norm{\bE^\eta_\fx\square_{\fx,\fy}^\theta F(\psi_{\fy+\fk_{\theta^c}}^{\flambda})}_m \lesssim \flambda^{\falpha}(\compAbs{\fx-\fy}+\flambda)^{\fgamma-\falpha}_{\theta}(\compAbs{\fx-\fy}+\flambda)^{\fdelta}_{\eta}\,,
        \end{equation*}
 where the constant in $\lesssim$ does not depend on $\fk$.
    \end{enumerate}
\end{lemma}
\begin{proof}
    $i)$ and $iii)$ are obviously equivalent, since $\square^\theta_{x,y} = \square^\theta_{x,y+{k_\theta^c}}$. 
    $i)\Rightarrow ii)$ follows from the simple identity $\square_{\fx,\fy}^\theta = \square_{\fx,\pi_\fy^\theta\fx}^\theta$.
    On the other hand, $ii)\Rightarrow i)$ is shown by using Lemma \ref{lem:shiftSquareIncrements} to calculate
    \begin{align*}
        \norm{\bE^\eta_\fx\square_{\fx,\fy}^\theta F(\psi_\fy^{\flambda})}_m &\le \sum_{\theta^c = \theta_1\sqcup\theta_2}\norm{\bE^\eta_\fx \square^{\theta\cup\theta_2}_{\pi_\fy^{\theta_1} \fx,\fy}F(\psi_\fy^{\flambda})}_m\lesssim\flambda^{\falpha}(\compAbs{\fx-\fy}+\flambda)^{\fgamma-\falpha}_{\theta}(\compAbs{\fx-\fy}+\flambda)^{\fdelta}_{\eta}\,.
    \end{align*}
\end{proof}

 Using these notions, we can now state precisely the stochastic reconstruction theorem.

\begin{theorem}[Stochastic reconstruction theorem on rectangular increments]\label{theo:reconstruction}
Let $(F_\fx)$ be a stochastic $(\falpha,\fgamma,\fdelta)$-coherent germ for $\falpha,\fgamma,\fdelta\in\bR^d, \fdelta\ge 0, \fgamma\ge\falpha,\falpha<\foo$ and let $m\ge 2$. Assume that $\fgamma+\fdelta>\foo$, and that $\fgamma > -\frac{1}{2}\fone$. Then there exists a unique family of random distributions $\mathcal{R}_\fx^\eta(F)\in \cD^{rd}((0,T)^d)$  with $\fx\in [0,T]^d$ and $\eta\subset[d]$ characterized by the following properties:

\begin{enumerate}
    \item $\mathcal{R}^\emptyset_\fx (F)= F_\fx$.
    \item $\mathcal{R}^\theta_\fx(F)$ only depends on $x_i$ for $i\in\theta^c$, i.e. $\mathcal{R}^{\theta}_\fx(F) = \mathcal{R}^{\theta}_{\pi_\foo^{\theta} \fx}(F)$.
    \item For all $\theta\subset[d]$, $\fx\in[0,T]^d$ and $\psi$ with support in $[0,y_1]\times\dots\times[0,y_d]$, $\mathcal{R}^\theta_\fx(F)(\psi)$ is $\mathcal F_{\fx\vee\fy}$-measurable.
    \item For all $\psi$ with support in $[1/4,3/4]^{\theta}\times[-3/4,3/4]^{\theta^c}$ and $\norm{\psi}_{C_c^{rd}} = 1$, any subset $\eta\subset\theta\neq\emptyset$ one has the estimate
\begin{equation}\label{cond4Recon}
\norm{E^\eta_\fx\sum_{\hat\theta\subset\theta}(-1)^{\sharp{\hat\theta}} \mathcal{R}^{\hat\theta}_\fx(F)(\psi^{\flambda}_\fz)}_m \lesssim \norm{F}_{G^{\falpha,\fgamma,\fdelta}L_m}\flambda^{\fgamma}_{\theta}\flambda^\fdelta_\eta\flambda^{\falpha}_{\theta^c}\,,
    \end{equation}
  uniformly over $\fz\in [0,T]^d$,  $\flambda\in(0,1]^d$ such that $[\fz-\flambda_{\eta^c},\fz+\flambda]\subset[0,T]^d$ and $\fx\in[0,T]^d$ with $\fx_\theta = \fz_\theta$.
\end{enumerate}
We call $\mathcal{R}^{[d]}_{\fx}(F)$ the multiparameter  reconstruction of the stochastic germ $F$ which we denote by $\mathcal{R} (F)$.
\end{theorem}

\begin{remark}
    Note that we switch to test functions with $rd$ derivatives, because we want to use  Lemma \ref{lem:inequalityPhi} in the proof of the reconstruction theorem.
\end{remark}

\begin{remark}
While the coherence property works for test functions with support in $[0,1]^\theta\times[-1,1]^{\theta^c}$, we have to reduce this set by some fixed distance in \eqref{cond4Recon}. This is because for any test function $\psi$ with support in some compact set $K$, its wavelet approximation
    \[
        \sum_{\fx\in\Delta_\fn} \phi^\fn_\fx\scalar{\phi^\fn_\fx,\psi}
    \]
    has support in the $\epsilon_\fn$-fattening of $K$ given by
    \[
        \{\fy\in\bR^d\colon \text{dist}(\fy,K)\le\epsilon_\fn\}\,.
    \]
    for some $\epsilon_\fn$ which approaches $0$ as $\fn\to\infty$. Thus, by choosing $\psi$ with a slightly smaller support, we can apply the coherence property to its wavelet approximation for sufficiently large $\fn$, see also Remark \ref{rk:eps_rmk}.
\end{remark}

\subsection{Proof of the stochastic reconstruction theorem}

We now pass to the proof of Theorem \ref{theo:reconstruction}. To achieve that we recall a preliminary result whose proof can be found in \cite[Lemma 19]{multiparameterStochSewing}.

\begin{lemma}\label{lem:technical2}
Let $(h_{\fn})_{\fn\in\bN^{d}}$ be a multiparameter sequence in a Banach space $\mathcal H$. We define the operators $I^ih_{\fn} = h_{\fn+\fone_{\{i\}}}$. Assume that there is a constant $C>0$ and $\falpha = (\alpha_1,\dots,\alpha_d)>\foo$ such that for all $\fn\in\bN^d$ and $\theta\subset[d]$
\begin{equation*}
    \norm{\prod_{i\in\theta} (\Id- I^i)h_{\fn_\theta}} \le C2^{-\falpha*\fn}_{\theta}\,.
\end{equation*}
Then for all $\fn\le\fk$ and $\theta\subset[d]$ we have
\begin{equation*}
    \norm{h_{\fn_\theta}-h_{\fk_\theta}}\lesssim C\sum_{i\in\theta}2^{-n_i\alpha_i}\,.
\end{equation*}
It especially follows that $(h_{\fn})$ is a Cauchy sequence.
\end{lemma}
We fix wavelets $\phi,\hat\phi^\chi, \chi\subset[d]$ as described in Section \ref{sec:wavelet}. Knowing the above result, to construct the family of distributions $\mathcal{R}_\fx^\eta(F)$ in Theorem \ref{theo:reconstruction}, it is natural to consider the following multiparameter approximations
\[
    \mathcal{R}^{\eta,\fn}_\fx(\psi) := \sum_{\fy\in\Delta_\fn} F_{\pi^\eta_\fy \fx}(\phi^\fn_\fy)\scalar{\phi^\fn_\fy,\psi}\,.
\]
(we will drop the dependency on $F$ for the sake of clarity) and to show proper bounds of the quantity 
\begin{equation}\label{eq:fN}
    \norm{\prod_{i\in\theta}(\Id- I^i)\mathcal{R}^{\eta,\fn}_\fx(\psi)}_m\,,
\end{equation}
where $I^i\mathcal{R}^{\eta,\fn}_\fx = \cR^{\eta,\fn+\fone_{\{i\}}}_\fx$. Once we have done so, we can conclude that $\mathcal{R}^{\eta,\fn}_\fx$ converges to some limit. Before we show the inequality, let us introduce some more notation: For any $\kappa\subset[d]$, we set
\begin{equation*}
    g^{\kappa,\fn}_\fx(\psi) := \sum_{\eta\subset\kappa}(-1)^{\sharp\eta}\mathcal{R}^{\eta,\fn}_\fx(\psi) = (-1)^{\sharp\kappa}\sum_{\fy\in\Delta_\fn} \square^\kappa_{\fx,\fy} F(\phi_\fy^\fn)\scalar{\phi_\fy^\fn,\psi}\,,
\end{equation*}
and set $I^ig^{\kappa,\fn}_\fx := g^{\kappa,\fn+\fone_{\{i\}}}_\fx$, as before. Further, recall that the one-dimensional wavelets $\varphi,\hat\varphi$ fulfill the property
\[
    \varphi_z^n = \sum_{k\in\Delta_{n+1}}a_k^n \varphi^{n+1}_{z+k}\,.
\]
Comparing this to \eqref{eq:selfReplicable} one has $a_k^n = a_{2^n k}$. We similarly introduce the notation $b^n_k := \scalar{\varphi_{z+k}^{n+1},\hat\varphi_z^n}$. Extending the properties of $1$ dimensional wavelets to  $d$-dimensional wavelets one has 
\begin{equation}\label{eq:scalarProductWavelets}
\scalar{\phi^{\fn+\fone_\chi+\fone_\theta}_{\fz+\fk} , \hat\phi^{\chi,\fn}_\fz} = a_{\fk,\theta}^{\fn}b_{\fk,\chi}^{\fn}\delta_{\foo,\fk}^{(\chi\cup\theta)^c}\,, 
\end{equation}
for all $\chi,\theta\subset[d]$ such that $\chi\cap \theta=\emptyset$, $ \fz\in\Delta_\fn, \fk\in\Delta_{\fn+\fone_\chi+\fone_\theta}$, where we use the notations
\[a_{\fk,\theta}^{\fn} := \prod_{i\in\theta} a_{k_i}^{n_i}\,, \quad b_{\fk,\chi}^{\fn} := \prod_{i\in\chi}b_{k_i}^{n_i}\,, \quad  \delta_{\foo,\fk}^{(\chi\cup\theta)^c} = \prod_{i\in(\chi\cup\theta)^c}\delta_{0,k_i}\,,\] where $\delta_{0,k_i} = 1$ if $k_i = 0$ and $0$, otherwise. Instead of \eqref{eq:fN}, we will instead show a bound on
\begin{equation*}
    \norm{\bE^\eta_\fx \prod_{i\in\theta}(\Id-I^i)g^{\kappa,\fn}_\fx(\psi_\fx^\flambda)}_m\,.
\end{equation*}
which has the advantage that it can also be used to show \eqref{cond4Recon}. To get there, we split the underlying test function $\psi $ according to the functions  $\hat{\phi}^{\chi}$
\begin{equation}\label{eq:splitXi}
\prod_{i\in\theta}(\Id-I^i)g^{\kappa,\fn}_\fx(\psi) = \sum_{\chi\subset\theta}\prod_{i\in\theta}(\Id-I^i)g^{\kappa,\fn}_\fx(\hat P^{\chi,\fn}\psi)\,,
\end{equation}
where we used that $g^{\kappa, \fn}_\fx$ belongs to $V^{\fn}$. A first  lemma bounds the right-hand side of \eqref{eq:splitXi}.

\begin{lemma}\label{lem:technicalBoundg}
    Under the assumptions of the reconstruction theorem, it holds that for all $\fn\in\bN^d$, $\chi\subset\theta\subset[d]$, $\fx, \fz\in\Delta_{\fn}\cap[0,T]^d$, such that $\fx_{\kappa\setminus(\theta\setminus\chi)}\le \fz_{\kappa\setminus(\theta\setminus\chi)}$  one has
    \begin{align*}
\norm{\bE^\eta_{\pi_\fz^{\theta\setminus\chi}\fx}\prod_{i\in\theta}(\Id-I^i)g_\fx^{\kappa,\fn}(\hat\phi^{\chi,\fn}_\fz)}_m\lesssim \norm{F}_{G^{\falpha,\fgamma,\fdelta}}2^{-\fn*(\frac 12\fone+\falpha)}&(\pi_0^{\theta\setminus\chi}\compAbs{\fx-\fz}+2^{-\fn})^{(\fgamma-\falpha)}_{\kappa}\\ \times&(\pi_0^{\theta\setminus\chi}\compAbs{\fx-\fz}+2^{-\fn})^{\fdelta}_{\eta}\,,
    \end{align*}
 where we assume $\supp(\hat\phi_\fz^{\chi,\fn})\subset[0,T]^d$ and  $\eta\subset\kappa\subset[d]$.
\end{lemma}

\begin{proof}
By definition of  $ g^{\kappa,\fn}_\fx(\psi) $ we have the identity
\begin{equation}\label{eq:first_identity}
\prod_{i\in\theta}(\Id-I^i)g_\fx^{\kappa,\fn}(\hat\phi^{\chi,\fn}_\fz)= \sum_{\eta\subset\kappa}\prod_{i\in\theta}(\Id-I^i)(-1)^{\sharp\eta}\mathcal{R}^{\eta,\fn}_\fx(\hat\phi^{\chi,\fn}_\fz)\,.
\end{equation}
By construction of our wavelets, we have $\scalar{\phi^\fm_\fx,\hat\phi_\fy^{\chi,\fn}} = 0$ if there exists an $i\in\chi$ such that $m_i\le n_i$, especially implying $\mathcal R_\fx^{\kappa,\fm}(\hat\phi_\fy^{\chi,\fn}) = 0$ for such $\chi,\fn,\fm$. Multiplying out the product over $i\in\theta\cap\chi$ and eliminating all zero terms therefore gives us the identity
\begin{equation*}
\prod_{i\in\theta}(\Id- I^i)\mathcal{R}_\fx^{\kappa,\fn}(\hat\phi_\fz^{\chi,\fn}) = (-1)^{\sharp\chi}\prod_{i\in\theta\setminus\chi}(\Id- I^i)\mathcal{R}^{\kappa,\fn+\fone_\chi}_\fx(\hat\phi^{\chi,\fn}_\fz)\,.
\end{equation*}
Note that in the one-dimensional case, one has
\[\sum_{y\in\Delta_n} F_x(\varphi^n_y)\scalar{\varphi^n_y,\varphi^n_z} = F_x(\varphi^n_z) = \sum_{y\in\Delta_{n+1}} F_x(\varphi^{n+1}_y)\scalar{\varphi^{n+1}_y,\varphi^n_z}\,,\] 
implying in $d\ge 1$ that for all $i\in\kappa^c$, we have
\begin{equation*}
(\Id- I^i)\mathcal{R}^{\kappa,\fn+\fone_\chi}_\fx(\hat\phi_\fz^{\chi,\fn}) = 0\,.
\end{equation*}
Thus, we conclude that
\begin{equation}\label{eq:ThetaXiSubsetKappa}
        \prod_{i\in\theta\setminus\chi}(\Id- I^i)\mathcal{R}^{\kappa,\fn+\fone_\chi}_\fx(\hat\phi_\fz^{\chi,\fn}) = 0
    \end{equation}
unless $\theta\setminus\chi\subset\kappa$. Therefore we can rewrite the right-hand side of \eqref{eq:first_identity} into     
\begin{equation}\label{eq:sumOverTildeKappa}
\prod_{i\in\theta}(\Id-I^i)g_\fx^{\kappa,\fn}(\hat\phi^{\chi,\fn}_\fz) = \sum_{\theta\setminus\chi\subset\tilde\kappa\subset\kappa}~\prod_{i\in\theta\setminus\chi}(\Id- I^i)(-1)^{\sharp{\tilde\kappa}+\sharp\chi}\mathcal{R}_\fx^{\tilde\kappa,\fn+\fone_\chi}(\hat\phi_\fz^{\chi,\fn})\,.
\end{equation}

We further analyze the product on the right-hand side by expanding $\cR^{\tilde\kappa,\fn+\fone_\chi}_\fx$.
\begin{align*}
\prod_{i\in\theta\setminus\chi} (\Id-I^i)\mathcal{R}_\fx^{\tilde\kappa,\fn+\fone_\chi}(\hat\phi_\fz^{\chi,\fn}) &= \sum_{\tilde\theta\subset\theta\setminus\chi}(-1)^{\sharp{\tilde\theta}} \mathcal{R}^{\tilde\kappa,\fn+\fone_{\tilde\theta}+\fone_\chi}_\fx(\hat\phi_\fz^{\fn,\chi})\\
&= \sum_{\tilde\theta\subset\theta\setminus\chi}\sum_{\fy\in\Delta_{\fn+\fone_\chi+\fone_{\tilde\theta}}} (-1)^{\sharp\tilde\theta}F_{\pi^{\tilde\kappa}_{\fz+\fy}\fx}(\phi_{\fz+\fy}^{\fn+\fone_\chi+\fone_{\tilde\theta}})\scalar{\phi^{\fn+\fone_\chi+\fone_{\tilde\theta}}_{\fz+\fy},\hat\phi_\fz^{\fn,\chi}}\,.
    \end{align*}
    Using the identity \eqref{eq:scalarProductWavelets} together with the expansion 
    \[\phi_{\fz+\fy}^{\fn+\fone_\chi+\fone_{\tilde\theta}} = \sum_{\fk\in\Delta_{\fn+\fone_\theta}^{\theta\setminus(\chi\cup \tilde\theta)}} a^{\fn}_{\fk,\theta\setminus(\chi\cup\tilde\theta)}\phi_{\fz+\fy+\fk}^{\fn+\fone_\theta}\,,\] which itself follows from \eqref{eq:selfReplicable}, we get
    \begin{equation*}
        \prod_{i\in\theta\setminus\chi} (\Id-I^i)\mathcal{R}_\fx^{\tilde\kappa,\fn+\fone_\chi}(\hat\phi_\fz^{\chi,\fn}) = \sum_{\tilde\theta\subset\theta\setminus\chi} \sum_{\fy\in\Delta_{\fn+\fone_\theta}^{\chi\cup\tilde\theta}}\sum_{\fk\in\Delta_{\fn+\fone_\theta}^{\theta\setminus(\chi\cup\tilde\theta)}} (-1)^{\sharp\tilde\theta}a_{\fy+\fk, \theta\setminus\chi}^{\fn} b_{\fy,\chi}^{\fn}F_{\pi^{\tilde\kappa}_{\fz+\fy}\fx}(\phi_{\fz+\fy+\fk}^{\fn+\fone_\theta})\,.
    \end{equation*}
    Let $\fu := \pi_{\fy+\fk}^{\theta\setminus\chi} \foo, \fv := \pi_{\fy}^{\chi} \foo$, such that $\fy+\fk = \fu+\fv$ holds. Using these notations allows us to pull  the sum over $\tilde\theta\subset\theta\setminus\chi$ into the other two sums, which leads to
    \begin{align*}
        \prod_{i\in\theta\setminus\chi} (\Id-I^i)\mathcal{R}_\fx^{\tilde\kappa,\fn+\fone_\chi}(\hat\phi_\fz^{\chi,\fn}) &= \sum_{\fu\in\Delta_{\fn+\fone_\theta}^{\theta\setminus\chi}}\sum_{\fv\in\Delta_{\fn+\fone_\theta}^\chi} a_{\fu, \theta\setminus\chi}^{\fn} b_{\fv,\chi}^{\fn}\sum_{\tilde\theta\subset\theta\setminus\chi} (-1)^{\sharp\tilde\theta} F_{\pi^{\tilde\kappa}_{\pi^{\tilde\theta}_{\fz+\fu+\fv}(\fz+\fv)} \fx}(\phi_{\fz+\fu+\fv}^{\fn+\fone_\theta})\\
        &= \sum_{\fu\in\Delta_{\fn+\fone_\theta}^{\theta\setminus\chi}}\sum_{\fv\in\Delta_{\fn+\fone_\theta}^\chi} a_{\fu, \theta\setminus\chi}^{\fn} b_{\fv,\chi}^{\fn}(-1)^{\sharp(\theta\setminus\chi)}\square_{\fz+\fv,\fz+\fv+\fu}^{\theta\setminus\chi} F_{\pi^{\tilde\kappa}_{\cdot} \fx}(\phi_{\fz+\fu+\fv}^{\fn+\fone_\theta})\\
        &= \sum_{\fu\in\Delta_{\fn+\fone_\theta}^{\theta\setminus\chi}}\sum_{\fv\in\Delta_{\fn+\fone_\theta}^\chi} a_{\fu, \theta\setminus\chi}^{\fn} b_{\fv,\chi}^{\fn} (-1)^{\sharp(\theta\setminus\chi)}\square_{\pi^{\tilde\kappa}_{\fz+\fv}\fx,\fz+\fv+\fu}^{\theta\setminus\chi} F(\phi_{\fz+\fu+\fv}^{\fn+\fone_\theta})\\
    \end{align*}
    where we use the fact that $\theta\setminus\chi\subset\tilde\kappa$ holds. By summing  the last part over $(-1)^{\sharp\tilde\kappa}$ one has 
    \begin{align*}
\sum_{\theta\setminus\chi\subset\tilde\kappa\subset\kappa}(-1)^{\sharp\tilde\kappa} \square_{\pi^{\tilde\kappa}_{\fz+\fv}\fx,\fz+\fv+\fu}^{\theta\setminus\chi} F &= \sum_{\hat\kappa\subset\kappa\setminus(\theta\setminus\chi)}(-1)^{\sharp\hat\kappa+\sharp(\theta\setminus\chi)} \square^{\theta\setminus\chi}_{\pi^{\hat\kappa}_{\fz+\fv} \pi_{\fz}^{\theta\setminus\chi} \fx,\fz+\fv+\fu}F\\
&= (-1)^{\sharp \kappa} \square^{\kappa\setminus(\theta\setminus\chi)}_{\pi^{\theta\setminus\chi}_{\fz}\fx,\fz+\fv+\fu}\square^{\theta\setminus\chi}_{\cdot,\fz+\fu+\fv} F\,,
\end{align*}
where we use the fact that $u_i = 0$ for all $i\in \kappa\setminus(\theta\setminus\chi)$ and $v_i = 0$ for all $i\in\theta\setminus\chi$. Using Lemma \ref{lem:splitSquareIncrement}, it follows that
 \begin{equation*}
\sum_{\theta\setminus\chi\subset\tilde\kappa\subset\kappa}(-1)^{\sharp\tilde\kappa} \square_{\pi^{\tilde\kappa}_{\fz+\fv}\fx,\fz+\fv+\fu}^{\theta\setminus\chi} F = (-1)^{\sharp\kappa} \square^\kappa_{\pi^{\theta\setminus\chi}_{\fz}\fx,\fz+\fv+\fu} F.
    \end{equation*}
Putting this into \eqref{eq:sumOverTildeKappa} one obtains
\begin{equation*}
 \prod_{i\in\theta}(\Id-I^i)g_\fx^{\kappa,\fn}(\hat\phi^{\chi,\fn}_\fz) = (-1)^{\sharp\theta+\sharp\kappa} \sum_{\fu\in\Delta_{\fn+\fone_\theta}^{\theta\setminus\chi}}\sum_{\fv\in\Delta_{\fn+\fone_\theta}^\chi} a_{\fu, \theta\setminus\chi}^{\fn} b_{\fv,\chi}^{\fn}\square^\kappa_{\pi^{\theta\setminus\chi}_{\fz}\fx,\fz+\fv+\fu} F(\phi^{\fn+\fone_\theta}_{\fz+\fu+\fv})\,.
    \end{equation*}
If $\fx_{\kappa\setminus(\theta\setminus\chi)}\le \fz_{\kappa\setminus(\theta\setminus\chi)}$ holds, we get by the coherence property for $\eta\subset\kappa$:
    \begin{align*}
\norm{\bE^\eta_{\pi^{\theta\setminus\chi}_\fz\fx}\square^\kappa_{\pi^{\theta\setminus\chi}_{\fz}\fx,\fz+\fv+\fu} F(\phi^{\fn+\mathbf{1}_\theta}_{\fz+\fu+\fv})}_m \lesssim \norm{F}_{G^{\falpha,\fgamma,\fdelta}} 2^{-\fn*\falpha-\frac\fn2}&(\pi_{\foo}^{\theta\setminus\chi}\compAbs{\fx-\fz}+2^{-\fn})^{\fgamma-\falpha}_{\kappa}\\ \times &(\pi_{\foo}^{\theta\setminus\chi}\compAbs{\fx-\fz}+2^{-\fn})^{\fdelta}_{\eta}\,,
\end{align*}
    where we use the property that $a_{\fu,\theta\setminus\chi}^{\fn} b_{\fv,\chi}^{\fn} \neq 0$ implies that $\abs{u_i+v_i}\lesssim 2^{-n_i}$ for all  $i\in\theta$. The claim then follows from the fact that the sum
    \[\sum_{\fu\in\Delta_{\fn+\fone_\theta}^{\theta\setminus\chi}}\sum_{\fv\in\Delta_{\fn+\fone_\theta}^{\chi}}a_{\fu,\theta\setminus\chi}^{\fn} b_{\fv,\chi}^{\fn} = \sum_{\fu\in\Delta_{\theta}^{\theta\setminus\chi}}\sum_{\fv\in\Delta_{\theta}^{\chi}}a_{\fu,\theta\setminus\chi}^{\foo} b_{\fv,\chi}^{\foo}\]
    is independent of $\fn$. 
\end{proof}
The next step is to test $\prod_{i\in\theta}(\Id-I^i)g^{\kappa,\fn}_\fx$ against a general test function $\psi_\fz^\flambda$. To compare $\flambda$ with the dyadic scale we will write $\flambda\approx 2^{-\fn}$ if there exists a constant $\fc>\foo$, such that $2^{-n_i-1} \le c_i \lambda_i\le 2^{-n_i}$ for all $i\in[d]$. Whenever such a constant $\fc>\foo$ is chosen, we allow all constants in $\lesssim$ to depend on $\fc$. 
\begin{lemma}\label{lem:technical3}
    Let $\foo<\flambda<\fone$, $\eta\subset\kappa\subset[d]$ and $\theta\subset[d]$. Choose $\fn$ such that $\flambda\approx 2^{\fn}$ and $\fz,\fx\in [0,T]^d$, such that $z_i = x_i$ for all $i\in\kappa$. Finally, choose a test function $\psi\in C^{rd}_c([1/4,3/4]^\kappa\times[-3/4,3/4]^{\kappa^c})$, and assume that $[\fz, \fz+\flambda]^\kappa\times[\fz-\flambda,\fz+\flambda]^{\kappa^c}\subset[0,T]^d$. Then for any $\fk\ge\foo$ one has 
    \begin{equation*}
        \norm{\bE_\fx^\eta\prod_{i\in\theta}(\Id-I^i)g_\fx^{\kappa,\fn+\fk}(\psi_\fz^{\flambda})}_m \lesssim \norm{F}_{G^{\falpha,\fgamma,\fdelta}} \norm{\psi}_{C_c^{rd}} \flambda^{\falpha}_{\kappa^c}\flambda^{\fgamma}_{\kappa}\flambda^{\fdelta}_{\eta} \cdot 2^{-\fk*\falpha}_{\theta^c}2^{-\fk*\fxi}_{\theta}\,,
    \end{equation*}
    where $\fxi = (\falpha+r\fone)\wedge(\fgamma+\fdelta)\wedge(\fgamma+\frac 12\fone)>\foo$.
\end{lemma}

\begin{proof}
  Using again the expressions \eqref{eq:splitXi}, and \eqref{eq:ThetaXiSubsetKappa} we have
    \begin{align*}
        \prod_{i\in\theta} (\Id-I^i)g_\fx^{\kappa,\fn+\fk}(\psi_\fz^\flambda) &= \sum_{\substack{\chi\subset\theta \\ \theta\setminus\chi\subset\kappa}} \prod_{i\in\theta}(\Id-I^i)g_{\fx}^{\kappa,\fn+\fk}(\hat P_{\fn+\fk}^\chi\psi_\fz^\flambda) \\
        &= \sum_{\substack{\chi\subset\theta \\ \theta\setminus\chi\subset\kappa}}\sum_{\fy\in\Delta_{\fn+\fk}} \prod_{i\in\theta}(\Id-I^i)g_{\fx}^{\kappa,\fn+\fk}(\hat\phi_\fy^{\chi,\fn+\fk})\scalar{\hat\phi^{\chi,\fn+\fk}_\fy,\psi_\fz^\flambda}\,.
    \end{align*}
    Let us fix an indexset $\chi$ and set $\tilde\theta := \theta\setminus\chi$, $\tilde\eta := \eta\cap\tilde\theta$. Note that the scalar product is only non-zero if $\abs{y_i-z_i}\lesssim \lambda_i$ for all $i\in[d]$, which holds true for at most $\approx 2^{k_i}$ many $\fy\in\Delta_{\fn+\fk}^{\{i\}}$ for each $i\in[d]$. For each $\fy\in\Delta_{\fn+\fk}$, we we write  $\fy= \fy_{\tilde\theta} +\fy_{\tilde\theta^c}$ and we further split the sum over $\fy\in\Delta_{\fn+\fk}$ to obtain 
    \[\sum_{\substack{\chi\subset\theta \\ \theta\setminus\chi\subset\kappa}}\sum_{\fy_{\tilde\theta}\in\Delta_{\fn+\fk}^{\tilde\theta}}\sum_{\fy_{\tilde\theta^c}\in\Delta_{\fn+\fk}^{\tilde\theta^c}} \prod_{i\in\theta}(\Id-I^i)g_{\fx}^{\kappa,\fn+\fk}(\hat\phi_\fy^{\chi,\fn+\fk})\scalar{\hat\phi^{\chi,\fn+\fk}_\fy,\psi_\fz^\flambda}  \,\] 
    Let us assume for the moment that we can apply the  BDG-type inequality \eqref{ineq:lem3Result} to  the quantity
    \[h(\fy_{\tilde\theta}) := \sum_{\fy_{\tilde\theta^c}\in\Delta^{\tilde\theta^c}_{\fn+\fk}} \prod_{i\in\theta}(\Id-I^i)g_\fx^{\kappa,\fn+\fk}(\hat\phi^{\chi,\fn+\fk}_\fy)\scalar{\hat\phi^{\chi,\fn+\fk}_\fy,\psi_\fz^{\flambda}}\,,
    \]
    to estimate
    \[
        \norm{\sum_{\fy_{\tilde\theta} \in\Delta_{\fn+\fk}^{\tilde\theta}} h(\fy_{\tilde\theta})}_m\,.
    \]
    Combining Lemma \ref{lem:technicalBoundg} and Lemma \ref{lem:inequalityPhi} one has
    \begin{align*}
        &\norm{\bE_{\pi^{\tilde\theta}_\fy \fx}^\eta  h(\fy_{\tilde\theta}) }_m\lesssim\\&\lesssim \norm{F}_{G^{\falpha,\fgamma,\fdelta}}\norm{\psi}_{C_c^{rd}}2^{\fk}_{{\tilde\theta^c}}2^{-(\fn+\fk)*(\fone+\falpha+r\fone_\chi)}\flambda^{-(\fone+r\fone_\chi)}(\pi_\foo^{\tilde\theta}\flambda+2^{-(\fn+\fk)})^{(\fgamma-\falpha)}_{\kappa} (\pi_\foo^{\tilde\theta}\flambda+2^{-(\fn+\fk)})^{\fdelta}_{\eta}\\
        &\lesssim \norm{F}_{G^{\falpha,\fgamma,\fdelta}}\norm{\psi}_{C_c^{rd}}\flambda^{\falpha}_{{\kappa^c}}\flambda^{\fgamma}_{\kappa}\flambda^{\fdelta}_{\eta}\cdot 2^{-\fk*(\fone_{\tilde\theta}+\falpha_{\tilde\theta^c}+r\fone_\chi+\fgamma_{\tilde\theta}+\fdelta_{\tilde\eta})}\,,
        \end{align*}
    where we used that the support of $\psi$ gives us that $\fy_\kappa\ge\fx_\kappa$, as long as one chooses $\fc$  small enough in $\flambda\approx 2^{-\fn}$. Starting from this estimate, we use \eqref{ineq:lem3Result} and get
    \begin{align*}
        &\norm{\bE^\eta_\fx\prod_{i\in\theta}(\Id-I^i)g_\fx^{\kappa,\fn+\fk}(\psi_\fz^{\flambda})}_m = \norm{\sum_{\substack{\chi\subset\theta \\ \theta\setminus\chi\subset\kappa}} \bE^\eta_\fx\sum_{\fy_{\tilde\theta} \in\Delta_{\fn+\fk}^{\tilde\theta}} h(\fy_{\tilde\theta})}_m\\
        &\qquad\lesssim \norm{F}_{G^{\falpha,\fgamma,\fdelta}}\norm{\psi}_{C_c^{rd}}\flambda^{\falpha}_{{\kappa^c}}\flambda^{\fgamma}_{\kappa}\flambda^{\fdelta}_{\eta}\cdot 2^{-\fk*\falpha,}_{\tilde\theta^c}2^{-r\fk}_\chi\\&\qquad\qquad\sum_{\tilde\theta\setminus\tilde\eta = \theta_1\sqcup\theta_2}\underbrace{\left(\sum_{\fy_1\in\Delta_{\fn+\fk}^{\theta_1\cup\tilde \eta}}2^{-\fk*(\fone+\fgamma+\fdelta)}_{{\theta_1\cup\tilde\eta}}\right)}_{\lesssim 2^{-\fk*(\fgamma+\fdelta)}_{{\theta_1\cup\tilde\eta}}}\underbrace{\left(\sum_{\fy_2\in\Delta_{\fn+\fk}^{\theta_2}}2^{-2\fk*(\fone+\fgamma)}_{{\theta_2}}\right)^\frac 12}_{\lesssim 2^{-\fk*(\fgamma+\frac 12\fone),}_{{\theta_2}}}\\
        &\qquad \lesssim \norm{F}_{G^{\falpha,\fgamma,\fdelta}}\norm{\psi}_{C_c^{rd}}\flambda^{\falpha}_{{\kappa^c}}\flambda^{\fgamma}_{\kappa}\flambda^{\fdelta}_{\eta}\cdot 2^{-\fk*(\falpha_{\theta^c}+(\falpha+r\fone)_\chi+((\fgamma+\fdelta)\wedge(\fgamma+\frac 12\fone))_{\tilde\theta})}\\
        &\qquad\lesssim \norm{F}_{G^{\falpha,\fgamma,\fdelta}}\norm{\psi}_{C_c^{rd}}\flambda^{\falpha}_{\kappa^c}\flambda^{\fgamma}_{\kappa}\flambda^{\fdelta}_{\eta} 2^{-\fk*\falpha}_{{\theta^c}}2^{-\fk*\fxi}_{\theta}\,,
    \end{align*}
    where we implicitly do not sum over all $\fy$, but only the non-zero terms. This result is conditional on the possibility of applying Lemma \ref{lem:multiparameterBDG}. Unfortunately the step size of $\Delta_{\fn+\fk}$ is given by $\fu=(2^{-n_1-k_1},\dots,2^{-n_d-k_d})$, so we would need $h(\fy_{\tilde\theta})$ to be $\cF^{\tilde\theta}_{\fy+\fu}$-measurable to apply Lemma \ref{lem:multiparameterBDG}. But since the support of $\hat\phi^\chi$ is in $[C,R]^d$ for some $0<C<R$ instead of $[0,1]^d$, $h(\fy_{\tilde\theta})$ is $\cF^{\tilde\theta}_{\fy+2R\cdot\fu}$-measurable. This can easily be solved by splitting the sum over $\fy_{\tilde\theta}\in\Delta^{\tilde\theta}_{\fn+\fk}$ into finitely many sums over coarser grids:
    \[
        \sum_{\fy_{\tilde\theta}\in\Delta_{\fn+\fk}^{\tilde\theta}} h(\fy_{\tilde\theta}) = \sum_{\fr}\sum_{\fy_{\tilde\theta}\in\Delta_{\fn+\fk}^{\tilde\theta}(\fr)} h(\fy_{\tilde\theta})\,,
    \]
    see \cite{kern2021}, pages 19-20 for the details. This procedure only adds a constant depending on $R$ and the dimension $d$. 
\end{proof}

 With all these estimates, we can now tackle the proof of the reconstruction theorem.

\begin{proof}[Proof of Theorem \ref{theo:reconstruction}]
Set $h_{\fk} := g^{\kappa,\fk}_\fx(\psi_{\fz}^{\flambda})$ for a $\psi\in C_c^{rd}((0,1)^\kappa\times(-1,1)^{\kappa^c})$ and a $\fz\in[0,T]^d$ such that $z_i = x_i$ for all $i\in\kappa$. Lemma \ref{lem:technical3} together with Lemma \ref{lem:technical2} immediately give us that $g^{\kappa,\fk}_\fx(\psi_\fx^{\flambda})$ is a Cauchy sequence in $L_m$. Thus, it is convergent. It follows by induction over $\sharp\kappa$, that all $\mathcal{R}^{\kappa,\fk}_\fx(\psi_\fz^\flambda)$ are convergent sequences in $L_m$. Since $\mathcal{R}^{\kappa,\fk}_\fx(\psi_\fz^\flambda)$ does not depend on $x_i, i\in\kappa$ by construction, we can drop the assumption on $\fz$ and find that $\mathcal{R}^{\kappa,\fk}_\fx(\psi_\fz^\flambda) = \mathcal{R}^{\kappa,\fk}_{\pi_\fz^\kappa\fx}(\psi_\fz^\flambda)$ is convergent for all $\fz\in[0,T]^d$ such that $\supp(\psi_\fz^\flambda)\subset (0,T)^d$. It follows that $\mathcal{R}^{\kappa,\fk}_\fx(\psi)$ is a convergent series for all $\kappa\subset[d]$, $ \fx\in[0,T]^d$ and $\psi$ with support in $(0,T)^d$.

    We need to show that the limits $\mathcal{R}^\eta_\fx(F)$ fulfill the properties in the statement. For this, let $\foo<\flambda<\fone$ and $\fn\in\bN^d$ such that $2^{\fn} \approx \flambda$. Observe that again by Lemma \ref{lem:technical2}, this time using $h_{\fk} = \bE^\eta_\fx g^{\kappa,\fn+\fk}_\fx(\psi_\fz^{\flambda})$:

    \begin{equation*}
\norm{\bE^{\eta}_\fx(g^{\kappa,\fn}_\fx(\psi_\fz^\flambda)- g^{\kappa,\fn+\fk}_\fx(\psi_\fz^\flambda))}_m\lesssim\norm{F}_{G^{\falpha,\fgamma,\fdelta}}\norm{\psi}_{C_c^{rd}}\flambda^{\falpha}_{{\kappa^c}}\flambda^{\fgamma}_{\kappa}\flambda^{\fdelta}_{\eta}\,.
    \end{equation*}
    So it suffices to show that
    \begin{equation}\label{ineq:gkappan}
\norm{\bE^{\eta}_\fx(g^{\kappa,\fn}_\fx(\psi_\fz^\lambda))}_m\lesssim\norm{F}_{G^{\falpha,\fgamma,\fdelta}}\norm{\psi}_{C_c^{rd}}\flambda^{\falpha}_{{\kappa^c}}\flambda^{\fgamma}_{\kappa}\flambda^{\fdelta}_{\eta}
    \end{equation}
to show \eqref{cond4Recon} by letting $\fk\to\infty$. To see \eqref{ineq:gkappan}, we calculate
    \begin{align*}
        g^{\kappa,\fn}_\fx(\psi_\fz^\flambda) &= \sum_{\tilde\kappa\subset\kappa}(-1)^{\sharp{\tilde\kappa}}\mathcal{R}^{\tilde\kappa,\fn}_\fx(\psi_\fz^\flambda) \\
        &= \sum_{\fy\in\Delta_{\fn}}\sum_{\tilde\kappa\subset\kappa} (-1)^{\sharp{\tilde\kappa}}F_{\pi_\fy^{\tilde\kappa} \fx}(\phi_\fy^{\fn})\scalar{\phi_\fy^{\fn},\psi_\fz^{\flambda}}\\
        &= \sum_{\fy\in\Delta_{\fn}}\square^{\kappa}_{\fx,\fy}F(\phi_\fy^{\fn})\scalar{\phi_\fy^{\fn},\psi_\fz^{\flambda}}\,.
    \end{align*}
    Using coherence and  Lemma \ref{lem:inequalityPhi}, we get that for all $\eta\subset\kappa$:
    \begin{equation*}
\norm{\bE^{\eta}_\fx(g^{\kappa,\fn}_\fx(\psi_\fz^\flambda))}_m\lesssim \norm{F}_{G^{\falpha,\fgamma,\fdelta}}\norm{\psi}_{C_c^{rd}}\flambda^{\falpha}_{{\kappa^c}}\flambda^{\fgamma}_{\kappa}\flambda^{\fdelta}_{\eta}\,.
    \end{equation*}
where we use that the order of non-zero terms is independent of $\fn$ and that $\abs{x_i-y_i}\lesssim \lambda_i$ for $i\in[d]$ and non-zero terms. It follows that for all $\fk\ge \foo$,
\begin{equation*}
\norm{\bE^{\eta}_\fx(g^{\kappa,\fn+\fk}_\fx(\psi_\fz^\flambda))}_m\lesssim \norm{F}_{G^{\falpha,\fgamma,\fdelta}}\norm{\psi}_{C_c^{rd}}\flambda^{\falpha}_{{\kappa^c}}\flambda^{\fgamma}_{\kappa}\flambda^{\fdelta}_{\eta}\,,
\end{equation*}
which shows \eqref{cond4Recon}. We further note that all $\mathcal{R}^\kappa_\fx(F)$ are linear maps from $C_c^{rd}\to\bR$ and by the above inequality bounded in $\norm{\psi}_{C_c^{rd}}$, so they are random distributions. By construction, $\mathcal{R}^\kappa_\fx(F)$ is adapted and $\mathcal{R}_\fx^\kappa(F) = \mathcal{R}_{\pi_\fx^\kappa \foo}(F)$. We finally show that
\begin{equation*}
        F_\fx(\psi) = \lim_{\fn\to\infty}\mathcal{R}^{\emptyset,\fn}_\fx(\psi) = \lim_{n\to\infty} F_\fx(P_{\fn}\psi)\,.
    \end{equation*}
We already know that $\mathcal{R}^{\emptyset, \fn}_\fx$ is a Cauchy-sequence and by \cite[Lemma 62]{kern2021}, it follows that $F_\fx = \lim_{n\to\infty} F_\fx(P_{n\fone}\psi)$, so the limit needs to be $F_\fx$. This finishes the proof of convergence and to show the  properties in Theorem \ref{theo:reconstruction}.

It remains to show uniqueness. To do so, assume that $Z^\eta_\fx$ has $Z^\emptyset_\fx = 0$ and fulfills properties 2-4. (Think $Z = \cR-\tilde \cR$ for two families $\cR,\tilde \cR$ fulfilling 1-4). We show by induction over $\sharp\eta$, that $Z^\eta_\fx = 0$. Fix a $\theta\subset[d]$ and a $\psi_\fz^\flambda$ with $\supp(\psi)\subset{[1/4,3/4]^\theta\times[-3/4,3/4]^{\theta^c}}$ and $\fz,\flambda$ such that $[\fz-\flambda_{\theta^c},\fz+\flambda]\subset[0,T]^d$. Since $Z^\theta_\fx$ does not depend on $\fx_\theta$, we can without loss of generality assume that $z_i = x_i$ for all $i\in\theta$. Using \eqref{cond4Recon} and the induction hypothesis, it follows that
 \begin{equation*}
        \norm{\bE^\eta_\fz Z^\theta_\fx(\psi_\fz^\flambda)}_m\lesssim \flambda^{\falpha}_{{\theta^c}}\flambda^{\fgamma}_{\theta}\flambda^{\fdelta}_{\eta}\,.
    \end{equation*}
    The important point here is again, that since $Z^\theta_\fx$ does not depend on $\fx_\theta$, the above inequality holds for all $\fx,\fz$, as long as the support of $\psi_\fz^\flambda$ lies in $(0,T)^d$.  We can especially apply this to our wavelets to find
    \begin{equation*}
        \norm{\bE^\eta_\fy Z_\fx^\theta(\hat\phi_\fy^{\chi,\fn})}_m \lesssim 2^{-\fn*(\frac12\fone+\falpha_{\theta^c}+\fgamma_\theta+\fdelta_\eta)}\,,
    \end{equation*}
    for all $\fy\in\Delta_\fn\cap[0,T)^d$, provided that $\fn$ is big enough for the expression to be well-defined. Let us now pick a test function $\psi$ with support in $(0,T)^d$ and a high enough $\fN\in\bN^d$, such that $Z_\fx^\theta(P_\fN\psi)$ is well-defined. It then follows for all $\fn\ge\fN$:
    \begin{align*}
        Z_\fx^\theta(P_\fn\psi) &= \sum_{\chi\subset\theta^c}\sum_{\fN_\chi\le \fm_\chi\le \fn_\chi}Z_\fx^\theta(\hat P^\chi_{\fn_\theta + m_\chi}\psi) \\
        &= \sum_{\chi\subset\theta^c}\sum_{\fN_\chi\le \fm_\chi\le \fn_\chi} \sum_{\fy\in\Delta_{\fn_\theta+\fm_\chi}}Z_\fx^\theta(\hat\phi^{\chi,\fn_\theta+\fm_\chi}_\fy)\scalar{\phi^{\chi,\fn_\theta+\fm_\chi}_\fy,\psi}\,.
    \end{align*}
    Note that at most $\lesssim 2^{-n_i}$ terms are non-zero for each direction $i\in\theta$ and $2^{-m_i}$ for $i\in\chi$. We apply Lemma \ref{lem:multiparameterBDG} to the sum over $\fy_\theta$ and get
    \begin{align*}   \norm{\sum_{\fy\in\Delta_{\fn_\theta+\fm_\chi}}Z_\fx^\theta(\hat\phi^{\chi,\fn_\theta+\fm_\chi}_\fy)\scalar{\phi^{\chi,\fn_\theta+\fm_\chi}_\fy,\psi}}_m &\lesssim \sum_{\theta = \theta_1\sqcup\theta_2} 2^{-(\fn_\theta+\fm_\chi)*(\falpha_{\theta^c} +\fgamma_\theta+\fdelta_{\theta_1} +\frac12\fone_{\theta_2} +r_\chi)} \\
        &\lesssim 2^{-\fn*(\fgamma+\fdelta)}_{\theta}2^{-\fm*(\falpha+\fr)}_{\chi}\,,
    \end{align*}
    where we used that $\delta_i\le 1/2$ for $i\in[d]$ holds. Summing over $\fm_\chi$, it gives us
    \begin{equation*}
        \norm{Z_\fx^\theta(P_\fn\psi)}_m \lesssim 2^{-n*(\fgamma+\fdelta)}_{\theta}\,,
    \end{equation*}
    which implies that $Z^\theta_\fx(\psi) = \lim_{n\to\infty} Z^\theta_\fx(P_{n\fone}\psi) = 0$, where we used \cite[Lemma 62]{kern2021}. 
\end{proof}
It remains to calculate the Hölder regularity of the reconstruction $\cR(F)$, which we claim to be in $C^{\falpha,\foo} L_m$. By the coherence property \eqref{ineq:coherence} with $\eta =\theta=\emptyset$, we have
\[
    \norm{\cR^\emptyset_\fx(F)(\psi_\fx^\flambda)}_m = F_\fx(\psi_\fx^\flambda)\le \norm{F}_{G^{\falpha,\fgamma,\fdelta}L_m}\flambda^\falpha\,,
\]
for all $\psi\in B^r([1/4,3/4]^d)$, $\fx\in[0,T]^d,\flambda\in(0,1]^d$ such that $[\fx,\fx+\flambda]\subset[0,T]^d$. Using \eqref{cond4Recon} together with the expansion
\begin{equation}\label{ineq:expansionfTheta}
\norm{\mathcal{R}^{\theta}_\fx(F)(\psi_\fx^\flambda)}_m \le \norm{\sum_{\hat\theta\subset\theta}(-1)^{\sharp{\hat\theta}} \mathcal{R}^{\hat\theta}_\fx(F)(\psi^{\flambda}_\fx)}_m + \norm{\sum_{\hat\theta\subsetneq\theta}(-1)^{\sharp{\hat\theta}} \mathcal{R}^{\hat\theta}_\fx(F)(\psi^{\flambda}_\fx)}_m
\end{equation}
inductively extends this argument to all $\cR^\theta_\fx(F), \theta\subset[d]$, as long as $\fgamma\ge\falpha$. That is, for all $\theta\subset[d]$, one sees that
\[
    \norm{\cR^\theta_\fx(F)(\psi_\fx^\flambda)}_m \lesssim \norm{F}_{G^{\falpha,\fgamma,\fdelta}L_m}\flambda^\falpha\,.
\]
Let us summarize this observation as a short corollary:

\begin{corollary}\label{cor:reconExtendedCoherent}
    Under the condition of Theorem \ref{theo:reconstruction}, the reconstruction $\mathcal{R}(F)$ is in $C^{\falpha}L_m$ and
\begin{equation}\label{ineq:continuityReconstructionAlpha}
        \norm{\mathcal{R}(F)}_{C^\falpha L_m}\lesssim \norm{F}_{G^{\falpha,\fgamma,\fdelta}L_m}\,,
    \end{equation}
\end{corollary}

\begin{proof}
This is a direct consequence of \eqref{ineq:expansionfTheta} together with \eqref{cond4Recon}.

\end{proof}

\begin{remark}\label{rem:ReconLosesDelta}
  It is interesting to see that  the stochastic parameter $\fdelta$ is lost  at  the level of the reconstruction, i.e. $\cR(F)$ will in general not be in $C^{\alpha,\fdelta}L_m$. This is natural, as $\cR(F)$, when evaluated against $\psi^{\flambda}_{\fx}$ locally behaves like $F_\fx(\psi^{\flambda}_{\fx})$ which is only in $C^{\alpha}L_m$. However, the multiplication with a (stochastic) white noise $F=(Y\cdot\xi)$ as in \eqref{eq_sto_integral} does result in $\cR(F)\in C^{\falpha,\fdelta}L_m$, see Corollary \ref{cor:ItoReconstructionInC} for details.
\end{remark}

\section{Application to a mixed SPDE}\label{spde_section}
Using the results established in the previous section  we will now show the existence and uniqueness of the equation \eqref{eq:spde}. As explained in the introduction, we can write \eqref{eq:spde} into its integral form to obtain formally 
\[
   u(\fx) =  I(v)(\fx) + \int_{\foo}^{\fx} \sigma(u(\fy)) \xi(d\fy) + \int_\foo^\fx f(\fy)u(\fy)\frac{\partial^d }{\partial x_1\ldots \partial x_d}Z(d\fy) \,,
\]
where the boundary function $I(v)\colon \mathbb R^d_+\to \mathbb R$ is explicitly given by 
\begin{equation}\label{equation_boundary_function}
   I(v)(\fx)= \sum_{\theta\subset\{1, \ldots d\},\; \theta\neq [d]} (-1)^{1+\sharp(\theta^c)} v(\pi^{\theta}_{\fx}\foo)\,.
\end{equation}
To derive it, it is sufficient to use the identity
$\int_\foo^\fx\frac{\partial^d }{\partial x_1\ldots \partial x_d}u(d\fy)=\square_{\foo,\fx}^{[d]}u$
which holds at the level of smooth functions $u$ and then separate the term depending on the boundary function. 

Both integrals in the above equation present formal products of distributions which cannot be defined point-wise but they will be interpreted rigorously via the action of reconstruction theorem and a proper integration map. More precisely, we can introduce the germs
\[(\sigma(u)\cdot  \xi)_{\fx}(\psi) := \sigma(u(\fx))\xi(\psi)\,, \quad  ((fu)\cdot \frac{\partial^d }{\partial x_1\ldots\partial x_d}Z)_{\fx}(\psi):= f(\fx)u(\fx)\frac{\partial^d }{\partial x_1\ldots\partial x_d}Z(\psi)\,. \]
The  reconstruction of these two germs can be easily linked with Walsh integration \cite{walsh1986} and Young product \cite{zambotti2020}. Then, as already explained in the introduction we will construct an integration map mapping a distribution onto its primitive function
for all distributions $f\in C^\fbeta L_m$ which are $\fbeta>-\fone$-H\"older.
Putting all together, we are interested in solving the  equation in $C^{\alpha,\fdelta}L_m$ 
\begin{equation}\label{eq:FixedPointGoal}
    u(\fx) =  I(v)(\fx) +\int_{\foo}^{\fx}\cR(\sigma(u)\cdot  \xi + (fu)\cdot \frac{\partial^d }{\partial x_1\ldots\partial x_d}Z)(d\fy)\,.
\end{equation} 
which will be constructed as a fixed point in $C^{\falpha,\fdelta}L_m$ for some $\falpha,\fdelta > \foo$.



\subsection{Composition and scalar multiplication with functions}

We recall that for any  given smooth function $g\colon \mathbb R\to \mathbb R$ and any  $u\in C^{\falpha}$, $g(u)\in C^\falpha$, i.e. composition with smooth functions preserves the regularity (For reference, see \cite[Lemma 3.1]{tindel07} or \cite[Lemma 15]{bechtold2022}). This, however, is no longer the case in the stochastic setting.  This is due to the fact, that one would like to apply the Taylor formula to $g(u)$ and use  $\square^{\{1\}}_{\fx,\fy}u \square^{\{2\}}_{\fx,\fy} u$ to bound $\square^{(1,2)}_{\fx,\fy} g(u)$. While for deterministic $u$ this gives the correct bounds for $\abs{\square^{\{1,2\}}_{\fx,\fy}g(u)}$, in the stochastic setting $\square^{\{1\}}_{\fx,\fy}u \square^{\{2\}}_{\fx,\fy} u$ is no longer in $L_m$, but only in $L_{m/2}$. Due to this lack of integrability, we can only hope for bounds in $C^{\falpha,\fdelta}L_{m/d}$. To keep the same index $m $ invariant we will need to lower  the indices $\falpha$ and $\fdelta$. In what follows for $\epsilon\in(0,1]$ we will use the notation $\mathcal{C}^\epsilon$ to denote the set of all $\epsilon$-Holder functions $g\colon \mathbb R\to \mathbb R$ with the convention of taking Lipschitz functions when $\epsilon=1$.


\begin{proposition}\label{prop:g(u)CanBeReconstructed}
    Let $u\in C^\falpha L_m$ and $g\in \mathcal{C}^\epsilon$ for some $\epsilon\in(0,1]$ and $\falpha\in (0,1]^d$. Then $g(u)\in C^{\frac{\falpha\epsilon}d}L_m$ and one has 
    \begin{equation*}
        \norm{g(u)}_{C^{\frac{\falpha\epsilon}d}L_m}\lesssim \abs{g(0)}+\norm{g}_{\mathcal{C}^\epsilon}\norm{u}_{C^\falpha L_m}\,,
    \end{equation*}
    where the constant in $\lesssim$ is allowed to depend on $T$.
\end{proposition}
\begin{proof}
The proof is based on the observation that for any $i \in [d]$  and  $\fx,\fy\in[0,T]^d$ one has 
    \begin{equation*}
        \norm{\square^{\{i\}}_{\fx,\fy} g(u)}_m \le \norm{g}_{\mathcal{C}^\epsilon}\norm{\abs{\square^{\{i\}}_{\fx,\fy} u}^\epsilon}_m \le \norm{g}_{\mathcal{C}^\epsilon}\norm{u}^\epsilon_{C^\falpha L_m}\abs{x_i-y_i}^{\alpha_i\epsilon}\,,
\end{equation*}
    where we used Jensen's inequality in the last step. If $\theta\subset[d]$ is an indexset and $i\in\theta$, we can use \eqref{eq:splitUpSquare} to extend this estimate to $\square^\theta_{\fx,\fy} g(u)$.
    \begin{align*}
        \norm{\square^\theta_{\fx,\fy} g(u)}_m &\le \sum_{\theta'\subset\theta\setminus\{i\}} \norm{\square^{\{i\}}_{\pi^{\theta'}_\fy \fx,\fy} g(u)}_m\\
        &\lesssim \norm{g}_{\mathcal{C}^\epsilon}\norm{u}^\epsilon_{C^\falpha L_m}\abs{x_i-y_i}^{\alpha_i\epsilon}\,.
    \end{align*}
Interpolation over $i\in\theta$ gives us for any non-empty $\theta\subset[d]$
    \begin{align*}
        \norm{\square^\theta_{\fx,\fy}g(u)}_m &=\prod_{i\in\theta} \norm{\square^\theta_{\fx,\fy}g(u)}_m^{\frac1{\sharp\theta}}\\
        &\lesssim \norm{g}_{\mathcal{C}^\epsilon}\norm{u}^\epsilon_{C^\falpha L_m}\prod_{i\in\theta}\abs{x_i-y_i}^{\frac{\alpha_i\epsilon}{\sharp\theta}}\,.
    \end{align*}
    The claim now follows from the observation that
    \begin{equation*}
        \norm{g(u)}_m \le \abs{g(0)}+\norm{g}_{\mathcal{C}^\epsilon}\norm{u}_{m}^{\epsilon}
    \end{equation*}
    holds.
\end{proof}
In the case of the scalar multiplication of $C^{\falpha,\fdelta}L_m$, the product follows the same structure of the usual H\"older inequality.
\begin{proposition}\label{prop:ProductCalphaSpace}
    Let $f\in C^{\falpha,\fdelta}L_n$ and $u\in C^{\falpha,\fdelta}L_m$ with $0\le\fdelta\le \falpha$. Then the scalar multiplication fulfills $fu\in C^{\falpha,\fdelta}L_k$ where $\frac 1k = \frac 1m+\frac 1n$. Furthermore, we have the inequality
    \begin{equation*}
        \norm{fu}_{C^{\falpha,\fdelta}L_k} \lesssim \norm{f}_{C^{\falpha,\fdelta}L_n}\norm{u}_{C^{\falpha,\fdelta}L_m}
    \end{equation*}
\end{proposition}
\begin{proof}
    The proof is based on the decomposition \eqref{eq:productOfSquares}. Note that by the adaptedness property of $f$ and $u$, we have that $\square_{\fx,\fy}^\theta f, \square_{\fx,\fy}^\theta u$ are $\cF_{\pi_\fx^{\theta^c} \fy}$-measurable for any $\fx\le\fy$ and $\theta\subset[d]$.

    Let $\fx\le\fy$ and $\eta\subset\theta\subset[d]$ with $\theta\neq\emptyset$. Whenever we are given two sets $\theta_1,\theta_2\subset\theta$ with $\theta=\theta_1\cup\theta_2$, we decompose $\eta$ into three disjoint sets
    \begin{equation*}
        \eta_1 = \eta\cap(\theta_1\setminus\theta_2), \qquad \eta_2 = \eta\cap(\theta_2\setminus\theta_1),\qquad \eta_3 = \eta\cap\theta_1\cap\theta_2\,.
    \end{equation*}
    One now calculates, using the adaptedness of $\square^{\theta_1}_{\fx,\fy} f, \square^{\theta_2}_{\fx,\fy} u$ and the Hölder inequality:
    \begin{align*}
        \norm{\bE^\eta_\fx \square^\theta_{\fx,\fy} (fu)}_k &\le \sum_{\theta = \theta_1\cup\theta_2} \norm{\bE^\eta_\fx(\square^{\theta_1}_{\fx,\fy} f\square^{\theta_2}_{\fx,\fy} u)}_k \\
        &= \sum_{\theta =\theta_1\cup\theta_2} \norm{\bE^{\eta_3}_\fx\left[ (\bE^{\eta_1}_\fx \square^{\theta_1}_{\fx,\fy} f)(\bE^{\eta_2}_\fx \square^{\theta_2}_{\fx,\fy} u)\right]}_k \\
        &\le \sum_{\theta =\theta_1\cup\theta_2} \norm{\bE^{\eta_1}_\fx\square^{\theta_1}_{\fx,\fy} f}_n \norm{\bE^{\eta_2}_\fx\square^{\theta_2}_{\fx,\fy} u}_m \\
        &\le \norm{f}_{C^{\falpha,\fdelta} L_n} \norm{u}_{C^{\falpha,\fdelta}L_m} \sum_{\theta=\theta_1\cup\theta_2}\abs{\fy-\fx}_{\text{c},\theta}^{\falpha_\theta + \fdelta_{\eta_1\cup\eta_2} + \falpha_{\eta_3}}\,.
    \end{align*}
    Using also the standard H\"older inequality 
    \[\norm{f(\fx)u(\fx)}_k \le \norm{f(\fx)}_n \norm{u(\fx)}_m  \]
    for all $\fx\in[0,T]^d$, one obtains the claim.
\end{proof}

\subsection{Products of processes and distributions}

This section discusses how multiplying a random field with a random distribution results in a stochastic coherent germ. More precisely, we want to understand under which conditions we can reconstruct the germ
\begin{equation}\label{eq:product}
    (u\cdot \zeta)_{\fx}(\psi):= u(\fx)\zeta(\psi)
\end{equation}
when $u\in C^{\falpha, \fdelta} L_m$ and $\zeta\in C^{\fbeta, \fdelta'} L_n$. Instead of a generic analysis, we will consider two different cases: The case in which $\zeta$ is deterministic, so we will need to use  the stochastic properties of $u$, and the case in which $\zeta = \xi$ is the  white noise we will obtain a coherent germ by taking simply  $u\in C^{\falpha} L_m$.

\subsubsection*{Deterministic distributions}\label{sec:productandrecon}

In the case in which $\zeta\in C^\fbeta$  for some $\fbeta<\foo$, any stochastic improvement needs to come from the process $u\in C^{\falpha,\fdelta}L_m$, so we will assume $\fdelta>\foo$. Note that we do not expect the reconstruction to have stochastic properties, as it behaves locally like an integrable random field times an element of $C^\fbeta$, see also Remark \ref{rem:ReconLosesDelta}. 

\begin{proposition}\label{lem:prodDeterministicZeta}
    Let $u\in C^{\falpha,\fdelta}L_m$ and $\zeta\in C^{\fbeta}$ for $\falpha,\fdelta>0$ and $\fbeta<0$. Then $u\cdot \zeta\in G^{\fbeta,\falpha+\fbeta,\fdelta}L_m$ and
    \begin{equation}\label{ineq:prodDeterministicZeta1}
    \norm{(u\cdot \zeta)}_{G^{\fbeta,\falpha+\fbeta,\fdelta}L_m}\le \norm{u}_{C^{\falpha,\fdelta}L_m}\norm{\zeta}_{C^\fbeta}\,.    
    \end{equation}
    It follows that $u\cdot\zeta$ can be reconstructed, as long as $\falpha+\fbeta > -\frac 12\fone$, $ \falpha+\fbeta+\fdelta > \foo$ and $m\ge 2$. In this case, we have $\cR(u\cdot\zeta)\in C^{\fbeta}L_m$ and the estimate
    \begin{equation}\label{ineq:prodDeterministicZeta2}
    \norm{\cR(u\cdot\zeta)}_{C^\fbeta L_m} \lesssim \norm{u}_{C^{\falpha,\fdelta}L_m}\norm{\zeta}_{C^\fbeta}\,.
    \end{equation}
\end{proposition}
\begin{proof}
    $u\cdot\zeta\in G^{\fbeta,\falpha+\fbeta, \fdelta}L_m$ as well as \eqref{ineq:prodDeterministicZeta1} follow immediately from the identity $\square^\theta_{\fx,\fy} (u\cdot \zeta)(\psi) = (\square_{\fx,\fy}^\theta u)\zeta(\psi)$ and definition of stochastic coherent germ. \eqref{ineq:prodDeterministicZeta2} is a direct consequence of Corollary \ref{cor:reconExtendedCoherent}.
\end{proof}

\subsubsection*{Product with white noise and stochastic integration}

We now consider the case of the product \eqref{eq:product} when $\zeta = \xi$ is white noise. Since $\xi\in C^{-\frac 12\fone,\infty}L_m$  regularise with ``infinite" regularity the conditional expectation from the filtration generated by the Brownian sheet  of $\xi$, we allow ourselves to consider integrands  $u\in C^{\falpha}L_m$. In this case, the reconstruction coincides with the stochastic integral with respect to the Brownian sheet, see \cite[Section 7]{kern2021}, \cite{Cairoli1975}
so we expect the reconstruction $f$ to have more properties than the ones given by the reconstruction theorem. 

\begin{proposition}\label{prop:ItoRecon}
    Let  $\xi$ be a white noise over $\mathbb R^d_+$ and $u\in C^{\falpha}L_m$, $\falpha> \foo$ is adapted to the filtration generated by the Brownian sheet  of $\xi$. Then the product germ $u\cdot \xi\in G^{-\frac 12\fone,-\frac 12\fone+\falpha, \infty}L_m$ and for all test functions $\psi\in C^r_c((0,T)^d) $, we have the identity
    \[\mathcal{R}(u \cdot\xi)(\psi)= \int_{\foo}^{+ \infty}\psi(\fy) u(\fy)\xi(d\fy) \]
    where the right-hand side is the Walsh integral associated with the white noise. Moreover, we have for $m\geq 2$
\begin{equation}\label{eq:ItoIsometryRecon}
        \norm{\mathcal{R}(u \cdot\xi)(\psi)}_{m} \lesssim \sup_{\fx\in[0,T]^d}\norm{u(\fx)}_m\norm{\psi}_{L_2([0,T]^d)}\,.
    \end{equation}
    Furthermore, if $(\cF_\fx)_{\fx\in[0,T]^d}$ is the filtration generated by $\xi$, for any $\fx\in[0,T]^d$, $i\in[d]$, and any test function $\psi$ with $\supp(\psi)\subset(x_1,T]\times\dots\times(x_d,T]$ we have 
\begin{equation}\label{eq:CondExpectationRuXi}
        \bE^i_\fx \cR(u\cdot\xi)(\psi) = 0\,.
    \end{equation}
\end{proposition}

\begin{remark}
    For the classical Walsh integral, \eqref{eq:ItoIsometryRecon} is a direct consequence of the Itô-isometry. Recall that in the case $m=2$ and for an adapted process $u$ the Walsh integral fulfills
    \begin{align*}
        \norm{\int_{\foo}^{+\infty} u(x) \psi(x) d\xi(x)}_2 &= \left(\int_{\foo}^{+\infty} \norm{u(x)}_2^2 \psi(x)^2 dx \right)^{\frac 12} \\
        &\le \sup_{\fx\in[0,T]^d} \norm{u(x)}_2 \norm{\psi}_{L_2([0,T]^d)}\,.
    \end{align*}
The general case with $m\geq 2$ generic was treated in \cite{Cairoli1970} but since we are in a distributional setting we present a proof of \eqref{eq:ItoIsometryRecon} which derives this directly from the reconstruction operator.
\end{remark}

\begin{proof}
    We start by showing $u\cdot\xi \in G^{-\frac 12\fone,-\frac 12\fone+\falpha, \infty}L_m$. Observe that for all $\fx,\fy\in[0,T]^d$ and $\theta\subset[d]$, we have
    \begin{equation*}
        \square_{\fx,\fy}^\theta (u\cdot \xi)(\psi) = (\square^\theta_{\fx,\fy} u)  \xi(\psi)\,.
    \end{equation*}
    For a $\psi$ with support in $[y_1,T]\times\dots\times[y_d,T]$ and $x_i\le y_i$ for at least one $i\in[d]$, we get by the adaptedness of $u$ and the fact that $\xi(\psi)$ is independent of $\cF^{\{i\}}_{y_i}$:
    \begin{align*}
        \norm{\square_{\fx,\fy}^\theta (u\cdot \xi)(\psi)}_m = \norm{\square^\theta_{\fx,\fy} u}_m \norm{\xi(\psi)}_m
    \end{align*}
    as well as
    \[
        \bE^\eta_\fx \square^\theta_{\fx,\fy}(u\cdot\xi)(\psi) = \bE^\eta_\fx[\square_{\fx,\fy}^\theta u ~\bE^i_\fy \xi(\psi)] = 0\,,
    \]
    for any $\eta\subset[d], i\in\eta$ such that $x_i\le y_i$ and $\psi$ as above. It follows easily that $u\cdot\xi\in G^{-\frac 12\fone,-\frac 12\fone+\falpha,\infty}L_m$. To see \eqref{eq:ItoIsometryRecon}, recall that
    \begin{equation*}
        \cR(u\cdot\xi)(\psi) = \lim_{\fn\to\infty} \sum_{\fx\in\Delta_\fn} (u\cdot\xi)_\fx(\phi_\fx^\fn)\scalar{\phi^\fn_\fx,\psi}\,.
    \end{equation*}
    We can use Lemma \ref{lem:multiparameterBDG} together with $\bE^i_\fx((u\cdot\xi)_\fx(\phi_\fx^n)) = 0$ for all $i\in[d]$, $\fx\in \Delta_\fn$ and $\fn\in\bN^d$, to show
    \begin{align*}
        \norm{\sum_{\fx\in\Delta_\fn} (u\cdot\xi)_\fx(\phi_\fx^\fn)\scalar{\phi^\fn_\fx,\psi}}_m &\lesssim \left(\sum_{\fx\in\Delta_\fn} \norm{u(\fx)}_m^2 \norm{\xi(\phi_\fx^\fn)}_m^2\abs{\scalar{\phi^\fn_\fx,\psi}}^2\right)^{\frac 12} \\
        &\le \sup_{\fx\in[0,T]^d}\norm{u(\fx)}_m \left(\sum_{\fx\in\Delta_\fn}\norm{\phi_\fx^\fn}_{L_2(\bR^d)}^2\abs{\scalar{\phi^\fn_\fx,\psi}}^2\right)^{\frac 12} \\
        &= \sup_{\fx\in[0,T]^d}\norm{u(\fx)}_m \norm{P_\fn\psi}_{L_2(\bR^d)}\,.
    \end{align*}
    \eqref{eq:ItoIsometryRecon} then follows from $\norm{\psi}_{L_2(\bR^d)} = \lim_{\fn\to\infty} \norm{P_\fn \psi}_{L_2(\bR^d)}$\,. To see \eqref{eq:CondExpectationRuXi}, we fix  $\psi$ with support in $(y_1,T]\times\dots\times(y_d,T]$. Thus, there exists a $\fz>\fy$ such that
    \[
        \supp(\psi) \subset [z_1,T]\times\dots [z_d,T] \subset (y_1,T]\times\dots\times(y_d,T]\,.
    \]
    Assume $\fx,\fn$ are chosen in such a way, that $\scalar{\phi_\fx^\fn,\psi} \neq 0$. Then the support of the functions must overlap, and by construction of our wavelets (recall $\supp(\phi)\subset[C,R]^d$), we can conclude that
    \[
        \supp(\phi_\fx^\fn)\subset[z_1-2^{-n_1}(R-C),T+2^{-n_1}(R-C)] \times\dots\times [z_d-2^{-n_d}(R-C),T+2^{-n_d}(R-C)]\,.
    \]
    Thus, provided that $\fn$ is large enough such that $z_i-2^{-n_i}(R-C) > y_i$ for $i\in[d]$ and denoting $\delta_i := 2^{-n_i}(R-C)$, we have for all $\fx\in\Delta_\fn$ that at least one of the following properties holds:
    \[
        \supp(\phi_\fx^\fn)\subset(y_1,T+\delta_1]\times\dots\times(y_d,T+\delta_d]\quad \text{or} \quad \scalar{\phi^\fn_\fx,\psi} = 0\,.
    \]
    Since both cases imply that
    \[
        \bE^i_\fy (u\cdot\xi)_\fx(\phi_\fx^\fn)\scalar{\phi_\fx^\fn,\psi} = 0
    \]
    for any $i\in[d]$, we conclude that
    \[\bE^i_\fy \cR(u\cdot\xi)(\psi) = \lim_{\fn\to\infty} \sum_{\fx\in\Delta_\fn} \bE^i_\fx (u\cdot\xi)_\fx(\phi_\fx^\fn)\scalar{\phi_\fx^\fn,\psi} = 0\,.\]
\end{proof}

Gathering all the above inequalities, one has simpler estimates on $\cR(u\cdot \xi)$.
\begin{corollary}\label{cor:ItoReconstructionInC}
    Let $u\in C^{\falpha} L_m$, $\falpha >\foo$ and $\xi$ be white noise over $\bR_+^d$. Then $\cR(u\cdot \xi)\in C^{-\frac 12\fone,\infty} L_m$ and we have
    \begin{equation}\label{ineq:ItoReconstructionInC}
        \norm{\cR(u\cdot\xi)}_{C^{-\frac 12\fone,\infty}L_m}\lesssim \norm{u}_{C^\falpha L_m}\,.
    \end{equation}
    If $g\in \cC^1$ is Lipschitz with Lipschitz constant $\norm{g}_{\cC^1}$ and $u,v\in C^\falpha L_m$, we have
    \begin{equation}\label{ineq:ItoLikeIsometry}
        \norm{\cR(g(u)\cdot\xi)- \cR(g(v)\cdot\xi)}_{C^{-\frac 12\fone,\infty}L_m} \lesssim \norm{g}_{\cC^1} \norm{u-v}_{C^\falpha L_m}
    \end{equation}
\end{corollary}

\begin{proof}
    Observe that for all $\fx\in[0,T]^d,\flambda\in(0,1]^d$ we have
    \[
        \sup_{\fx\in[0,T]^d} \norm{u(\fx)}_m\norm{\psi_\fx^\flambda}_{L_2([0,T]^d)} \le \norm{u}_{C^{\falpha}L_m}\flambda^{-\frac 12\fone}\norm{\psi}_{L_2([0,T]^d)}\,.
    \]
    Putting this in \eqref{eq:ItoIsometryRecon}, together with \eqref{eq:CondExpectationRuXi} we obtain \eqref{ineq:ItoReconstructionInC}. To see \eqref{ineq:ItoLikeIsometry}, we use \eqref{eq:ItoIsometryRecon} to calculate
    \begin{align*}
        \norm{(\cR(g(u)\cdot\xi)- \cR(g(v)\cdot\xi))(\psi)}_m &=  \norm{(\cR(g(u)-g(v))\cdot\xi)(\psi)}_m\\
        &\lesssim \sup_{\fx\in[0,T]^d}\norm{g(u(\fx))-g(v(\fx))}_m\norm{\psi}_{L_2(\bR^d)} \\
        &\lesssim \norm{g}_{\cC^1} \sup_{\fx\in[0,T]^d}\norm{u(\fx)-v(\fx)}_m\norm{\psi}_{L_2(\bR^d)}\,.
    \end{align*}
    The relation  \eqref{ineq:ItoLikeIsometry} then follows with the same argument as \eqref{ineq:ItoReconstructionInC}.
\end{proof}

\subsection{Primitives and integration}
We finally provide the last step to formulate \eqref{eq:FixedPointGoal} rigorously. That is an integration map 
\[f\mapsto \int_{\foo}^{\fx}f(d\fy)\]
to recover a function after obtaining a distribution. Following the same approach of \cite[Lemma 3.10]{Brault2019} we show that our stochastic spaces $C^{\falpha,\fdelta}L_m$ behave as one expects under integration. To do so, we fix a wavelet basis $\Phi =\{\hat\phi^{\fn,\zeta}_\fx\colon \fn\in\bN^d,\fx\in\Delta_\fn,\zeta\subset[d]\}$ like in Definition \ref{defn_wavelets} and we denote by $P_\fn$ the projection onto the respective spaces.

Let us address a problem arising from the wavelet approximation applied to $1_{[\fs,\ft]}$ for any $0\le\fs\le\ft\le T\fone$. Since the approximation enlarges the support, for any fixed $\fn\in\bN^d$, $P_\fn 1_{[\fs,\ft]}$ might be supported outside of $[0,T]^d$, causing undefined terms. We further will struggle to condition the term onto $\cF_\fs$, due to the same fattening effect. The first problem can be solved by extending our distribution $f$ to $\bR^d$ by setting $\tilde f(\psi) = f(\psi\vert_{[0,T]^d})$, but this does not solve the second problem. Because of this, we use another approach and only consider $\fs,\ft$ which have a certain distance $\fepsilon$ to the boundary. We further condition these points on a $\hat\fs<\fs$, also depending on the chosen $\fepsilon$. Under these conditions, we can use Brault's construction and we will call the result the \emph{interior primitive}. We can then use a dyadic argument to pass to the whole space $\fs,\ft\in[0,T]^d$ and condition on $\cF_\fs$ sharply.

\begin{proposition}\label{prop:h}
    Let $f\in C^{\falpha,\fdelta}L_m$ with  $-\fone<\falpha<\foo$ . Then  $f(P_{\fn}1_{(\fs,\ft]})$ converges in $L_m$ for all $0< \fs\leq \ft<\fone T$ as $\fn\to\infty$. We denote its limit by 
    \[\oint_{\fs}^{\ft}f(d\fy)\,.\]
    and we call it the interior primitive of $f$. The resulting limit satisfies the a.s. identity
    \begin{equation}\label{eq:additive_integral}
\oint_{\fs}^{\ft}f(d\fy)=\oint_{\fs}^{\pi^i_{\fu}\ft}f(d\fy)+ \oint_{\pi^i_{\fu}\fs}^{\ft}f(d\fy)
    \end{equation}
    for any $\fs<\fu<\ft\in (0,T)^2$ and $i\in [d]$. Furthermore, for each $\emptyset\neq\eta\subset[d]$, $\fepsilon>\foo$ we have the estimate 
    \begin{align}
        \norm{\oint_{\fs}^{\ft}f(d\fy)}_m &\lesssim \fepsilon^{\falpha}\norm{f}_{C^{\falpha,\fdelta}L_m} \compAbs{\ft-\fs}^{\falpha+\fone}\label{ineq:technicalh1}\\
        \norm{\bE^\eta_{\hat\fs} \oint_{\fs}^{\ft}f(d\fy)}_m &\lesssim \fepsilon^{\falpha}\norm{f}_{C^{\falpha,\fdelta}L_m} \compAbs{\ft-\fs}^{\falpha+\fone}\abs{\ft-\fs}^{\fdelta}_{\text{c},\eta}\,,\label{ineq:technicalh2}
    \end{align}
    where $\fs,\ft\in [0,T]^d$ satisfy $\hat\fs := \fs-\compAbs{\ft-\fs}*\fepsilon\in[0,T]^d$  and $\hat\ft := \ft+\compAbs{\ft-\fs}*\fepsilon\in[0,T]^d$, and the constant in $\lesssim$ is allowed to depend on $\lambda$.
\end{proposition}
\begin{remark}
    Note that the distance between $\hat s$ and $s$, as well as $\hat t$ and $t$, depends on $\compAbs{\ft-\fs}$. This choice allows us to later use a dyadic argument.
\end{remark}
\begin{remark}\label{rem:additiveIsOneParameter}
Property \eqref{eq:additive_integral} extends the usual additive property of a two-parameter function $A\colon \mathbb R\times \mathbb R \to \mathbb R$
\[A(s,t)=A(s,u)+ A(u,t)\]
at the level of $d$-parameter functions see \cite[Definition 4]{multiparameterStochSewing}. From this condition one identifies for any $\fs\in (0,T)^d$ a unique function 
\[\fy\mapsto \oint_{\fs}^{\fy}f(d\fz)=:I_{\fs, \fy}(f)\]
such that $I_{\fs, \fy}(f)=0$ for any $\fy$ such that $\pi^i\fy=\pi^i\fs$ for some $i \in [d]$ and 
\[\oint_{\fx}^{\fy}f(d\fz)= \square_{\fx,\fy}^{[d]}I_{\fs, \cdot}(f)\]
for any $\fs\leq \fx\leq  \fy\in (0,T)^d$. Since we cannot choose $\fs = \foo$ at this stage, we prefer to think of the interior primitive as a $2d$-parameter object and will only go the $d$-parameter setting with the primitive itself.
\end{remark}
    \begin{proof}

     Given $\fs$ and $\ft$ we choose $\fn$ such that $2^{-n_i}\le \frac{\epsilon_i}{R}\abs{t_i-s_i}\le 2^{-n_i+1}$ for $i\in[d]$, where $R$ comes from the wavelet $\phi$ (recall from Section \ref{sec:preliminaries} that $\supp(\phi)\subset[C,R]^d$ for some $0<C<R$). Then for any $\fk\ge\fn$ we write
    \begin{equation}\label{eq:proof_sum}
        P_\fk 1_{(\fs,\ft]} = \sum_{\chi\subset[d]} ~\sum_{\substack{\fn_\chi \le \fm_\chi\le \fk_\chi\\ \fm_{\chi^c} = \fn_{\chi^c}}} \hat P^\chi_\fm 1_{(\fs,\ft]}\,,
    \end{equation}
    In order to prove the convergence of $ P_\fk 1_{(\fs,\ft]}$ it suffices to find a proper bound on $f(\hat P^\chi_\fm 1_{[\fs,\ft]})$\,, which we can estimate  as
    \begin{equation*}
        \norm{f(\hat P_\fm^\chi 1_{(\fs,\ft]})}_m \le \sum_{\fx\in\Delta_{\fm}} \norm{f(\hat\phi_\fx^{\chi,\fm})}_m \abs{\scalar{\hat\phi_\fx^{\chi,\fm},1_{(\fs,\ft]}}}\,.
    \end{equation*}
  By our choice of $\fn$, all non-zero terms $\fx\in \Delta_{\fm} $ satisfy $\fx\in[\hat\fs,\ft]$ and the support of $\hat\phi^{\chi,\fm}_\fx$ is in $[\hat \fs,\hat\ft]\subset[0,T]^d$. For any fixed subset $\chi\subset[d]$ we note that in each direction $i\notin\chi$, the number of non-zero terms is of order $2^{-m_i}\abs{t_i-s_i}$. Moreover, note that for the one-dimensional wavelet $\hat\varphi$ we have
  \[
    \scalar{\hat\varphi,1_{(s,t]}} = 0\,,
  \]
  unless the support of $\hat{\varphi}$ contains one of the points $s,t\in[0,T]$. It follows that for all $i\in\chi$, the number of $x_i\in\Delta_{m_i}$ such that $\scalar{\hat\varphi^{m_i}_{x_i}, 1_{(s_i,t_i]}}$ is non-zero is a number independent of $m_i$ or $\abs{t_i-s_i}$. Thus, by counting the non-zero terms and using Lemma \ref{lem:inequalityPhi} with $r=0$, we get
    \begin{equation}\label{ineq:technicalh3}
    \begin{split}
        \norm{f(\hat P_\fm^\chi 1_{(\fs,\ft]})}_m&\lesssim \norm{f}_{C^{\falpha,\fdelta}L_m} 2^{-\fm*\falpha}2^{-\fm}_\chi\abs{\ft-\fs}^{\fone}_{\text{c},\chi^c}\\&= \norm{f}_{C^{\falpha,\fdelta}L_m} 2^{-\fm*\falpha}_{\chi^c}2^{-\fm*(\falpha+\fone)}_\chi\abs{\ft-\fs}^{\fone}_{\text{c},\chi^c}\,.
    \end{split}
    \end{equation}
    Note that the support of our wavelet $\phi$ is in general not in $[1/4,3/4]^d$. But we can always chose  $l\in\bN$ and $\fy\in\Delta_{l\fone}$ such that $\phi^{l\fone}_\fy$ has support in $[1/4,3/4]^d$, allowing us to use \eqref{ineq:defNegativHoelderSpace} on $f(\phi^\fm_\fx) = f(\phi^{l\fone_\fy})^{\fm-l\fone}_{\fx-2^{l\fone-\fm}\fy}$ for any $\fm\ge l\fone$. This only adds a constant depending on the wavelet to our calculations. 
    
    Since  $\falpha>-\fone $, one has 
    \[\sum_{\fm_\chi\ge\fn_\chi}\norm{f(\hat P_\fm^\chi 1_{(\fs,\ft]})}_m<\infty\,.\]  

    By standard arguments $f(P_\fk 1_{(\fs,\ft]})$ is then a Cauchy sequence in the parameter $\fk$. By summing $\eqref{ineq:technicalh3}$ over the respective $\fm_\chi\ge\fn_\chi$, we achieve a bound for \eqref{eq:proof_sum} independent from $\fk$:
    \[
        \norm{f(P_\fk  1_{(\fs,\ft]}) }_m \lesssim \sum_{\chi\subset[d]}\norm{f}_{C^{\falpha,\fdelta}L_m} 2^{-\fn*\falpha}_{\chi^c}2^{-\fn(\falpha+\fone)}_\chi\abs{\ft-\fs}_{\text c,\chi^c}^\fone\,.
    \]
    By letting $\fk\to\infty$, the first inequality \eqref{ineq:technicalh1} then from $2^{-n_i} \approx \epsilon_i\abs{t_i-s_i}$. The estimate \eqref{ineq:technicalh2} follows analogously. To see the identity \eqref{eq:additive_integral}, we simply observe that $P_\fn1_{(\fs,\ft]}$ is additive for each $\fn\in\bN^d$.
\end{proof}

At this point, the interior integral
$I_{\fs, \fy}(f)$ might depend on the choice of our wavelet basis $\Phi$. We show now  that the limit is unique by showing that its mixed distributional derivative is $f$. That, together with the fact that $I$ vanishes on the boundary by construction, that is, $I_{\fs,\pi^i_{s_i}\fx} = 0$ for all $\fs,\fx\in(0,T^d)$ and $i\in[d]$, uniquely characterizes $(\fs,\fy)\mapsto I_{\fs,\fy}$.

\begin{lemma}
    For each $\foo<\fs\le\fy<\fT$ let $I_{\fs,\fy}$ be as above. Further let $\psi\in C_c^{rd}$ with support in $(\fs,\fT)$. Then it holds that
    \[
(-1)^d\scalar{I_{\fs,\cdot},\partial^\fone\psi} = f(\psi)\,.
    \]
\end{lemma}
\begin{proof}
    Direct computation gives us for any $\fn\in\bN^d$
    \begin{align*}
        (-1)^d\scalar{f(P_\fn 1_{(\fs,\cdot]}),\partial^\fone\psi} &= \sum_{\fx\in\Delta_\fn}f(\phi_\fx^\fn)(-1)^d \scalar{\int_\fs^\cdot \phi_\fx^\fn(\fz)d\fz,\partial^\fone\psi} \\
        &= \sum_{\fx\in\Delta_\fn}f(\phi_\fx^\fn) \scalar{\phi_\fx^\fn,\psi} = f(P_\fn\psi)\,.
    \end{align*}
    Due to \cite[Lemma 62]{kern2021} and the continuity of the random distribution $f\in C^{\falpha,\fdelta}L_m$, we know that the right-hand side converges to $f(\psi)$ in $L_m$, at least along the diagonal sequence $\tilde\fn = (n,\dots,n)$ as $n\to\infty$.

    Let us analyse the convergence of the left-hand side. Using Minkowski inequality, \eqref{eq:proof_sum} and \eqref{ineq:technicalh3}, we get
    \begin{align*}
        \norm{\scalar{f(P_\fn 1_{(\fs,\cdot]})-I_{\fs,\cdot},\partial^\fone\psi}}_m &\le \int_{\foo}^\fT\norm{f(P_\fn 1_{[\fs,\cdot]})-I_{\fs,\cdot}}_m \abs{\partial^\fone\psi(\fz)} d\fz\\
        &\lesssim \norm{\psi}_{C^{rd}} \int_\foo^\fT \sum_{\emptyset\neq\chi\subset[d]} \sum_{\substack{\fm_\chi\ge \fn_\chi\\ m_{\chi^c}=\fn_{\chi^c}}}\norm{f(\hat P^\chi_\fm1_{(\fs,\fz]})}_m d\fz \\
        &\lesssim \norm{\psi}_{C^{rd}}\norm{f}_{C^{\falpha,\fdelta}L_m} \sum_{\emptyset\neq\chi\subset[d]} 2^{-\fn*\falpha} 2^{-\fn}_\chi \fT^{2\cdot\fone} \xrightarrow{\fn\to\infty}0\,.
    \end{align*}
    It follows that
    \[
        (-1)^d\scalar{f(P_\fn 1_{(\fs,\cdot]}),\partial^\fone\psi} \xrightarrow{\fn\to\infty} (-1)^d\scalar{I_{\fs,\cdot},\partial^\fone\psi}\,,
    \]
    finishing the proof.
\end{proof}

Due to the presence of the variable $\fepsilon$ in Proposition \ref{prop:h}, the interior primitive of $f$ does not suffice to define a proper integration over the whole cube $[0,T]^d$ when $\fs$ and $\ft$ approach the boundary. If only $\fs$ or $\ft$ approaches the boundary while the other one has some fixed distance to it, we can absorb this singularity via a dyadic expansion. Due to the additive property \eqref{eq:additive_integral}, this is an extension of the interior primitive.

\begin{lemma}\label{prop:integration}
Let $f\in C^{\falpha,\fdelta}L_m, -\fone<\falpha<\foo$ and $\foo\le\fs\le\ft< T\fone$ such that $\ft+\frac 12\compAbs{\ft-\fs} \in [0,T]^d$. For any $\fk\in \bN^d$, we define $\fs_\fk$ by $(\fs_\fk)_i := s_i +2^{-k_i}\abs{t_i-s_i}, i\in[d]$. Then, the limit
\begin{equation*}
    \int_\fs^\ft f(d\fy) := \sum_{\fk\ge \fone}\oint_{\fs_\fk}^{\fs_{\fk-\fone}} f(d\fy)
\end{equation*}
exists in $L_m$. It satisfies the additivity property \eqref{eq:additive_integral} and is $\cF_\ft$-measurable. Furthermore, we have
\begin{align}
    \norm{\int_\fs^\ft f(d\fy)}_m &\lesssim \norm{f}_{C^{\falpha,\fdelta}L_m}\abs{\ft-\fs}_{\text c}^{\falpha+\fone} \label{ineq:primitiveHoelder1}\\
    \norm{\bE^\eta_\fs\int_\fs^\ft f(d\fy)}_m &\lesssim \norm{f}_{C^{\falpha,\fdelta}L_m}\abs{\ft-\fs}_{\text c}^{\falpha+\fone}\abs{\ft-\fs}_{\text c,\eta}^{\fdelta}\,.\label{ineq:primitiveHolder2}
\end{align}
for all $\eta\subset[d]$.
\end{lemma}

\begin{proof}
    By construction of $\fs_\fk$, we have for all $\fk\ge\fone$ and $i\in[d]$
    \begin{equation*}
        \abs{(\fs_{\fk-\fone})_i-(\fs_\fk)_i} = 2^{-k_i}\abs{t_i-s_i} = \abs{(\fs_\fk)_i-s_i}\,.
    \end{equation*}
    Together with the assumption on $\ft$, this allows us to apply Proposition \ref{prop:h} to $\oint_{\fs_\fk}^{\fs_{\fk-\fone}} f(d\fy)$ to get
    \[
        \norm{\oint_{\fs_\fk}^{\fs_{\fk-\fone}} f(d\fy)} \lesssim \norm{f}_{C^{\falpha,\fdelta}L_m}2^{-\fk*(\falpha+\fone)}\compAbs{\ft-\fs}^{\falpha+\fone}\,.
    \]
    This is clearly summable in $\fk$ and summing over $\fk$ gives \eqref{ineq:primitiveHoelder1}. \eqref{ineq:primitiveHolder2} follows analogously.
\end{proof}

\begin{remark}
    With the same proof, we can also tackle the case in which $\fs-\frac 12\abs{\ft-\fs}_{\text c} \in[0,T]$ and $\fs\le\ft\le T\fone$. More generally, the same argument allows us to construct
    \[
        \int_\fs^\ft f(d\fy)
    \]
    for any $\foo\le\fs\le\ft\le T\fone$, such that for any $i\in[d]$, we have
    \[
        0\le s_i-\frac 12\abs{t_i-s_i} \qquad\text{or}\qquad t_i+\frac 12\abs{t_i-s_i}\le T\,.
    \]
\end{remark}

If we are given any $\fs\le\ft\in[0,T]^d$, we can extend the primitive by separating the rectangle spanned between $\fs$ and $\ft$ into $2^d$ smaller rectangles at the middle point $u=\frac 12(\fs+\ft)$. The additivity property \eqref{eq:additive_integral} as well as \eqref{ineq:primitiveHoelder1} and \eqref{ineq:primitiveHolder2} immediately extend to this case. Let us summarize this in a proposition.

\begin{proposition}
    Let $f\in C^{\falpha,\fdelta}L_m$ with $-\fone<\falpha<\foo$. For any $\fs\le\ft\in[0,T]^d$, let $\fu = \frac 12(\fs+\ft)$. Define
    \begin{equation*}
        \int_\fs^\ft f(d\fy) := \sum_{\theta\subset[d]} \int_{\pi^\theta_\fs \fu}^{\pi^\theta_\fu\ft} f(d\fy)\,.
    \end{equation*}
    Then the above is well-defined, fulfills the additivity property \eqref{eq:additive_integral}, and is $\cF_\ft$-measurable. Furthermore, we have for all $\eta\subset[d]$
    \begin{align}
    \norm{\int_\fs^\ft f(d\fy)}_m &\lesssim \norm{f}_{C^{\falpha,\fdelta}L_m}\abs{\ft-\fs}_{\text c}^{\falpha+\fone} \label{ineq:g1}\\
    \norm{\bE^\eta_\fs\int_\fs^\ft f(d\fy)}_m &\lesssim \norm{f}_{C^{\falpha,\fdelta}L_m}\abs{\ft-\fs}_{\text c}^{\falpha+\fone}\abs{\ft-\fs}_{\text c,\eta}^{\fdelta}\,.\label{ineq:g2}
\end{align}
\end{proposition}

Similar to Remark \ref{rem:additiveIsOneParameter}, due to its additivity we can think of the primitive of $f$ as a one-parameter process $Y_\ft = \int_\foo^\ft f(d\fy)$. \eqref{ineq:g1} and \eqref{ineq:g2} immediately give us the Hölder regularity of $Y$.

\begin{corollary}\label{cor:IntegrationInCalpha}
    Due to the additivity of $\int f(d\fy)$, there is a unique process $Y$, which is $0$ on the boundary $\bigcup_{i\in[d]} \pi^{\{i\}^c}_\foo [0,T]^d$ such that
    \[
        \square^{[d]}_{\fx,\fy} Y = \int_\fx^\fy f(d\fy)\,.
    \]
    This implies for any $\theta\subset[d]$, we have
    \[
        \square^\theta_{\fx,\fy} Y = \int_{\pi_\fx^\theta \foo}^{\pi^\theta_\fy \fx} f(d\fy)\,.
    \]
    It follows that $Y\in C^{\falpha+\fone,\fdelta}L_m$. We call $Y$ the primitive of $f$, and we have
    \[
\norm{Y}_{C^{\falpha+\fone,\fdelta}L_m} \lesssim \norm{f}_{C^{\falpha,\fdelta}L_m}\,.
    \]
\end{corollary}

\subsection{Existence and uniqueness}

We are now ready to state the main theorem for the existence and uniqueness of \eqref{eq:FixedPointGoal} on any finite time interval $[0,T]^d$ and an $m\ge 2$. 

\begin{theorem}\label{spde_thm}
    Let $\xi$ be white noise over $\bR^d_+$ and $Z\in C^{\fbeta}([0,T]^d)$ for some $\frac 12\fone <\fbeta <\fone$. Let also be $\fdelta =\fbeta-\frac 12\fone$ and $\falpha$  satisfying $\fdelta\le\falpha<\frac 12\fone$  and $\fdelta+ \fbeta+\falpha>\fone$. Under the assumption $\sigma\in\cC^1(\bR)$ and $f\in C^{\falpha, \fdelta}L_{\infty}([0,T]^d)$  for any initial datum $v\colon \partial_{\foo} [0,T]^d\to \mathbb R$ such that $I(v)\in C^{\falpha, \fdelta}L_{m}([0,T]^d)$ there exists a unique solution $u\in C^{\falpha, \fdelta}L_{m}([0,T]^d)$ to the equation \eqref{eq:FixedPointGoal}, with $ C^{\falpha, \fdelta}L_{m}([0,T]^d) $ defined with the filtration generated by the Brownian sheet  of $\xi$.
\end{theorem}
 

\begin{proof}
We will use the shorthand notation 
\[\zeta =\frac{\partial^d Z}{\partial x_1\ldots \partial x_d}\in C^{\fbeta- \fone}\,, \quad  v_0=\mathcal{I}(v)\,.\] We start by showing local existence and uniqueness, that is we first show that for any initial condition $v_0\in C^{\falpha,\fdelta}L_m$ there exists a  time horizon $\tau> 0$ 
where the equation \eqref{eq:FixedPointGoal} admits a unique solution. The global result will then be constructed by patching together local solutions over the whole time $T>0$. For any $0<\tau\leq T$ we consider the affine space
    \[
    C_{\tau, v_0} = \{u\in C^{\falpha,\fdelta}L_m([0,\tau]^d)\colon  \pi^i_{\foo}u= \pi^i_{\foo}v_0 \quad \text{for any} \; i\in [d] \}\,.    \]
 Similarly, we define the solution map
    \begin{align*}
    \cS_{\tau}(u) = \left\{\fx\mapsto v_0(\fx)+\int_\foo^\fx \cR(\sigma(u)\cdot\xi + (fu)\cdot\zeta)(d\fy)\colon \fx\in [0,\tau]^d\right\}
    \end{align*}
    We claim that $\cS_{\tau}\colon C_{\tau, v_0}\to C_{\tau, v_0}$ is a well-posed application which becomes a contraction for $\tau>0$ small enough, by endowing  $C_{\tau, v_0}$ with the distance 
  $  d(u,v) = \norm{u-v}_{C^{\falpha,\fdelta}L_m([0,\tau]^d)}$.
   Let us first check that $\cS_{\tau}(u)\in C_{\tau, v_0}$. Our choice of parameters ensures $\fbeta-\fone +\falpha > -\frac 12\fone$ as well as $\fbeta- \fone +\falpha+\fdelta >\foo$. Thus, by Proposition \ref{prop:ProductCalphaSpace} together with Proposition \ref{lem:prodDeterministicZeta},  we have
   \[\mathcal{R}((fu)\cdot\zeta))\in  C^{\fbeta- \fone}L_m\]
   Similarly, from Proposition \ref{prop:g(u)CanBeReconstructed} and Corollary \ref{cor:ItoReconstructionInC}
   \[\cR(\sigma(u)\cdot\xi) \in C^{-\frac 12\fone,+\infty}L_m\]
   Summing the two distributions, we have from Lemma \ref{lem:sumOfCalpha} that 
   \[\cR(\sigma(u)\cdot\xi + (fu)\cdot\zeta)\in C^{-\frac 12\fone,\fdelta}L_m\,,\]
   where we used $\fdelta = \fbeta-\frac 12\fone$. After the integration and the trivial inclusion $C^{\frac12\fone,\fdelta}L_m\subset C^{\falpha,\fdelta}L_m\ $ we can conclude that $\cS_{\tau}(u)\in C_{\tau, v_0}$. Passing to the contraction property,    we fix $u,v\in C_{\tau, v_0}$. Using Corollary \ref{cor:ItoReconstructionInC}  we have the estimate
    \begin{align}
        \norm{\cR(\sigma(u)\cdot\xi)-\cR(\sigma(v)\cdot\xi)}_{C^{-\frac 12\fone,\infty}L_m} &\lesssim \norm{\sigma}_{\cC^1}\norm{u-v}_{C^{\falpha}L_m} \label{ineq:estimateRSigmaXi}  \\    \nonumber    &\lesssim \norm{\sigma}_{\cC^1}\norm{u-v}_{C^{\falpha,\fdelta}L_m}\,,
    \end{align}
    where we used the H\"older embedding given in Proposition \ref{prop:CalphaEmbedding2} in the second line and we simply bound the elements with $\tau$ by $T$. 
    Let us now analyze the term $\cR((fu)\cdot\zeta -(fv)\cdot\zeta)$.  Propositions \ref{lem:prodDeterministicZeta} and \ref{prop:ProductCalphaSpace} come also with the estimate
    \begin{equation}\label{ineq:estimateRfZeta}
    \begin{split}
         &\norm{\cR((fu)\cdot\zeta)-\cR((f v)\cdot\zeta)}_{C^{\fbeta-\fone}L_m} =\norm{\cR((f(u-v))\cdot\zeta)}_{C^{\fbeta-\fone}L_m} \\&\lesssim\norm{f(u-v)}_{C^{\falpha,\fdelta}L_m}\norm{\zeta}_{C^\fbeta}
        \lesssim 
    \norm{f}_{C^{\falpha,\fdelta}L_\infty}\norm{u-v}_{C^{\alpha,\fdelta}L_m}\norm{\zeta}_{C^\fbeta}\,.
    \end{split}
    \end{equation}
     By Lemma \ref{lem:sumOfCalpha}, we have
     \begin{align*}
         &\norm{\cR(\sigma(u)\cdot\xi)+\cR((fu)\cdot\zeta) - \cR(\sigma(v)\cdot\xi) - \cR((fv)\cdot\zeta)}_{C^{-\frac 12\fone,\fdelta}L_m} \\
         &\qquad\qquad\qquad\qquad\lesssim \left(\norm{f}_{C^{\falpha,\fdelta}L_\infty} \norm{\zeta}_{C^\fbeta} + \norm{\sigma}_{\cC^1}\right)\norm{u-v}_{C^{\falpha,\fdelta}L_m}\,.
     \end{align*}
 Using the Corollary \eqref{cor:IntegrationInCalpha}, it follows that 
 \begin{align*}
&\norm{\int_\foo^{\cdot}\left(\cR(\sigma(u)\cdot\xi)+\cR((fu)\cdot\zeta) - \cR(\sigma(v)\cdot\xi) - \cR((fv)\cdot\zeta)\right)(d\fy)}_{C^{\frac 12\fone,\fdelta}L_m} \\
&\qquad\qquad\qquad\qquad\lesssim \left(\norm{f}_{C^{\falpha,\fdelta}L_\infty} \norm{\zeta}_{C^\fbeta} + \norm{\sigma}_{\cC^1}\right)\norm{u-v}_{C^{\falpha,\fbeta}L_m}
\end{align*}
 Finally, we use the H\"older embedding result shown in Proposition \ref{prop:CalphaEmbedding} to compute
     \begin{align*}
         &\norm{\cS_{\tau}(u)-\cS_{\tau}(v)}_{C^{\falpha,\fdelta}L_m}\lesssim \sum_{\theta\subset[d]} (\tau\fone)^{\frac 12\fone-\falpha_\theta}\norm{\cS_{\tau}(u)-\cS_{\tau}(v)}_{C^{\frac 12\fone,\fdelta}L_m}\\& \sum_{\theta\subset[d]} (\tau\fone)^{\frac 12\fone-\falpha_\theta}\norm{\int_\foo^{\cdot}\left(\cR(\sigma(u)\cdot\xi)+\cR((fu)\cdot\zeta) - \cR(\sigma(v)\cdot\xi) - \cR((fv)\cdot\zeta)\right)(d\fy)}_{C^{\frac 12\fone,\fdelta}L_m}\\
         &\lesssim \sum_{\theta\subset[d]} (\tau\fone)^{\frac 12\fone-\falpha_\theta}\left(\norm{f}_{C^{\falpha,\fdelta}L_\infty} \norm{\zeta}_{C^\fbeta} + \norm{\sigma}_{\cC^1}\right)\norm{u-v}_{C^{\falpha,\fbeta}L_m}\,,
     \end{align*}
     which becomes a contraction for some parameter $\tau>0$ small enough, which we denote by $T^{*}$. By the uniqueness of Banach's fixed point theorem, we immediately have the uniqueness of the local (and thus global) solution. To see the global existence, we will explicitly show how to construct the solution on $[0,T]^2$, for the sake of clarity  leaving the general construction in general dimension $d$ as a trivial extension. Note that the value $T^{*}$ constructed above does not depend on $v_0$, but only on $\norm{f}_{C^{\falpha,\fdelta}L_\infty}$, $ \norm{\zeta}_{C^\fbeta}$, $\norm{\sigma}_{\cC^1}$. Thus, by choosing $N$ large enough, such that $T/N\le T^*$, the above proof shows the existence and uniqueness of the solution over every small patch 
     \[
     A^{\fk} = A^{(k_1,k_2)} = [k_1 T/N, (k_1+1) T/N] \times [k_2 T/N,(k_2+1) T/N]\,,
     \]
     where $k_1,k_2 = 0,\dots, N-1$ and the initial condition for each patch is given by the equation solved on the previous patches in lexicographic order. We call the solutions on those patches $u^{\fk}$ and set for any $\fx\in A^\fk$ $u(\fx) = u^{\fk}(\fx)$\,.
     This gives us a unique solution to the equation \eqref{eq:FixedPointGoal}, as long as $u\in C^{\falpha,\fdelta}L_m$. To see this, we use the following notation: For any $x_i\in[0,T], i=1,2$ and $k\in[N-1]$, set
     \begin{equation*}
         x_i^{(k)} := \begin{cases}
            x_i\qquad &x_i \in[kT/N,(k+1)T/N] \\
            kT/N \qquad &x_i<kT/N \\
            (k+1)T/N \qquad &x_i>(k+1)T/N\,,
         \end{cases}
     \end{equation*}
     and for any patch $A^{\fk}$, $\fx\in[0,T]^2$ we set $\fx^{(\fk)} := (x_1^{(k_1)},x_2^{(k_2)})$. Then, the additivity of the rectangular increment immediately gives us for any $\fx\le\fy$ in $[0,T]^2$ and $\eta\subset\theta\subset[d]$:
     \begin{equation*}
         \norm{\bE^\eta_\fx\square^\theta_{\fx,\fy} u}_m \le \sum_{\fk\in[N-1]^2} \norm{\bE^\eta_\fx \square_{\fx^{(k)},\fy^{(k)}}^\theta u}_m\,.
     \end{equation*}
     Note that $\square^\theta_{\fx^{(k)}, \fy^{(k)}} u = 0$ for all $\fk$ such that $x_i\le y_i \le k_iT/N$ for some $i\in\theta$, so we get no problems with measurability. Furthermore, we have $\abs{x^{(k)} -y^{(k)}} \le\abs{x-y}$ for all $x,y\in[0,T]$ and $k\in[N-1]$. One easily calculates now that
     \begin{equation*}
         \norm{u}_{C^{\falpha,\fdelta}L_m} \le \sum_{\fk\in[N-1]^2} \norm{u^\fk}_{C^{\falpha,\fdelta}L_m} <\infty.
     \end{equation*}
\end{proof}

\section{Connection to stochastic multiparameter sewing}\label{sec:sto_sewing}

In dimension $d=1$, sewing can be seen as a special case of the reconstruction theorem \cite{Broux2022}. This observation also extends to the stochastic case \cite{kern2021}, but in general does not hold in dimension $d>1$. In fact, Harang's multiparameter sewing lemma \cite{harang21} and the reconstruction theorem described in \cite{zambotti2020} are not connected in a similar way. We claim that our formulation using rectangular increments is the missing link between the two results. As in dimension $1$, Theorem \ref{theo:reconstruction} extends the stochastic multiparameter sewing lemma from \cite{multiparameterStochSewing}. It also goes beyond it, as reconstruction is able to handle negative regularities. In this section, we show how one can derive the stochastic multiparameter sewing lemma from Theorem \ref{theo:reconstruction} and then present an example which lies beyond the scope of sewing, but can be solved by our reconstruction theorem.

\subsection{Deriving  sewing from  reconstruction}

Let us start by recalling the sewing lemma in question. Throughout this section, we fix a commuting filtration $(\cF_\ft)_{\ft\in[0,T]^d}$. A germ in the world of sewing is a $2d$-parameter process given by a map $\Xi:\Delta_T\times\Omega\to\bR$, where
\[
    \Delta_T = \{(\fs,\ft)\colon \foo\le\fs\le\ft\le T\fone\}
\]
is the simplex over $[0,T]^d$. We say that $\Xi$ is $\cF$-adapted, if for all $(\fs,\ft)\in\Delta_T$, $\Xi_{\fs,\ft}$ is $\cF_\ft$ measurable. One aims to construct a Riemann-type sum
\begin{equation}\label{eq:limitSewing}
    \cI\Xi_{\fs,\ft} := \lim_{\abs{\cP}\to 0} \sum_{[\fu,\fv]\in\cP} \Xi_{\fu,\fv}\,,
\end{equation}
where the limit goes over grid-like partitions with vanishing mesh size. A grid-like partition of the interval $[s_1,t_1]\times\dots\times[s_d,t_d]$ is a set of the form 
\[
\cP = \cP_1 \times\dots\times\cP_d\,
\]
where $\cP_i$ is a partition of $[s_i,t_i], i\in[d]$ and its mesh size is the supremum
\[
    \abs{\cP} = \sup_{[u_i,v_i]\in\cP_i, i\in[d]}\abs{u_i-v_i}\,.
\]
For each indexset $\theta\subset[d]$, we define 
\[
\cP^\theta = \cQ_1\times\dots\cQ_d\,
\]
$\cQ_i = \cP_i$ for $i\in\theta$ and $\cQ_i = \{[s_i,t_i]\}$ for $i\notin\theta$. Similar to one-dimensional sewing, the multiparameter sewing lemma is based on the $\delta$-operator, which in $d = 1$ is given by
\begin{equation*}
    \delta_u \Xi_{s,t} = \Xi_{s,t}-(\Xi_{s,u}+\Xi_{u,t})\,,
\end{equation*}
for $s\le u\le t$. It can be generalized to a multiparameter setting by applying a $\delta$ operator in each direction. More precisely, in $d\ge1$ we set for all $i\in[d]$ and $\fs\le\fu\le\ft$
\[
\delta^i_\fu \Xi_{\fs,\ft} = \Xi_{\fs,\ft} - (\Xi_{\fs,\pi^i_\fu \ft} + \Xi_{\pi^i_\fu\fs,\ft})\,.
\]
For indexsets $\eta\subset[d]$, this can be extended as usual by
\[
    \delta_\fu^\eta :=\prod_{i\in\eta} \delta^i_\fu\,.
\]
With these definitions, the coherence property for $2d$-parameter objects reads as follows.
\begin{definition}
    Let $\Xi:\Delta_T\times\Omega\to\bR$ be an $\cF$-adapted $2d$-parameter process. Further let $\falpha,\fbeta,\fgamma\ge\foo$ and $1\le m <\infty$. Then we call $\Xi$ stochastic $(\falpha,\fbeta,\fgamma)$-coherent in the sewing sense if there is a constant $C>0$ such that for all $(\fs,\ft)\in\Delta_T$ 
    \begin{equation}\label{ineq:coherenceSewing2}
        \norm{\Xi_{\fs,\ft}}_m\le C\compAbs{\ft-\fs}^{\falpha}\,,
    \end{equation}
    holds and for all $\eta\subset\theta\subset[d]$ with $\theta\neq\emptyset$ and $\fs\le\fu\le\ft$,
    \begin{equation}\label{ineq:coherenceSewing}
        \norm{\bE^\eta_\fs \delta^\theta_\fu\Xi_{\fs,\ft}}_m \le C\abs{\ft-\fs}^{\falpha}_{\text c,\theta^c}\abs{\ft-\fs}^{\fbeta}_{\text c,\theta}\abs{\ft-\fs}^{\fgamma}_{\text c,\eta}\,.
    \end{equation}
    We write $\Xi\in C_2^{\falpha,\fbeta,\fgamma} L_m$ and denote by $\norm{\Xi}_{C_2^{\falpha,\fbeta,\fgamma}L_m}$ the smallest possible constant $C\ge 0$ such that \eqref{ineq:coherenceSewing}, \eqref{ineq:coherenceSewing2} hold.
\end{definition}
\noindent This coherence property suffices to show the following.
\begin{theorem}[Multiparameter stochastic sewing, \cite{multiparameterStochSewing}]\label{theo:sewing}
    Let $\Xi\in C_2^{\falpha,\fbeta,\fgamma}L_m$. Assume $2\le m<\infty, \falpha\in[0,1]^d$, $\fgamma\in[0,\frac 12)^d, \fbeta\in(\frac 12,\infty)^d$ with $\fbeta+\fgamma > \fone$. Then, for $\theta\subset[d],(\fs,\ft)\in\Delta_T$ the following limit exists:
    \[
        \cI^\theta\Xi_{\fs,\ft} =\lim_{\abs{\cP}\to 0} \sum_{[\fu,\fv]\in\cP^\theta} \Xi_{\fu,\fv}\,\quad \text{in}\; L_m.
    \]
    Furthermore, the limits $(\cI^\theta\Xi_{\fs,\ft})_{\theta\subset[d],(\fs,\ft)\in\Delta_T}$ are the unique family of $2d$-parameter processes such that $\cI^\theta\Xi_{\fs,\ft}\in L_m(\Omega)$ holds for all $\theta\subset[d]$ and $(\fs,\ft)\in\Delta_T$ and the following four properties hold.
    \begin{itemize}
        \item[i)] $\cI^\emptyset\Xi = \Xi$.
        \item[ii)] $\cI^\theta\Xi_{\fs,\ft}$ is $\cF_\ft$-measurable for all $\theta\subset[d], \fs\le\ft$.
        \item[iii)] For all $\theta\subset[d]$, $\cI^\theta\Xi$ is additive in $\bR^\theta$. I.e. for all $i\in\theta$, we have that $\delta^i_\fu\cI^\theta\Xi_{\fs,\ft} = 0$ for all $\fs\le\fu\le\ft$.
        \item[iv)] For all $\eta\subset\theta\subset[d]$ with $\theta\neq\emptyset$ and $\fs\le\ft$, we have:
        \begin{equation*}
            \norm{\bE^\eta_\fs\sum_{\hat\theta\subset\theta} (-1)^{\abs{\hat\theta}}\cI^{\hat\theta}\Xi_{\fs,\ft}}_m \lesssim \norm{\Xi}_{C_2^{\falpha,\fbeta,\fgamma}L_m}\abs{\ft-\fs}^{\falpha}_{\text c,\theta^c}\abs{\ft-\fs}^{\fbeta}_{\text c,\theta}\abs{\ft-\fs}^{\fgamma}_{\text c,\eta}\,.
        \end{equation*}
    \end{itemize}
\end{theorem}

\begin{remark}\label{rem:strictlyOrderedSetting}
    Comparing the coherence properties for sewing and reconstruction, we see a slight difference in the choice of $\fx,\fy\in[0,T]^d$ compared to $(\fs,\ft)\in \Delta_T$: Reconstruction allows all $\fx,\fy$ such that $\fx_\theta\le\fy_\theta$ holds in equation \eqref{ineq:coherence}, whereas for sewing, \eqref{ineq:coherenceSewing} only holds for $\fs\le\ft$. The property $\fx_\theta\le\fy_\theta$ is necessary to construct the reconstructions $f_\fx^\theta(\psi)$ for test functions with support in $(0,T)^d$. If we restrict \eqref{ineq:coherence} to only hold for $\fx\le\fy$, the limit $f_\fx^\theta(\psi)$ only exists for test functions $\psi$ with support in $(x_1,T)\times\dots(x_d,T)$. Apart from this, Theorem \ref{theo:reconstruction} still holds in this setting. This restriction on $\supp(\psi)$ mirrors the fact that $\cI^\theta \Xi_{\fs,\ft}$ only exists for $\fs\le\ft$.
\end{remark}

\noindent We claim that under the assumption of Theorem \ref{theo:sewing}, the distributional derivative of $\Xi$ given by
\[
    \cF_\fx(\psi) := (-1)^d\int_{[0,T]^d} \Xi_{\fx,\fz}\partial^\fone \psi(\fz)d\fz = \scalar{\partial^\fone\Xi_{\fx,\cdot},\psi}
\]
is stochastic $(-\fone,\fbeta-\fone,\fgamma)$-coherent and can thus be reconstructed (up to the restriction described in Remark \ref{rem:strictlyOrderedSetting}). Here, $\partial^\fone \psi := \frac{\partial^d}{\partial x_1 \dots\partial x_d}\psi$. Furthermore, comparing the four properties uniquely characterizing $\cI^\theta\Xi$ and $f^\theta$, one quickly sees that $f^\theta$ is given by the distributional derivative of $\cI^\theta\Xi$. Let us formalize this observation.

\begin{proposition}\label{prop:SewingIsReconstruction}
    Let $\Xi\in C_2^{\falpha,\fbeta,\fgamma}L_m$ with $\falpha,\fbeta,\fgamma,m$ as in Theorem \ref{theo:sewing} and we set $r=2$. Further assume that for all $\fs\in[0,T]^d$, $\ft\mapsto \Xi_{\fs,\ft}$ is almost surely in $L^1([s_1,T]\times\dots\times[s_d,T])$. Then the germ $F_\fx := \partial^\fone\Xi_{\fx,\cdot}$ is extended stochastic $(-\fone,\fbeta-\fone,\fgamma)$-coherent. It follows that it can be reconstructed, and the reconstruction $f^\theta$ is connected to the sewings $\cI^\theta\Xi$ via
    \begin{equation*}
         f^\theta_\fx(\psi) = (-1)^d \int_{[0,T]^d} \cI^\theta\Xi_{\fx,\fz}\partial^\fone\psi(\fz)d\fz = \scalar{\partial^\fone\cI^\theta_{\fx,\cdot},\psi}\,.
    \end{equation*}
    for all $\theta\subset[d], \fx\in[0,T]^d$ and $\psi\in C^{rd}_c((x_1,T)\times\dots\times(x_d,T))$.
\end{proposition}

\noindent As mentioned in Remark \ref{rem:strictlyOrderedSetting}, $f^\theta(\psi)$ only exists for $\psi\in C^{rd}_c((x_1,T)\times\dots\times(x_d,T))$.

\begin{proof}
    We define the projections $\pi^{i,1},\pi^{i,2}$ by
    \begin{equation*}
        \pi^{i,1}_\fu(\fs,\ft) = (\pi_\fu^i\fs,\ft) \qquad \pi^{i,2}_\fu(\fs,\ft) = (\fs,\pi^i_\fu\ft)\,,       
    \end{equation*}
    for all $i\in[d]$. Since we only consider test functions with compact support, Fubini gives us for all $\psi\in C_c^{rd}([0,T]^d)$:
    \begin{equation*}
        \int_{[0,T]^d} \pi^{i,2}_\fy\Xi_{\fx,\fz}\partial^\fone\psi(\fz) d\fz = 0
    \end{equation*}
    for all $i\in[d]$ and $\fx\le\fy$. Thus, the following holds for $\fx\le\fy$ and $\psi\in C_c^{rd}((y_1,T)\times\dots\times(y_d,T))$.
    \begin{align*}
        \square^\theta_{\fx,\fy} F(\psi) &= (-1)^d\prod_{i\in\theta}\scalar{(\pi_\fy^{i,1}-\Id)\Xi_{\fx,\cdot},\partial^\fone\psi} \\
        &= (-1)^d\prod_{i\in\theta}\scalar{(\pi_\fy^{i,1}+\pi_\fy^{i,2}-\Id)\Xi_{\fx,\cdot},\partial^\fone\psi}\\
        &= (-1)^{d+\sharp{\theta}}\scalar{\delta^\theta_\fy\Xi_{\fx,\cdot},\partial^\fone\psi}\,,
    \end{align*}
    Let us conclude that $F$ is stochastic coherent. For $\fx,\fy,\psi$ and $\flambda$ as in Definition \ref{def:extendedCoherence} and $\eta\subset\theta\subset[d]$, we have
    \begin{align}
        \norm{\bE^\eta_\fx \square^\theta_{\fx,\fy}F(\psi_\fy^\flambda)}_m &\le \sup_{\fz\in\supp(\psi_\fy^\flambda)}\norm{\bE^\eta_\fx\delta^\theta_\fy\Xi_{\fx,\fz}}_m \norm{\partial^\fone\psi_\fy^\flambda}_{L_1(\bR^d)}\nonumber\\& \le \norm{\Xi}_{C_2^{\falpha,\fbeta,\fgamma}L_m}\flambda^{-\fone}(\compAbs{\fx-\fy}+\flambda)^{\fbeta_\theta+\fgamma_\eta+\falpha_{\theta^c}}\label{ineq:cohFromSewing}\\
        &\lesssim \norm{\Xi}_{C_2^{\falpha,\fbeta,\fgamma}L_m}\flambda^{-\fone}(\compAbs{\fx-\fy}+\flambda)^{\fbeta_\theta+\fgamma_\eta}\,.\nonumber
    \end{align}
    By assumption, $\fbeta-\fone > -\frac 12\fone$ and $\fbeta+\fgamma-\fone > 0$. Thus, we can apply Theorem \ref{theo:reconstruction} to get reconstructions $f^\theta$, $\theta\subset[d]$. We claim that $f^\theta_\fx = \partial^\fone \cI^\theta\Xi_{\fx,\cdot}$ for all $\theta\subset[d],\fx\in[0,T]^d$. To prove this, we check the four properties uniquely characterizing $f^\theta$ in Theorem \ref{theo:reconstruction}:
    \begin{itemize}
        \item $\partial^\fone\cI^\emptyset\Xi_{\fx,\cdot} = \partial^\fone \Xi_{\fx,\cdot} = F_\fx$.
        \item $\scalar{\partial^\fone\cI^\theta\Xi_{\fx,\cdot},\psi}$ is $\cF_{\fx\vee\fy}$-measurable for all $\psi$ with $\supp(\psi)\subset(-\infty,\fy]$.
        \item As before, $\scalar{\partial^\fone \cI^\theta\Xi_{\fx,\pi^i_{\tilde\fx}\cdot},\psi} = 0$ for all $\fx\le\tilde\fx$ and $i\in[d]$ due to $\psi$ having a compact support. Together with the additivity of $\cI^\theta\Xi$, this implies that
        \[
        \scalar{\partial^\fone\cI^\theta\Xi_{\fx,\cdot},\psi} = \scalar{\partial^\fone \cI^\theta\Xi_{\pi^i_{\tilde\fx}\fx,\cdot}+ \partial^\fone \cI^\theta\Xi_{\fx,\pi^i_{\tilde\fx}\cdot},\psi} = \scalar{\partial^\fone \cI^\theta\Xi_{\pi^i_{\tilde\fx}\fx,\cdot},\psi}\,,
        \] 
        for all $i\in\theta$. It follows that $\partial^\fone \cI^\theta\Xi_{\fx,\cdot}$ only depends on $x_i$ for $i\notin\theta$.
        \item Let $\psi,\fx\le\fz,\flambda,\eta,\theta$ be as in \eqref{cond4Recon} of Theorem \ref{theo:reconstruction}. Then:
          \begin{align*}
\norm{\bE^\eta_\fx\sum_{\hat\theta\subset\theta}(-1)^{\sharp{\hat\theta}} \scalar{\partial^\fone\cI^{\hat\theta}\Xi_{\fx,\cdot},\psi_\fz^\flambda}}_m &\le \sup_{\fy\in\supp(\psi_\fz^\flambda)}\norm{\bE^\eta_\fx \sum_{\hat\theta\subset\theta}(-1)^{\sharp{\hat\theta}} \cI^{\hat\theta}\Xi_{\fx,\fy}}_m\norm{\partial^\fone\psi_\fz^\flambda}_{L^1(\bR^d)} \\
            &\lesssim \flambda^{-\fone}\flambda^{\fbeta}_\theta\flambda^{\fgamma}_\eta\,,
        \end{align*}
        where we used $\abs{y_i-x_i} = \abs{y_i-z_i}\lesssim\lambda_i$ for all $\fy\in\supp(\psi_\fz^\flambda)$ and $i\in\theta$.
    \end{itemize}
    This shows that $\partial^\fone\cI^\theta\Xi_{\fx,\cdot} = f_\fx^\theta$, which finishes the proof.
\end{proof}

\begin{remark}
    Surprisingly, the parameter $\falpha$ from the sewing lemma does not turn up at all on the reconstruction side. It is rather replaced by the regularity of the distributional derivative $-\fone$. If one wants to keep track of the H\"older regularity $\falpha$, one would need to exchange \eqref{ineq:coherence} in the coherence property with a term of the form
    \begin{equation*}
        \norm{\bE^\eta_\fx\square_{\fx,\fy}^\theta F(\psi_\fy^\flambda)}_p \lesssim \flambda^{\falpha}(\compAbs{\fx-\fy}+\flambda)^{(\fgamma-\falpha)_\theta+\fdelta_\eta+\tilde\falpha_{\theta^c}}\,,
    \end{equation*}
    where $\falpha$ would be the typical parameter encountered in reconstruction and $\tilde\falpha$ is the H\"older continuity inherited from sewing.
\end{remark}
\begin{remark}\label{rk_sew_fix_point}
By linking the stochastic multiparameter  reconstruction with stochastic multiparameter sewing, it is natural to ask if  we can obtain an alternative proof of Theorem  \eqref{spde_thm} using just the second tool. This is indeed possible provided that the resulting path $\fx\to \mathcal{I}^{[d]}\Xi_{\foo, \fx}$  lies in the right $C^{\falpha,\fdelta}L_m$ space in the same way as the combined action of the primitive operation defined in Corollary \ref{cor:IntegrationInCalpha} and the multiparameter stochastic reconstruction theorem. We briefly sketch how to adapt the proof of Theorem  \eqref{spde_thm} in this setting. Firstly, one writes the equation 
\[
   u(\fx) =  I(v)(\fx) + \int_{\foo}^{\fx} \sigma(u(\fy)) \xi(d\fy) + \int_\foo^\fx f(\fy)u(\fy)\frac{\partial^d }{\partial x_1\ldots \partial x_d}Z(d\fy) \,,
\]
via sewing as 
\[u(\fx) =  I(v)(\fx) + \mathcal{I}^{[d]}_{\foo,\fx}(\Xi^{\sigma}(u))+ \mathcal{I}^{[d]}_{\foo,\fx}(\Xi^{f}(u))\,,\]
where $\Xi^{\sigma}(u)$ and $\Xi^{f}(u)$ are the $2d$ parameter processes
\[\Xi^{\sigma}(u)_{\fu,\fv} := \sigma(u(\fu))\square_{\fu,\fv}^{[d]}B\,, \quad  \Xi^{f}(u)_{\fu,\fv} := f(\fu)u(\fu)\square_{\fu,\fv}^{[d]}Z\,, \]
Using the properties of $C^{\falpha,\fdelta}L_m$, in order to prove the same results of Theorem  \eqref{spde_thm} it is sufficient to show that for any $\mathcal{I}^{[d]}(\Xi^{\sigma}(u))\in C^{\frac{1}{2}\fone,+\infty}L_m$ and $\mathcal{I}^{[d]}(\Xi^{f}(u))\in C^{\fbeta,0}L_m$. From this one can apply Lemma  \ref{lem:sumOfCalpha} and obtain that the sum belongs to $C^{\falpha,\delta}L_m$ with $\fdelta= \fbeta-\frac{1}{2}\fone$ and close the fix point.
\end{remark}

\subsection{Example: Products between functions and stochastic noises}

One of the core advantages of reconstruction over sewing is its ability to handle objects of arbitrarily (i.e., negative) regularity. We did not fully exploit this in the solution of Equation \ref{equation_mild}, which could also be solved using stochastic multiparameter sewing, see Remark \ref{rk_sew_fix_point}.  In this last part, we present a simple example outside the reaches of stochastic multiparameter sewing where one need to apply the stochastic multiparameter reconstruction.

Let us first recall how sewing can be used to construct integrals against white noise. Since $\xi(1_{[\fs,\ft]}) := \int_{\fs}^\ft d\xi$ corresponds to the rectangular increment of a Brownian sheet $\square_{\fs,\ft}^{[d]} B$, the integral $\int Z_\fx d\xi(\fx)$ of some stochastic process $Z_\fx$ can be constructed by sewing the germ $\Xi_{\fs,\ft} = Z_\fs \square_{\fs,\ft}^{[d]} B$. However, if we replace $\xi$ with a more irregular noise $\zeta$, $\zeta(1_{[\fs,\ft]})$ ceases to be well defined and one can no longer construct $\int Z_\fx\zeta(d\fx)$ via stochastic multiparameter sewing. Therefore, multiplying a sufficiently smooth function $g$ by a low-regularity noise $\zeta$ serves as an example where reconstruction becomes necessary.

Throughout this section, we chose $\zeta\in C^{\falpha,\fdelta}L_m([0,T]^d)$ for arbitrarily low $\falpha<\foo$ and $\fdelta\ge \foo$. When $\alpha_i <0$ for at least one $i\in[d]$, multiparameter sewing no longer suffices to construct integrals of the form $\int Z_\fx\zeta(d\fx)$. If it is less than $-\frac 12$, it is outside the scope of stochastic multiparameter sewing. Our goal is to construct the product $g\cdot\zeta$ for deterministic functions $g\in C^\fbeta,\fbeta>\foo$, given a fixed noise $\zeta\in C^{\falpha,\fdelta}L_m$ for arbitrarily low $\falpha <\foo$, while minimizing the required regularity $\fbeta>\foo$ of $g$. We begin by extending the space $C^\fbeta$ to a broader regime $\fbeta >\fone$:

\begin{definition}\label{def:Calpha}
    Let $\fbeta\in[0,\infty)^d$ and choose $\fn\in\bN^d, \fgamma\in(0,1]^d$ such that $\fbeta = \fn+\fgamma$. If $\beta_i = 0$ for some $i\in[d]$, we choose $n_i = \gamma_i= 0$. Let $g:[0,T]^d\to \bR$ be such that all derivatives $\partial^\fk g$ exist for $\foo\le\fk\le\fn$. Then we say that $g\in C^\fbeta$ if the norm
    \[
        \norm{g}_{C^\fbeta} := \sum_{\fk\le\fn} \left(\norm{\partial^\fk g}_\infty + \norm{\partial^{\fk} g}_{C^\fgamma}\right)
    \]
    is finite. Here, $\norm{\cdot}_{C^\fgamma}$ is defined as in \eqref{eq_det_Calpha}.
\end{definition}

\begin{remark}
    In $d=1$, it is more common to use the norm
    \begin{equation}\label{eq:CalphaUsual}
        \norm{g}^*_{C^\beta} = \left(\sum_{k\le n}\norm{\partial^k g}_\infty\right) + \norm{\partial^ng}_{C^\gamma}\,,
    \end{equation}
    than the one used in Definition \ref{def:Calpha}. Although \eqref{eq:CalphaUsual} can be extended to a rectangular increment setting, it requires surprisingly heavy notation as one needs to deal with $\fk$, such that $k_i=n_i$ for some $i\in[d]$ but not all. Since these norms are equivalent on compact sets $[0,T]^d$, we prefer to work with the norm specified in Definition \ref{def:Calpha}.
\end{remark}

\noindent Let us briefly discuss the expected regularity $\fbeta$: Without the use of rectangular increments or stochastic techniques, we still have $\zeta\in C^{-\abs{\falpha}}$ and $g\in C^{\min_{i\in[d]} \beta_i}$. Thus, the Young product (e.g., \cite[Section 14]{zambotti2020}) guarantees the existence of $g\cdot\zeta$ whenever $\beta_i>\abs{\falpha}$ for all $i\in[d]$. If we use rectangular increments, one should be able to make full use of all derivatives of $g$, allowing one to relax the assumption to $\falpha+\fbeta > \foo$. Using stochastic techniques, we gain $\frac 12$ regularity in all directions as before, provided that $\fdelta$ is sufficiently large. This leads us to anticipate the following conditions on $\fbeta$:
\begin{align*}
\fbeta+\falpha &> -\frac 12\fone\\   
\fbeta+\falpha+\fdelta&>\foo\,.
\end{align*}
Proposition \ref{prop:YoungProduct} confirms that this is indeed the right set of conditions.

Before presenting our main results, we develop some Taylor estimates for the rectangular increment setting. These are motivated by \cite[Example 4.1]{zambotti2020}. Given a sufficiently smooth function $g:[0,T]^d\to\bR$, $i\in[d]$ and $n_i\in\bN$ we use the notation $\partial^{i,n_i} := \frac{\partial^{n_i}}{\partial x_i^{n_i}}$, and define the $n_i$-th Taylor approximation in the $i-th$ coordinate by
\[
    T^{i,n_i}_{x_i} g(\fz) := \sum_{k_i\le n_i}\frac{\partial^{i,k_i} g(\pi^i_{x_i}\fz)}{k_i!} (z_i-x_i)^{k_i}\,.
\]
Recall the remainder formula
\[
    (\Id-T^{i,n_i}_{x_i}) g(\fz) = R^ig(x_i,\fz,n_i) = \int_{x_i}^{z_i} \frac{(z_i-t_i)^{n_i-1}}{(n_i-1)!} (\partial^{i,n_i} g(\pi^i_{t_i}\fz)- \partial^{i,n_i} g(\pi^i_{x_i}\fz)) dt_i\,,
\]
for any $n_i\ge 1$. For $n_i=0$, we simply set
\[
    R^ig(x_i,\fz,n_i) = g(\fz)-g(\pi^i_{x_i}\fz)\,.
\]
For any $\theta\subset[d]$, $\fn\in\bN^d$ and $\fx\in[0,T]^d$, we can extend these definitions by setting
\[
    T^{\theta,\fn}_\fx := \prod_{i\in\theta} T^{i,n_i}_{x_i},\qquad R^\theta g(\fx,\fz,\fn) = \prod_{i\in\theta} (\Id-T^{i,n_i}_{x_i})g(\fz)\,.
\]
We write $T^\fn_\fx := T^{[d],\fn}_\fx$. A direct calculation leads to an explicit formula for $R^\theta g(\fx,\fz,\fn)$. Let $\tilde\theta\subset\theta$ be such that $i\in\tilde\theta$ if and only if $n_i\ge 1$. Then
\begin{equation}\label{def:remainderTaylorRegIncrement}
R^\theta g(\fx,\fz,\fn) := \int_{\fx_{\tilde\theta}}^{\fz_{\tilde\theta}} \prod_{i\in\tilde\theta} \frac{(z_i-t_i)^{n_i-1}}{(n_i-1)!} \square_{\pi^\theta_\fx\fz,\pi_\ft^{\tilde\theta}\fz}^\theta \partial^{\fn_\theta} g ~d\ft_\theta\,.
\end{equation}

\begin{lemma}\label{lem:TaylorRecIncrement}
    Let $g\in C^{\fbeta}, \fbeta>\foo$ and let $\fn,\fgamma$ be as in Definition \ref{def:Calpha}. Let $\theta\subset[d], \fx,\fy,\fz\in[0,T]^d$. Then for all $\fk\le\fn$, there exists $A(\fx,\fy,\fk,\theta)\in\bR$ such that
    \begin{align}
    \begin{split}\label{ineq:Taylor1}
        \square^\theta_{\fx,\fy} T^{\fn} g(\fz) &= \sum_{\fk\le\fn} (\fz-\fy)^\fk A(\fx,\fy,\fk,\theta)\\
        \abs{A(\fx,\fy,\fk,\theta)} &\lesssim \compAbs{\fx-\fy}^{\fbeta_\theta-\fk}\norm{g}_{C^\fbeta}\,.
        \end{split}
    \end{align}
    Here, the rectangular increment is taken with respect to the map $\fx\mapsto T^\fn_\fx g(\fz)$.
    
    Furthermore, if $g\in C^{r+\abs{\fn_\theta+\fone_\theta}}$, then for each $\theta\subset[d],~\fx,\fy,\fz\in[0,T]^d$ such that $\fx_\theta = \fz_\theta$ and $\flambda\in(0,1]^d$, there exists a $B(\fy,\fz,\flambda,\theta)\in\bR$, such that
    \begin{align}
    \begin{split}\label{ineq:Taylor2}
        R^\theta g(\fx,\flambda*\fy+\fz,\fn) &= \flambda^{\fn_\theta+1_\theta} B(\fy,\fz,\flambda,\theta)\,, \\
        \norm{B(\cdot,\fz,\flambda,\theta)}_{C^r}&\lesssim \norm{g}_{C^{r+\abs{\fn_\theta+\fone_\theta}}}
    \end{split}
    \end{align}
\end{lemma}

\begin{proof}Let us start with proofing \eqref{ineq:Taylor1}. Writing out the rectangular increment using \eqref{rectangular_increment}, one gets
    \begin{align*}
        \square_{\fx,\fy}^\theta T^\fn g(\fz) &= \left[\prod_{i\in\theta} (\pi^i_{y_i} -\Id) \prod_{i=1}^d T_{x_i}^{i,n_i}\right]g(\fz) \\
        &=\left[\prod_{i\in\theta} (T_{y_i}^{i,n_i}-T_{x_i}^{i,n_i}) \prod_{i\notin\theta} T_{x_i}^{i,n_i}\right]g(\fz)\,.
    \end{align*}
    We write $\tilde g(\fz) = T^{\theta^c,\fn}_\fx g(\fz)$. Observe that writing out the Taylor expansion $T^{i,n_i}_{y_i}$ for some $i\in\theta$ and expanding all $\partial^{i,k_i}g(\pi^i_{y_i}\fz)$ in $x_i$ leads to
     \begin{align*}
        T^{i,n_i}_{y_i} \tilde g(\fz) &= \sum_{k_i=0}^{n_i} \frac{1}{k_i!} \partial^{i,k_i} \tilde g(\pi_{y_i}^i\fz) (z_i-y_i)^{k_i}\\
        &= \sum_{k_i=0}^{n_i}  \frac{(z_i-y_i)^{k_i}}{k_i!} \left( \sum_{l_i= 0}^{n_i-k_i}\frac{1}{l_i!} \partial^{i,k_i+l_i} \tilde g(\pi^i_{x_i}\fz)(y_i-x_i)^{l_i} + R^i\tilde g(x_i,\pi_{y_i}^i\fz,n_i-k_i)\right)\,.
    \end{align*}
    Using the binomial formula leads to
    \begin{equation}\label{eq:NewSupportPoint}
    (T_{y_i}^{i,n_i} -T_{x_i}^{i,n_i})\tilde g(\fz) = \sum_{k_i=0}^{n_i} \frac{(z_i-y_i)^{k_i}}{k_i!} R^i\tilde g(x_i,\pi_{y_i}^i\fz,n_i-k_i)\,.
    \end{equation}
    Note that due to the definition of $R^i\tilde g$, we can write for any $j\neq i$
    \[
    (T^{j,n_j}_{y_j}- T^{j,n_j}_{x_j}) R^i\tilde g(x_i,\pi^i_{y_i}\fz,n_i) = \int_{x_i}^{y_i} \frac{(y_i-t_i)^{n_i-1}}{(n_i-1)!} (\pi^i_{t_i} - \pi^i_{x_i})(T^{j,n_j}_{y_j}- T^{j,n_j}_{x_j})\left(\partial^{i,n_i}\tilde g\right)(\fz) dt_i\,,
    \]
    allowing us to iteratively apply the identity \eqref{eq:NewSupportPoint} to all $i\in\theta$. This leads to the formula
    \begin{equation}\label{eq:rectangIncrementTaylorExpansion}
        \square_{\fx,\fy}^\theta T^\fn g(\fz) = \prod_{i\in\theta}(T^{i,n_i}_{y_i}- T^{i,n_i}_{x_i})\tilde g(\fz) = \sum_{\fk_\theta = \foo}^{\fn_\theta}\prod_{i\in\theta}\frac{(z_i-y_i)^{k_i}}{k_i!} R^\theta \tilde g(\fx,\pi^\theta_\fy \fz,\fn-\fk)\,.
    \end{equation}
    The term $R^\theta \tilde g(\fx,\pi^\theta_\fy \fz,\fn-\fk)$ still depends on $z_i$ for $i\notin\theta$, so we extend it as follows:
    \begin{align}\begin{split}\label{eq:ExpandingGTaylor}
        R^\theta T_\fx^{\theta^c,\fn} g(\fx,\pi^\theta_\fy \fz,\fn-\fk) &= \sum_{\fk_{\theta^c}\le\fn_{\theta^c}} \prod_{i\in\theta^c} \frac{(z_i-x_i)^{k_i}}{k_i!} \\&\qquad\int_{\fx_{\tilde\theta}}^{\fy_{\tilde\theta}} \prod_{i\in\tilde\theta} \frac{(y_i-t_i)^{n_i-k_i-1}}{(n_i-k_i-1)!} \square^\theta_{\pi_\fx^\theta\fz,\pi_{\ft_{\tilde\theta}}^{\tilde\theta}\fy} \partial^{\fn_\theta-\fk_{\theta}+\fk_{\theta^c}}g(\pi_\fx^{\theta^c}\cdot)d\ft_{\tilde\theta}\,,
    \end{split}\end{align}
    where $\tilde\theta\subset\theta$ such that $i\in\tilde\theta$ if and only if $n_i-k_i\ge 1$. Note that
    \begin{align*}
        A_1(\fx,\fy,\fk,\theta) :&= \int_{\fx_{\tilde\theta}}^{\fy_{\tilde\theta}} \prod_{i\in\tilde\theta} \frac{(y_i-t_i)^{n_i-k_i-1}}{(n_i-k_i-1)!} \square^\theta_{\pi_\fx^\theta\fz,\pi_{\ft_{\tilde\theta}}^{\tilde\theta}\fy} \partial^{\fn_\theta-\fk_\theta+\fk_{\theta^c}}g(\pi_\fx^{\theta^c}\cdot)d\ft_{\tilde\theta}\\
        &=\int_{\fx_{\tilde\theta}}^{\fy_{\tilde\theta}} \prod_{i\in\tilde\theta} \frac{(y_i-t_i)^{n_i-k_i-1}}{(n_i-k_i-1)!} \square^\theta_{\fx,\pi_{\ft_{\tilde\theta}}^{\tilde\theta}\fy} \partial^{\fn_\theta-\fk_\theta+\fk_{\theta^c}}gd\ft_{\tilde\theta}
    \end{align*}
    does not depend on $\fz$ at all. Using the definition of $\norm{g}_{C^\fbeta}$, it is easy to see
    \[
        \abs{A_1(\fx,\fy,\fk,\theta)} \lesssim \abs{\fx-\fy}^{\fbeta-\fk}_{\text{c},\theta} \norm{g}_{C^\fbeta}\,.
    \]
    Putting this back into \eqref{eq:ExpandingGTaylor} and using the binomial formula on $(z_i-x_i)^{k_i}, i\in\theta^c$ leads to
    \begin{equation}\label{def_A_2}
        R^\theta T_\fx^{\theta^c,\fn} g(\fx,\pi^\theta_\fy \fz,\fn-\fk) = \sum_{\fl_{\theta^c}\le\fk_{\theta^c}\le\fn_{\theta^c}} (\fz-\fy)^{\fl_{\theta^c}}\underbrace{\prod_{i\in\theta^c} \frac{(y_i-x_i)^{k_i-l_i}}{l_i! (k_i-l_i)!}A_1(\fx,\fy,\fk,\theta)}_{A_2(\fx,\fy,\fk,\fl_{\theta^c},\theta)}\,,
    \end{equation}
    where $A_2$ clearly fulfils
    \[
    \abs{A_2(\fx,\fy,\fk,\fl_{\theta^c},\theta)} \lesssim \abs{\fx-\fy}^{\fbeta_\theta-\fk_\theta+\fk_{\theta^c}-\fl_{\theta^c}}_{\text{c}} \norm{g}_{C^\fbeta}\,.
    \]
    By summing out $\fk_{\theta^c}$, we get
    \begin{align}\begin{split}\label{def_A_3}
        R^\theta T_\fx^{\theta^c,\fn} g(\fx,\pi^\theta_\fy \fz,\fn-\fk) &= \sum_{\fl_{\theta^c}\le\fn_{\theta^c}} (\fz-\fy)^{\fl_{\theta^c}}\underbrace{\sum_{\fl_{\theta^c}\le\fk_{\theta^c}\le\fn_{\theta^c}}A_2(\fx,\fy,\fk,\fl_{\theta^c},\theta)}_{A_3(\fx,\fy,\fk_{\theta},\fl_{\theta^c},\theta)}\\
        \abs{A_3(\fx,\fy,\fk_\theta,\fl_{\theta^c},\theta)} &\lesssim \abs{\fx-\fy}^{\fbeta_\theta-\fk_\theta-\fl_{\theta^c}}_{\text{c}} \norm{g}_{C^\fbeta}
    \end{split}\end{align}
    Here, $\sum_{\fl_{\theta^c}\le\fk_{\theta^c}\le\fn_{\theta^c}}$ is seen as a sum over $\fk_{\theta^c}$ with fixed $\fl_{\theta^c}$. Putting \eqref{def_A_3} back into \eqref{eq:rectangIncrementTaylorExpansion} with $A(\fx,\fy,\fk,\theta) := A_3(\fx,\fy,\fk_\theta,\fk_{\theta^c},\theta)$ gives \eqref{ineq:Taylor1}.

    We now pass on to prove \eqref{ineq:Taylor2}. If $g$ has more than $n_i$ derivatives in some direction $i\in[d]$, one can work with the easier remainder formula 
    \begin{equation}
        R^ig(x_i,\fz,n_i) =\int_{x_i}^{z_i} \frac{(z_i-t_i)^{n_i}}{n_i!} \partial^{i,n_i+1} g(\pi^i_{t_i}\fz) dt_i\,,
    \end{equation}
    leading to
    \begin{equation}
        R^\theta g(\fx,\fz,\fn) = \int_{\fx_\theta}^{\fz_\theta} \prod_{i\in\theta}\frac{(z_i-t_i)^{n_i}}{n_i!} \partial^{\fn_\theta+\fone_\theta} g(\pi_{\ft_\theta}^\theta \fz) d\ft_\theta\,.
    \end{equation}
    The rest of the proof follows by using the substitution $\ft = \flambda*\tilde\ft + \fx$ and calculating (recall $\fx_\theta = \fz_\theta$)
    \begin{align*}
    R^\theta g(\fx,\flambda*\fy+\fz,\fn) &= \int_{\fx_\theta}^{\fy_\theta*\flambda_\theta +\fx_\theta}\prod_{i\in\theta}\frac{(\lambda_i y_i -x_i-t_i)^{n_i}}{n_i!} \partial^{\fn_\theta+\fone_\theta} g(\pi^\theta_{\ft_\theta} (\flambda*\fy+\fz)) d\ft_\theta\\
    &= \flambda^{\fn_\theta+\fone_\theta} \int_\foo^{\fy_\theta} \prod_{i\in\theta}\frac{(y_i-\tilde t_i)^{n_i}}{n_i!} \partial^{\fn_\theta+\fone_\theta} g(\flambda*\pi^\theta_{\tilde\ft}\fy+\fz)d\tilde\ft_\theta\\
    &= \flambda^{\fn_\theta+\fone_\theta} B(\fy,\fz,\flambda,\theta)\,.
\end{align*}
The definition of $B$ immediately gives $\norm{B(\cdot,\fz,\flambda,\theta)}_{C^r} \lesssim \norm{g}_{C^{r+\abs{\fn_\theta+\fone_\theta}}}$.
\end{proof}

\noindent With this, we can show the main result of this section. Let $g\in C^\fbeta$ and consider the germ
\[
F_\fx(\psi) = \zeta(T_\fx^\fn g\cdot\psi)\,,
\]
where we fix $\fn,\fgamma$ as in Definition \ref{def:Calpha}. $\zeta\in C^{\falpha,\fdelta}L_m$ gives us that the germ $(F_\fx)_{\fx\in[0,T]^d}$ is adapted to $(\cF_\fx)_{\fx\in[0,T]^d}$. As it turns out, it is also stochastically coherent, allowing the use of reconstruction to find the product between $g$ and $\zeta$. We fix $r\in\bN$ sufficiently large, such that $\zeta(\psi)$ is well defined for all $\psi\in C^r$.

\begin{proposition}\label{prop:YoungProduct}
    $(F_\fx)$ is stochastically $(\falpha, \falpha+\fbeta,\fdelta)$-coherent. It follows in particular that it can be reconstructed, as long as $2\le m<\infty$ and
    \begin{equation}\label{cond:TaylorRecon}
        \falpha+\fbeta > -\frac 12\fone\,,\qquad \alpha+\fbeta+\fdelta>\foo\,. 
    \end{equation}
    Furthermore, if $g$ is smooth, the reconstruction will be given by the canonical product between $\zeta$ and $g$
    \[
        \cR^{[d]}(F)(\psi) = \zeta(g\cdot\psi)
    \]
    for any test function $\psi\in C_c^{r}$.
\end{proposition}

\begin{remark}
    For smooth $g$, one can explicitly write $\cR^\theta(F)$ for each $\theta\subset[d]$, since the reconstruction ``undoes'' the Taylor-approximation. That is, for each $\theta\subset[d]$, $\fx\in[0,T]^d$ and $\psi\in C_c^{r}$ we have
    \begin{equation}\label{eq:ReconstructionTaylor}
        \cR^\theta_\fx(F)(\psi) = \zeta(T^{\theta^c,\fn}_\fx g\cdot\psi)\,.
    \end{equation}
\end{remark}


\noindent This result is the natural extension of \cite{zambotti2020}[Section 14] to the stochastic rectangular increment setting.

\begin{proof}
    Let $\eta\subset\theta\subset[d]$, $\fx,\fy\in[0,T]^d$ with $\fx_\theta\le\fy_\theta$, $\flambda\in(0,1]^d$ and $\psi\in C_c^r([-1,1]^d)$ such that $\supp(\psi_\fy^\flambda)\subset[0,T]^d$. Then \eqref{ineq:Taylor1} gives us
    \begin{align*}
        \norm{\bE_\fx^\eta\square^\theta_{\fx,\fy} F(\psi_\fy^\flambda)}_m &= \norm{\bE^\eta_\fx \zeta(\square^\theta_{\fx,\fy}T^{\fn}g\cdot\psi_\fy^\flambda )} \\
        &\le \sum_{\fk\le\fn} \abs{A(\fx,\fy,\fk,\theta)} \norm{\zeta((\cdot-\fy)^\fk\psi_\fy^\flambda)}_m\,.
    \end{align*}
    Recall 
    \[
    (\cdot-\fy)^\fk\psi_\fy^\flambda = \flambda^\fk\tilde\psi_\fy^\flambda\,,\quad\text{where}\quad \tilde\psi(\fz) = \fz^\fk\psi(\fz)\,.
    \]
    This together with $\zeta\in C^{\falpha,\fdelta}L_m$ and \eqref{ineq:Taylor1} then gives
    \begin{align*}
        \norm{\bE_\fx^\eta\square^\theta_{\fx,\fy} F(\psi_\fy^\flambda)}_m &\lesssim \sum_{\fk\le\fn}\compAbs{\fx-\fy}^{\fbeta_\theta-\fk} \flambda^{\fk+\falpha+\fdelta_\eta}\norm{g}_{C^\fbeta}\\
        &\lesssim \flambda^{\falpha}(\compAbs{\fx-\fy}+\flambda)^{\fbeta}_\theta (\compAbs{\fx-\fy}+\flambda)^{\fdelta}_\eta \norm{g}_{C^\fbeta}\,,
    \end{align*}
    showing stochastic $(\falpha,\falpha+\fbeta+\fdelta)$-coherence. Thus, the germ can be reconstructed as long as \eqref{cond:TaylorRecon} holds.

    It remains to show that for smooth $g$, the reconstruction is given by $\cR^{[d]} (F)(\psi) = \zeta(\psi\cdot g)$. We do this by showing the stronger condition \eqref{eq:ReconstructionTaylor}. It is clear that the family $(\cR^\theta_\fx(F))_{\theta\subset[d],\fx\in[0,T]^d}$,
    \[
        \cR^\theta_\fx(F)(\psi) = \zeta(T^{\theta^c,\fn}_\fx g\cdot\psi)\,,
    \]
    fulfils conditions 1.-3. from the reconstruction theorem \ref{theo:reconstruction}, so we only show 4. Let $\psi,\fz,\fx,\flambda$ and $\theta$ be as in Condition 4. We write
    \
    \begin{align*}
    \sum_{\hat\theta\subset\theta}(-1)^{\sharp{\hat\theta}} \mathcal{R}^{\hat\theta}_\fx(F)(\psi^{\flambda}_\fz) &= \sum_{\hat\theta\subset\theta}(-1)^{\sharp{\hat\theta}} \zeta(\psi^{\flambda}_\fz\cdot T^{\hat\theta^c,\fn}_\fx g)\\
    &=(-1)^{\sharp\theta} \zeta(\psi_\fz^\flambda\cdot\prod_{i\in\theta}(\Id-T_{x_i}^{i,n_i})T^{\theta^c,\fn}_\fx g) \\
    &= (-1)^{\sharp\theta} \zeta(\psi_\fz^\flambda\cdot R^\theta \tilde g(\fx,\cdot,\fn))\,,
    \end{align*}
    where $\tilde g := T_\fx^{\theta^c,\fn} g$. We use \eqref{ineq:Taylor2} to write
    \[
        \tilde\psi(\fy) := \psi(\fy)\cdot R^\theta\tilde g(\fx,\flambda*\fy+\fz,\fn) = \psi(\fy)\cdot\flambda^{\fn_\theta+\fone_\theta} B(\fy,\fz,\flambda,\theta)\,.
    \]
    The definition of $\tilde\psi$ implies
    \[
        \tilde\psi_\fz^\flambda = \psi_\fz^\flambda\cdot R^\theta \tilde g(\fx,\cdot,\fn)\,.
    \]
    Note that $\eta\subset\theta$ and $\fx_\theta = \fz_\theta$ implies $\bE^\eta_\fx = \bE^\eta_\fz$. Thus, we can use $\zeta\in C^{\falpha,\fdelta}L_m$ to conclude
    \begin{align*}
        \norm{E^\eta_\fx\sum_{\hat\theta\subset\theta}(-1)^{\sharp{\hat\theta}} \mathcal{R}^{\hat\theta}_\fx(F)(\psi^{\flambda}_\fz)}_m &= \norm{\bE_\fz^\eta \zeta(\tilde\psi_\fz^\lambda)}_m \\
        &\lesssim \flambda^{\falpha+\fdelta_\eta} \norm{\tilde\psi}_{C^r_c}\\
        &\lesssim \flambda^{\falpha+\delta_\eta} \flambda^{\fn_\theta+\fone_\theta} \norm{\psi\cdot B(\cdot,\fz,\flambda,\theta)}_{C^r_c} \\
        &\le \flambda^{\falpha+\fbeta_\theta+\fdelta_\eta} \norm{\psi\cdot B(\cdot,\fz,\flambda,\theta)}_{C^r_c}\,.
    \end{align*}
    Condition 4 then follows from $\norm{B(\cdot,\fz,\flambda,\theta)}_{C^r}$ being bounded by \\$\norm{\tilde g}_{C^{r+\abs{\fn_\theta+\fone_\theta}}} \lesssim \norm{g}_{C^{r+\abs{\fn+\fone_\theta}}}$.
\end{proof}

\bibliographystyle{alpha}
\bibliography{bibliography}

\end{document}